\documentclass{amsart}
\usepackage[utf8]{inputenc}
\usepackage{amsmath,amsfonts,amssymb,amscd,amsthm,amsbsy,upref, siunitx}
\usepackage[dvipsnames]{xcolor}
\usepackage{amsmath}
\usepackage{hyperref,cleveref,amssymb,mathtools}
\usepackage{bbm}
\usepackage{enumitem}
\definecolor{darkgreen}{rgb}{0.0, 0.27, 0.13}

\definecolor{ForestGreen}{RGB}{34,139,34}

\numberwithin{equation}{section}
\def\hangbox to #1 #2{\vskip3pt\hangindent #1\noindent \hbox to #1{#2}$\!\!$}

\usepackage{hyperref,cleveref,amssymb,mathtools}

\theoremstyle{plain}
\newtheorem{theorem}{Theorem}[section]
\newtheorem{proposition}[theorem]{Proposition}
\newtheorem{corollary}[theorem]{Corollary}
\newtheorem{lemma}[theorem]{Lemma}

\theoremstyle{definition}
\newtheorem{remark}[theorem]{Remark}

\newtheorem{definition}[theorem]{Definition}

\DeclareSymbolFont{bbold}{U}{bbold}{m}{n}
\DeclareSymbolFontAlphabet{\mathbbold}{bbold}

\usepackage{crossreftools}
\hypersetup{%
    bookmarksnumbered, bookmarksopen=true, bookmarksopenlevel=1,%
}

\def\sfrac#1#2{\kern.1em\raise.5ex\hbox{$#1$}
	\kern-.1em/\kern-.05em\lower.25ex\hbox{$#2$}}

\newcommand{\fw}{\text{\fw}}

\newcommand{\supp}{{\rm supp}}

\begin{document}

\title[Quasilinear Schr\"odinger]{Low regularity solutions for the general Quasilinear ultrahyperbolic Schr\"odinger equation}

\author{Ben Pineau}
\address{Department of Mathematics\\
University of California at Berkeley
} \email{bpineau@berkeley.edu}
\author{Mitchell A.\ Taylor}
\address{Department of Mathematics\\
ETH Z\"urich
} \email{mitchell.taylor@math.ethz.ch}

\begin{abstract}
We present a novel method for establishing large data local well-posedness in low regularity Sobolev spaces for general quasilinear Schr\"odinger equations with non-degenerate and nontrapping metrics. Our result represents a definitive improvement over the landmark results of Kenig, Ponce, Rolvung and Vega \cite{kenig2005variable, MR2263709, MR1660933, MR2096797}, as it weakens the regularity and decay assumptions to the same scale of spaces considered by Marzuola, Metcalfe and Tataru in  \cite{marzuola2021quasilinear}, but removes the uniform ellipticity assumption on the metric from their result. Our method has the additional benefit of being relatively simple but also very robust. In particular, it only relies on the use of pseudodifferential calculus for classical symbols. 
\end{abstract} 


 \maketitle

\tableofcontents

\section{Introduction}
In this article, we consider the large data local well-posedness problem for general quasilinear ultrahyperbolic Schr\"odinger equations of the form
\begin{equation}\label{fulleqn}
\begin{cases}
&i\partial_tu+g^{jk}(u,\overline{u},\nabla u,\nabla\overline{u})\partial_j\partial_k u=F(u,\overline{u},\nabla u,\nabla\overline{u}),\hspace{5mm} u:\mathbb{R}\times\mathbb{R}^d\to\mathbb{C}^m,
\\
&u(0,x)=u_0(x),
\end{cases}
\end{equation}
where $g$ and $F$ are assumed to be smooth functions of their arguments with $g$ real, symmetric and uniformly non-degenerate and $F$ vanishing at least quadratically at the origin.

\medskip

In a recent series of articles \cite{marzuola2012quasilinear, MR3263550, marzuola2021quasilinear},  Marzuola, Metcalfe and Tataru have studied the well-posedness of the system \eqref{fulleqn} in low regularity Sobolev spaces.  As a brief overview, the paper \cite{MR3263550}  considers the small data problem for cubic and higher nonlinearities in the Sobolev spaces $H^s(\mathbb{R}^d)$, $s>\frac{d+5}{2}$. The article \cite{marzuola2012quasilinear}, on the other hand, permits quadratic terms in the nonlinearity, but assumes that the data comes from the smaller space $l^1H^s(\mathbb{R}^d)$, $s>\frac{d}{2}+3$. Here, $l^1H^s$ is an appropriate  translation invariant Sobolev type space, imposing similar regularity requirements as $H^s$, but  slightly stronger decay. To see that some additional decay  is needed, it is instructive to look at the leading part of the linearized flow, which can be written schematically as
\begin{equation}\label{linearflowintro}
\begin{cases}
&i\partial_tv+\partial_jg^{jk}\partial_k v+b^j\partial_jv+\tilde{b}^j\partial_j\overline{v}=f,
\\
&v(0,x)=v_0(x).
\end{cases}
\end{equation}
Here, for the purposes of our heuristic discussion, we have written the principal operator in divergence form with $g^{jk}=g^{jk}(u,\overline{u})$ -- we will elaborate further on this reduction later on. As is well-known, a necessary condition for  $L^2$ well-posedness of a wide class of such linear systems is  integrability of the first order coefficient $\operatorname{Re}(b^j)$ along the bicharacteristic (or Hamilton) flow of the principal differential operator $\partial_jg^{jk}\partial_k$. This is usually referred to as the Mizohata (or Takeuchi–Mizohata) condition. See, for instance, \cite{MR0948533,jeong2023illposedness, marzuola2008wave, MR0170112, mizohata1981quelques, MR0860041, takeuchi1980cauchy} for several manifestations of this ill-posedness mechanism. For cubic and higher nonlinearities, the  integrability of $\operatorname{Re}(b^j)$ along the bicharacteristics is automatic for small $H^s$ data, but for quadratic nonlinearities it is not. That being said, there are several natural ways to recover the above integrability condition. One common approach is to work in weighted Sobolev spaces. However, the alternative  $l^1H^s$ spaces also achieve this goal, but have the additional advantage of being   translation invariant -- they are also far less restrictive in terms of regularity and decay, as we will see below.  
\medskip

In contrast to the case of small data, the third paper in the series by Marzuola, Metcalfe and Tataru \cite{marzuola2021quasilinear} considers the significantly more challenging large data problem. Here, the authors establish well-posedness in the same setting as their small data papers, but under two additional assumptions. The first assumption is that the initial metric $g(u_0)$ is nontrapping, meaning that all nontrivial bicharacteristics corresponding to the principal operator $\Delta_{g(u_0)}$ escape to spatial infinity at both ends. Such a condition is automatic in the small data regime (assuming sufficient regularity and asymptotic flatness of the metric) as in this setting the Hamilton trajectories are close to straight lines.  For large data well-posedness, a nontrapping assumption is completely natural, in light of the Mizohata condition. 
\medskip

On the other hand, the methods in \cite{marzuola2021quasilinear} also rely on the assumption of uniform ellipticity of the principal operator, i.e., the existence of a uniform constant $c>0$ such that
\begin{equation}\label{ellipticity}
c^{-1}|\xi|^2\leq g^{jk}(x)\xi_j\xi_k\leq c|\xi|^2.
\end{equation}
 This assumption is critically used in the above article to effectively diagonalize the linearized equation (\ref{linearflowintro}) and remove the complex conjugate first order term. Roughly speaking, this diagonalization proceeds by considering the new variable
\begin{equation*}
Sv:=v+\mathcal{R}\overline{v},
\end{equation*}
where $\mathcal{R}$ is a pseudodifferential operator of order $-1$ with symbol which is essentially of the form
\begin{equation*}
r(x,\xi)=\frac{i\tilde{b}^l\xi_l}{g^{jk}\xi_j\xi_k},
\end{equation*}
when $|\xi|\geq 1$. It is not difficult to see that, to leading order, $Sv$ formally satisfies an equation like (\ref{linearflowintro}), but without the complex conjugate first order term. This diagonalization procedure  is then used as a key ingredient in the proofs of the requisite local smoothing and $L_T^{\infty}L_x^2$ estimates for the linearized flow. A similar diagonalization is heavily relied upon in \cite{chihara1995local} and \cite{MR2096797}.
\medskip

The primary objective of the current article is to generalize the main result of \cite{marzuola2021quasilinear}  to the full class of ultrahyperbolic quasilinear Schr\"odinger flows, while keeping the regularity and function spaces identical. That is, we shall relax the uniform ellipticity assumption
 \begin{equation*}
c^{-1}|\xi|^2\leq g^{jk}\xi_j\xi_k\leq c|\xi|^2
\end{equation*}
 to the much weaker uniform non-degeneracy condition
 \begin{equation*}
c^{-1}|\xi|\leq |g^{jk}\xi_k|\leq c|\xi|.
 \end{equation*}
The lack of an ellipticity assumption on the metric in \eqref{fulleqn} causes significant difficulties. For example,  the diagonalization procedure outlined above completely breaks down in this case. This is what prompted the development of the new well-posedness scheme that we present in this article (we were also inspired by the scheme in \cite{jeong2023cauchy}). On the other hand, there are several physical sources of motivation for studying the general ultrahyperbolic problem. Some well-known examples arise naturally in the study of water waves \cite{davey1974three} and others arise in the theory of completely integrable models \cite{ishimori1984multi, zakharov1980degenerative}. More recently,  the Hall and electron magnetohydrodynamic equations without resistivity have been shown to behave at leading order like degenerate quasilinear Schr\"odinger systems of ultrahyperbolic type \cite{MR4456288} with principal operator of the form $B\cdot \nabla |\nabla|$ where $B$ is the magnetic field. This dispersive character of the equations was used to great effect in \cite{MR4456288,jeong2023cauchy},  leading to well-posedness in certain regimes and ill-posedness in others.
 \medskip
 
 Although \cite{marzuola2021quasilinear} requires ellipticity of the metric in order to achieve their low regularity results,  significant progress has been made towards removing the ellipticity assumptions from the well-posedness theory of \eqref{fulleqn} in the high regularity regime. This is  best illustrated by the pioneering series \cite{kenig2005variable, MR2263709, MR1660933, MR2096797} of Kenig, Ponce, Rolvung and Vega, which  culminates in a proof of large data well-posedness under the nontrapping assumption for systems of the form \eqref{fulleqn} in high regularity weighted Sobolev spaces of the form $H^s\cap L^2(\langle x\rangle^Ndx)$, where $s$ and $N$ are suitably large, dimension dependent parameters. In this fundamental series of papers,  \cite{MR2096797} studies the well-posedness problem assuming ellipticity of the principal operator $\partial_jg^{jk}\partial_k$, while  \cite{kenig2005variable, MR2263709, MR1660933} consider  symmetric, non-degenerate metrics, first in the constant coefficient case and then later for variable coefficients.  As should  be evident from these  articles,  the ellipticity assumption on the metric is not easy to remove, even in the high regularity regime. The main objective of the current paper is to give a much simpler proof of well-posedness for the general system \eqref{fulleqn} that is also robust enough to work in low regularity spaces. To the best of our knowledge, this is the first low regularity well-posedness result that applies to the full class of ultrahyperbolic quasilinear Schr\"odinger flows.
 \medskip
 
 The rough strategy used in \cite{kenig2005variable} to prove well-posedness of the ultrahyperbolic flow \eqref{fulleqn} in high regularity weighted spaces is to first establish an estimate for the local energy type norm 
\begin{equation*}
\|v\|_{LE}:=\|\langle x\rangle^{-\frac{N}{2}}\langle\nabla\rangle^{\frac{1}{2}}v\|_{L_T^2L_x^2},\hspace{5mm}N=N(d)\in\mathbb{N},
\end{equation*}
for the linearized equation (\ref{linearflowintro}) (assuming suitably strong asymptotic decay of the coefficients $b^j$, $\tilde{b}^j$ and $\nabla_x g^{jk}$) of the form
\begin{equation}\label{LEKENIG}
\|v\|_{LE}\lesssim \|v\|_{L_T^{\infty}L_x^2}+\|f\|_{LE^*+L_T^1L_x^2}.
\end{equation}
Here, $LE^*$ denotes the ``dual" local energy space. The  estimate \eqref{LEKENIG} shows that the local energy norm of $v$ remains under control, as long as  $v$ satisfies  an a priori $L_T^{\infty}L_x^2$ bound. The preliminary estimate \eqref{LEKENIG}  follows,  roughly speaking, from a suitable adaptation of Doi's construction in \cite{ichi1996remarks}  to the ultrahyperbolic problem. The more significant technical obstruction in \cite{kenig2005variable} is in establishing the a priori bound for the $L_T^{\infty}L_x^2$ norm. To understand  the difficulties, we first note that when the real part of the coefficient $b^j$ vanishes, it is a relatively straightforward exercise (in view of (\ref{LEKENIG})) to obtain the bound
\begin{equation*}
\|v\|_{L_T^{\infty}L_x^2}\lesssim \|v_0\|_{L_x^2}+\|f\|_{LE^*+L_T^1L_x^2}.
\end{equation*}
Indeed, this follows by a standard energy estimate, as one can integrate by parts to shift derivatives off of the first order terms and onto the coefficients $b^j$ and $\tilde{b}^j$. Therefore, in the general case, one is motivated 
to try to conjugate away the badly behaved first order term $\operatorname{Re}(b^j)\partial_jv$. In  \cite{kenig2005variable}, this  conjugation is accomplished by constructing a (formally) zeroth order operator $\mathcal{O}$ which achieves the approximate cancellation
\begin{equation}\label{approxcancellation}
[\mathcal{O},\partial_jg^{jk}\partial_k]+\mathcal{O}\operatorname{Re}(b^j)\partial_j\approx 0.
\end{equation}
The idea here is very loosely akin to the method of integrating factors from ODE. On a formal level, the symbol for the operator $\mathcal{O}$ achieving \eqref{approxcancellation} is given by
\begin{equation}\label{renormalizationnotideal}
O(x,\xi):=\exp\left(-\int_{-\infty}^{0}\operatorname{Re}(b(x^{t}))\cdot\xi^tdt\right),
\end{equation}
where $(x^t,\xi^t)$ denotes the bicharacteristic flow  
\begin{equation*}
(\dot{x}^t,\dot{\xi}^t)=(\nabla_{\xi}a(x^t,\xi^t),-\nabla_xa(x^t,\xi^t)),\hspace{5mm}(x^0,\xi^0)=(x,\xi),
\end{equation*}
corresponding to the principal symbol $a(x,\xi):=-g^{jk}(x)\xi_j\xi_k$.
Unfortunately, the symbol $O$ does not belong to the standard symbol class $S^0$. Rather, (assuming that $b$ has sufficient regularity and decay) it satisfies
\begin{equation}\label{CKSclass}
|\partial^{\alpha}_{\xi}\partial_x^{\beta}O(x,\xi)|\lesssim_{\alpha,\beta}\langle\xi\rangle^{-|\alpha|}\langle x\rangle^{|\alpha|}.
\end{equation}
In the case when the metric is positive-definite (i.e.~$\Delta_g$ is elliptic), the mapping properties of the pseudodifferential operators associated with this class of symbols were intensively studied in the paper \cite{craig1995microlocal} of Craig, Kappeler and Strauss. In the case of a merely non-degenerate metric, Kenig, Ponce, Rolvung and Vega in \cite{kenig2005variable} execute a systematic study of this symbol class as well as a very careful analysis of the bicharacteristic flow for $-g^{jk}(x)\xi_j\xi_k$ to establish suitable mapping properties for $\mathcal{O}$. In contrast, in the current article, to obtain the $L_T^{\infty}L_x^2$ estimate for (\ref{linearflowintro}) we will instead use a spatially truncated version of the above renormalization operator which achieves a suitable cancellation of the form (\ref{approxcancellation}), at least within a large compact set. The key advantage of this truncation is that the corresponding renormalization operator will be a classical pseudodifferential operator of order $0$, which will dramatically simplify the analysis (perhaps at the cost of estimating some extra error terms). Moreover, it will allow us to considerably lower the regularity and decay assumptions on the coefficients in (\ref{linearflowintro}) compared to \cite{kenig2005variable} when estimating the $L_T^{\infty}L_x^2$ norm of $v$. Of course, this idea comes with some technical caveats of its own, which will be discussed later.
\medskip

We remark that the idea of using the above spatial truncation to close the energy estimate for (\ref{linearflowintro}) is inspired by the article \cite{jeong2023cauchy} of Jeong and Oh, where they consider the well-posedness problem for the electron MHD equations near non-zero, constant magnetic fields, and perform an analogous truncation in their setting. As we shall see below, such a construction turns out to be tied heavily to the direction of propagation of the bicharacteristics of the principal part of the corresponding linear flow.  For the electron MHD equations, the bicharacteristics have a distinguished direction of propagation. However, the bicharacteristics for the Schr\"odinger equations that we consider in this article do not exhibit this feature. Therefore, one key novelty of the present paper is in dealing with the multi-directionality present in Schr\"odinger flows. Another important novelty is our ability to extend the truncation idea in order to give a new and very simple proof of the natural local smoothing type estimate for (\ref{linearflowintro}) in the local energy norms compatible with the translation invariant function spaces used in this paper. The method that we present is very robust and requires only  mild decay of the coefficients (e.g.~uniform integrability along the Hamilton flow). A more detailed outline of the argument will be given in \Cref{OOTP}.

\subsection{Statements of the results}
We now state our main results more precisely. As in \cite{marzuola2021quasilinear}, our primary focus will be on the case of quadratic nonlinear interactions.
\medskip

Let $d,m\geq 1$ and consider a system of equations of the form \eqref{fulleqn} where
\begin{equation}\label{mapping prop of g,f}
\begin{split}
&g:\mathbb{C}^m\times \mathbb{C}^m\times (\mathbb{C}^m)^d\times (\mathbb{C}^m)^d\to \mathbb{R}^{d\times d} \ \ \text{and} \ \ F:\mathbb{C}^m\times \mathbb{C}^m\times (\mathbb{C}^m)^d\times (\mathbb{C}^m)^d\to \mathbb{C}^m
\end{split}
\end{equation}
are smooth functions. We assume that $F$  vanishes at least quadratically at the origin, so that
\begin{equation}\label{QIP}
\begin{split}
& |F(y,z)|\approx \mathcal{O}\left(|y|^2+|z|^2\right) \ \text{near} \ (y,z)=(0,0).
\end{split}
\end{equation}
 In \cite{marzuola2021quasilinear}, the authors assume that the metric $g$ is uniformly elliptic and coincides with the identity matrix at the origin. That is, they assume that $g(0)=I_{d\times d}$ and that there is a  fixed constant $c>0$ so that
 \begin{equation*}
   c^{-1}|\xi|^2\leq g^{jk}(x)\xi_j\xi_k\leq c|\xi|^2,  \ \forall \xi\in \mathbb{R}^d, \ x\in\mathbb{C}^m\times \mathbb{C}^m\times (\mathbb{C}^m)^d\times (\mathbb{C}^m)^d.
\end{equation*}
In this article, we only assume that $g$ is symmetric and (uniformly) non-degenerate, in the sense that
\begin{equation}\label{nondeg metric}
    c^{-1}|\xi|\leq |g(x)\xi|\leq c|\xi|, \ \forall \xi\in \mathbb{R}^d, \ x\in\mathbb{C}^m\times \mathbb{C}^m\times (\mathbb{C}^m)^d\times (\mathbb{C}^m)^d,
\end{equation}
for some fixed constant $c>0$.
\medskip

As in \cite{marzuola2012quasilinear, MR3263550, marzuola2021quasilinear}, we also consider a second class of quasilinear Schr\"odinger equations of the form
\begin{equation}\label{div-form}
\begin{cases}
&i\partial_tu+\partial_j g^{jk}(u,\overline{u})\partial_k u=F(u,\overline{u},\nabla u,\nabla\overline{u}),\hspace{5mm} u:\mathbb{R}\times\mathbb{R}^d\to\mathbb{C}^m,
\\
&u(0,x)=u_0(x),
\end{cases}
\end{equation}
where  $F$ is as in \eqref{QIP}, but where the metric $g$ depends on $u$ but not on $\nabla u$. Such an equation arises by formally differentiating the  system \eqref{fulleqn}. Indeed, if $u$ solves \eqref{fulleqn} then $(u,\nabla u)$ solves an equation of the form \eqref{div-form} with a nonlinearity $F$ which depends at most quadratically on $\nabla u$. 
\medskip
\begin{remark}
    Note that the second order operator in \eqref{div-form} has a divergence structure, which can be achieved by commuting the first derivative with $g$ and viewing the commutator as an additional term on the right-hand side. In contrast, the second order operator in \eqref{fulleqn} cannot be written in divergence form without possibly changing the type of the equations.
\end{remark}
To state our main well-posedness theorem, we must recall the function spaces used in \cite{marzuola2012quasilinear, MR3263550, marzuola2021quasilinear}. For now, we limit ourselves to an expository summary, giving more precise definitions in \Cref{Func space}.
\medskip

Consider a standard spatial Littlewood-Paley decomposition
\begin{equation*}
    1=\sum_{j\in \mathbb{N}_0}S_j,
\end{equation*}
where $S_j$, $j\geq 1$, selects frequencies of size $\approx 2^j$ and $S_0$ selects all frequencies of size $\lesssim 1$. Corresponding to each dyadic frequency scale $2^j\geq 1$, we consider an associated partition $\mathcal{Q}_j$ of $\mathbb{R}^d$ into cubes of side length $2^j$ and an associated smooth partition of unity
\begin{equation*}
    1=\sum_{Q\in \mathcal{Q}_j}\chi_Q.
\end{equation*}
We define the $l_j^1L^2$ norm by 
\begin{equation}\label{ell_1sum intro}
    \|u\|_{l_j^1L^2}=\sum_{Q\in \mathcal{Q}_j}\|\chi_Qu\|_{L^2},
\end{equation}
and the space $l^1 H^s$ via the norm
\begin{equation}\label{ell1Hs intro}
   \|u\|^2_{l^1 H^s}=\sum_{j\geq 0}2^{2sj}\|S_ju\|_{l_j^1L^2}^2. 
\end{equation}
Note that if one replaces the $\ell^1$ sum by an $\ell^2$ sum in \eqref{ell_1sum intro} and defines $l^2 H^s$ analogously to \eqref{ell1Hs intro}, then $H^s=l^2H^s$ with equivalent norms. The extra summability in the definition of the $l^1H^s$ norm yields the decay necessary to circumvent Mizohata's ill-posedness mechanism. However, unlike the high regularity weighted Sobolev spaces used in previous works, the function spaces $l^1H^s$ admit translation invariant equivalent norms and contain functions exhibiting weaker regularity and decay. 
\medskip

As mentioned above, in the large data problem, one  has to contend with trapping. This is an obvious obstruction to well-posedness, so we will  need to impose a nontrapping assumption on the initial metric $g(u_0)$ to prevent this. Then, as part of our well-posedness theorem, we will    show  that the nontrapping assumption  propagates on a  time interval whose length depends on the data size and the profile of the initial metric. Our  definition of nontrapping is the same as \cite{marzuola2021quasilinear}.
\begin{definition}\label{def of nontrap}
    We say that the metric $g(u_0)$ is \emph{nontrapping} if all nontrivial bicharacteristics for $\Delta_{g(u_0)}$ escape to spatial infinity at both ends.
\end{definition}
The above qualitative definition of nontrapping suffices in order to state our main results. However, as we shall see, the proofs require us to introduce a parameter  $L$ which gives a quantitative description of nontrapping. The precise way in which we define  $L$ is slightly different than \cite{marzuola2021quasilinear}, so as to better handle  the case when  $\Delta_g$ is not elliptic. 
\medskip

With the above discussion in mind, we may  state our main well-posedness theorem as follows. 
\begin{theorem}\label{quadraticmain}
Let $s>\frac{d}{2}+3$ and suppose that the initial data $u_0\in l^1H^s$ makes $g(u_0)$ into  a real, symmetric, uniformly non-degenerate, nontrapping metric. Then \eqref{fulleqn} with the quadratic nonlinearity \eqref{QIP} is locally well-posed in $l^1H^s$. The same result holds if $s>\frac{d}{2}+2$ for the equation \eqref{div-form}.
\end{theorem}
\begin{remark}
  We will prove the latter result in \Cref{quadraticmain} as it will imply the former by differentiating the equation.   
\end{remark}
\begin{remark}
As in \cite{marzuola2021quasilinear}, the regularity and decay assumptions in the above results can be weakened if the metric and nonlinearity satisfy the stronger vanishing conditions
\begin{equation*}
g(y,z)=g(0)+\mathcal{O}\left(|y|^2+|z|^2\right),\hspace{5mm}|F(y,z)|\approx \mathcal{O}\left(|y|^3+|z|^3\right) \ \text{near} \ (y,z)=(0,0).  
\end{equation*}
Namely, it can be shown that \eqref{fulleqn} is well-posed in the same sense as \Cref{quadraticmain} when $u_0\in H^s$ and $s>\frac{d+5}{2}$. An analogous result holds if $s>\frac{d+3}{2}$ for the equation \eqref{div-form}. To prove this, one makes modifications to the quadratic case which are virtually identical  to those made in  \cite{marzuola2021quasilinear}. In order to simplify our exposition, we omit the details for these relatively straightforward modifications and instead focus on the general case of quadratic nonlinearities.
\end{remark}
\begin{remark}
    In the above results, well-posedness is to be interpreted in the standard quasilinear fashion. More precisely, in the setting of \Cref{quadraticmain} it includes the following key features.
    \begin{itemize}
        \item (Regular solutions). For large $\sigma$ and nontrapping initial data $u_0\in l^1H^\sigma$ there is a unique solution $u\in C([0,T];l^1H^\sigma)$ which persists and remains nontrapping on some nontrivial maximal time interval $I=[0,T_*).$
        \item (Rough solutions). For $s>\frac{d}{2}+3$ and nontrapping data $u_0\in l^1H^s$ there is a unique solution  $u\in C([0,T];l^1H^s)\cap l^1 X^s([0, T])$  which persists and remains nontrapping on some nontrivial maximal time interval $I=[0, T_*).$ Here, the auxiliary space $l^1X^s$ is a natural analogue of the local energy space $LE$ described earlier. A precise definition of this space will be given in \Cref{Func space}. 
        \item (Continuous dependence). The maximal time $T_*(u_0)$ is a lower semicontinuous function of $u_0$ with respect to the $l^1H^s$ topology and for each $T<T_{*}(u_0)$ the data-to-solution map $v_0\mapsto v$ is continuous near $u_0$ from $l^1H^s$ into $C([0,T];l^1H^s)\cap l^1 X^s([0,T]).$
    \end{itemize}
\end{remark}
\begin{remark}
    As in \cite[Remark 1.3.2]{marzuola2021quasilinear}, the maximal existence time $T_{*}(u_0)$  a priori depends on the full profile of the initial data $u_0$ rather than just its size in $l^1H^s$, due to the nontrapping condition on the metric.
\end{remark}
\begin{remark}
   As in \cite{marzuola2021quasilinear}, the arguments we use here are purely dispersive. This is in contrast to the viscosity methods used in earlier works, which are less tailored to the structure of the equations, and hence less suitable for low regularity analysis.
\end{remark}
\subsection{Organization of the paper} The paper is organized as follows. In \Cref{Prelims} we recall the precise functional setting used in \cite{marzuola2021quasilinear} as well as the standard Fourier-analytic, nonlinear, and pseudodifferential machinery that will be used throughout the article. In certain cases, we adapt this machinery in order to obtain refined estimates in the space-time function spaces where we aim to construct solutions to \eqref{fulleqn} and \eqref{div-form}. In \Cref{OOTP}, we provide a detailed outline of the proof. Then, in \Cref{THE BICHARAC FLOW}, we analyze the bicharacteristic flow. The key objectives of this section are to quantify nontrapping, show that nontrapping is stable under small perturbations, and establish suitable asymptotic bounds for the bicharacteristics.  In  \Cref{THE LINEAR ULTRA FLOW},  we state our main well-posedness theorem for the linearized flow and reduce the main linear estimate to establishing a simplified bound for the corresponding inhomogeneous linear paradifferential flow in the  $l^1X^s$ spaces where we intend to construct solutions. Then, in \Cref{Sec lin flow}, we aim to establish a suitable estimate for the $L_T^{\infty}L_x^2$ component of the $l^1X^s$ norm by constructing a truncated version of the renormalization operator $\mathcal{O}$ in (\ref{renormalizationnotideal}). Such an estimate will close on a short enough time interval, up to controlling a small factor of the local energy component of the $l^1X^s$ norm. In \Cref{LED sec}, we control this remaining component of the $l^1X^s$ norm for the linear paradifferential flow. Then, in  \Cref{proofoflinest}, we deduce the full $l^1X^s$ estimate for the paradifferential and linearized equations by combining the local energy estimate with the $L_T^{\infty}L_x^2$ estimate from \Cref{Sec lin flow}. Finally, in \Cref{finalsection} we use the linearized estimates from the previous sections along with a suitable paradifferential reduction of the full nonlinear equation to establish \Cref{quadraticmain}.
\subsection{Acknowledgements} We thank Sung-Jin Oh for several enlightening discussions. We also thank both Sung-Jin Oh and In-Jee Jeong for kindly sharing with us the method of spatially truncating the renormalization operator in the energy bound for the linearized equation, which was an idea that they first implemented in \cite{jeong2023cauchy} for the electron MHD equations. An important component of this argument, which we also use in our setting, is the high-frequency Calderon-Vaillancourt bound in \Cref{CVvariant}. During the writing of this paper,  the authors were partially supported by the NSF grant DMS-2054975 as well as by the Simons Investigator grant of Daniel Tataru.  
\section{Preliminaries}\label{Prelims}
In this section, we recall some basic Fourier-analytic tools as well as the definitions and elementary properties of the function spaces that will be used in our analysis. We also recall some standard facts about pseudodifferential operators and establish some new estimates for these operators in our function spaces.
\subsection{Littlewood-Paley decomposition} We begin by recalling the standard Littlewood-Paley decomposition. We let $\varphi:\mathbb{R}^d\to [0,1]$ be a smooth radial function supported in the ball of radius $2$, $B_{2}=B_2(0)$, which satisfies $\varphi=1$ on $B_1$. We define Fourier multipliers $S_0$ and $S_k$ by
\begin{equation*}
\begin{split}
&\widehat{S}_k:=\varphi(2^{-k}\xi)-\varphi(2^{-k+1}\xi),\hspace{5mm}k\in\mathbb{N},
\\
&\widehat{S}_0:=\varphi(\xi).
\end{split}
\end{equation*}
We then define for each $k\in\mathbb{N}$,
\begin{equation*}
S_{<k}:=\sum_{0\leq j<k}S_j,\hspace{5mm}S_{\geq k}:=\sum_{k\leq j<\infty}S_j.
\end{equation*}
With the above notation, we have the standard (inhomogeneous) Littlewood-Paley decomposition
\begin{equation*}
1=\sum_{k\geq 0}S_k.
\end{equation*}
In the sequel, we will often phrase our  bilinear and nonlinear estimates in the language of paradifferential calculus. For a suitable pair of complex-valued functions $f$ and $g$, we will write $T_gf$ to mean
\begin{equation}\label{paraproductdef}
T_gf:=\sum_{k\geq 0}S_{<k-4}gS_kf.
\end{equation}
 In other words, $T_gf$ selects the portion of the product $fg$ where $f$ is at high frequency compared to $g$. With this notation, we have the so-called Bony decomposition or Littlewood-Paley trichotomy,
 \begin{equation*}
fg=T_fg+T_gf+\Pi (f,g).     
 \end{equation*}
We refer the reader to \cite{bony1981calcul} and \cite{metivier2008differential} for some basic properties of these operators. To compactify the above notation, we will sometimes write $f_{<k}$ as shorthand for $S_{<k}f$ and $f_{\geq k}$ as shorthand for $S_{\geq k}f$. 
\subsection{Function spaces}\label{Func space} Next, we recall the definitions and basic properties of the function spaces that will be used in our analysis.  Much of the material here is recalled from \cite{marzuola2012quasilinear} and the large data paper \cite{marzuola2021quasilinear}. For each frequency scale $2^k$, we consider a partition of $\mathbb{R}^d$ into a set $Q_k$ of disjoint cubes of side length $2^k$ along with a smooth partition of unity in physical space,
\begin{equation*}
1=\sum_{Q\in Q_k}\chi_{Q}.
\end{equation*}
For a translation-invariant Sobolev type space $U$, we define the spaces $l_k^pU$ by
\begin{equation*}
\begin{split}
\|u\|_{l_k^pU}:=\left(\sum_{Q\in Q_k}\|\chi_{Q}u\|_{U}^p\right)^{\frac{1}{p}},\hspace{2mm} 1\leq p<\infty,\hspace{5mm}\|u\|_{l_k^{\infty}U}:=\sup_{Q\in Q_k}\|\chi_{Q}u\|_{U}.
\end{split}
\end{equation*}
As noted in \cite{marzuola2012quasilinear}, these spaces have a translation invariant equivalent norm, obtained by replacing the sum over cubes with an integral. Moreover, up to norm equivalence, the smooth partition by compactly supported cutoffs  can be replaced by a partition consisting of  cutoffs which are all localized to  frequency zero. 
\medskip

We next recall the definition of the local energy type space $X$, which is defined for each $T>0$ by
\begin{equation*}\label{Xdef}
\|u\|_{X}:=\sup_{l\in\mathbb{N}_0}\sup_{Q\in Q_l} 2^{-\frac{l}{2}} \| u\|_{L_T^2L_x^2([0,T]\times Q)}.
\end{equation*} 
Associated to $X$ is the ``dual" local energy type space $Y\subset L^2([0,T]\times \mathbb{R}^d)$ which satisfies the relation $X=Y^*$. See \cite{marzuola2012quasilinear} for more details on the properties and construction of this space. For each non-negative integer $k$, we define
\begin{equation*}
X_k:=2^{-\frac{k}{2}}X\cap L_T^{\infty}L_x^2,\hspace{5mm}\|u\|_{X_k}:=2^{\frac{k}{2}}\|u\|_{X}+\|u\|_{L_T^{\infty}L_x^2}
\end{equation*}
and
\begin{equation*}
Y_k:=2^{\frac{k}{2}}Y+L_T^1L_x^2,\hspace{5mm}\|u\|_{Y_k}:=\inf\{2^{-\frac{k}{2}}\|u_1\|_{Y}+\|u_2\|_{L_T^1L_x^2}: u=u_1+u_2\}.
\end{equation*}
Loosely speaking, we will use $X_k$ to measure solutions to the Schr\"odinger equation localized at frequency $2^k$ whereas $Y_k$ will be used to measure inhomogeneous source terms localized at this frequency. Next, we define for each $s\in\mathbb{R}$,
\begin{equation*}
\|u\|_{l^pX^s}:=\left(\sum_{k\geq 0}2^{2ks}\|S_ku\|_{l^p_kX_k}^2\right)^{\frac{1}{2}},\hspace{5mm}\|u\|_{l^pY^s}:=\left(\sum_{k\geq 0}2^{2ks}\|S_ku\|_{l^p_kY_k}^2\right)^{\frac{1}{2}},
\end{equation*}
for $1\leq p<\infty$ (with the natural modification for $p=\infty$). We will also work with the corresponding spaces without the $\ell^p$ summability,
\begin{equation*}
\|u\|_{X^s}:=\left(\sum_{k\geq 0}2^{2ks}\|S_ku\|_{X_k}^2\right)^{\frac{1}{2}},\hspace{5mm}\|u\|_{Y^s}:=\left(\sum_{k\geq 0}2^{2ks}\|S_ku\|_{Y_k}^2\right)^{\frac{1}{2}}.
\end{equation*}
As already mentioned, throughout the article we will frequently make use of the standard tools of paradifferential calculus to estimate various multilinear and nonlinear expressions. A very nice bookkeeping device for efficiently tracking the frequency distribution of such terms is the language of frequency envelopes introduced by Tao in \cite{tao2001global}. To define
these, suppose that we are given a translation-invariant Sobolev type space $U$ with the orthogonality relation,
\begin{equation*}
\|u\|_{U}\approx \left(\sum_{k\geq 0}\|S_ku\|_{U}^2\right)^{\frac{1}{2}}.
\end{equation*}
An admissible frequency envelope for $u\in U$ is a positive sequence $(c_k)\subset\mathbb{N}_0$ such that for each $k\in\mathbb{N}_0$, we have
\begin{enumerate}
\item (Boundedness and size).
\begin{equation*}
\|S_ku\|_{U}\lesssim c_k\|u\|_{U},\hspace{5mm}\|c_k\|_{l^2_k}\approx 1.
\end{equation*}
\item (Left-slowly varying).
\begin{equation*}
c_j\geq 2^{\delta (j-k)}c_k,\hspace{5mm}j<k,
\end{equation*}
for some fixed parameter $\delta>0$.
\item (Right-uniformly varying).
\begin{equation*}
c_j\geq 2^{\sigma(k-j)}c_k,\hspace{5mm}j>k    ,
\end{equation*}
for some  fixed parameter $\sigma>0$.
\end{enumerate}
For nonzero $u$, such a frequency envelope always exists. For instance, we may define
\begin{equation*}
c_j=\|u\|_{U}^{-1}\left(\max_{k\geq j}2^{-\delta |j-k|}\|S_ku\|_{U}+\max_{k\leq j}2^{-\sigma |j-k|}\|S_ku\|_{U}\right).
\end{equation*}
In this article, the primary purpose of the above frequency envelopes will be to facilitate the proof of the continuity of the data-to-solution map for the quasilinear Schr\"odinger systems we consider.
\subsection{Pseudodifferential calculus}\label{psi do} Our  objective in this subsection is to recall some basic properties of pseudodifferential operators and then establish some refined estimates for these operators in the local energy and ``dual" local energy spaces defined above.
\medskip

For $m\in\mathbb{R}$, we recall that  the standard symbol class $S^m:=S^m_{1,0}$ is defined by
\begin{equation*}
S^m:=\{a\in C^{\infty}(\mathbb{R}^{2d}): |a|_{S^m}^{(j)}<\infty\, ,\hspace{2mm}j\in\mathbb{N}_0\},
\end{equation*}
where the corresponding seminorms $|a|_{S^m}^{(j)}$ are given by 
\begin{equation*}
|a|_{S^m}^{(j)}:=\sup\{\|\langle\xi\rangle^{|\alpha|-m}\partial_x^{\beta}\partial^{\alpha}_{\xi} a(x,\xi)\|_{L^{\infty}(\mathbb{R}^{2d})} :\hspace{2mm}|\alpha+\beta|\leq j\}.
\end{equation*}
To each symbol $a\in S^m$ we can associate the pseudodifferential operator $Op(a)\in OPS^m$, defined for $f\in \mathcal{S}(\mathbb{R}^d)$ by the quantization
\begin{equation*}
Op(a)f(x)=\frac{1}{(2\pi)^d}\int_{\mathbb{R}^d}a(x,\xi)e^{ix\cdot\xi}\widehat{f}(\xi)d\xi.
\end{equation*}
We now list some basic  properties of pseudodifferential operators; proofs can be found in the standard reference \cite{taylor1991pseudodifferential}.  We begin with  an elementary result on Sobolev boundedness. 
\begin{proposition}[Sobolev boundedness]\label{Sobolevbound} Let $s,m\in\mathbb{R}$ and let $a\in S^m$. Then $Op(a)$ extends to a bounded linear operator from $H^{s+m}$ to $H^s$ and there exists $j$ depending only on $s,$ $m$ and the dimension such that
\begin{equation*}
\|Op(a)\|_{H^{s+m}\to H^s}\lesssim |a|_{S^m}^{(j)}.
\end{equation*}
\end{proposition}
We next recall the sharp G\r{a}rding inequality for symbols $a\in S^1$. 
\begin{proposition}[Sharp G\r{a}rding inequality] \label{garding}Let $a\in S^1$ and let $R>0$ be such that $\operatorname{Re}(a)\geq 0$ for $|\xi|\geq R$. Then $Op(a)$ is semi-positive. That is, there exists $j$ depending on $d$ such that for $f\in \mathcal{S}(\mathbb{R}^d)$, we have
\begin{equation*}
\operatorname{Re}\langle Op(a)f,f\rangle\gtrsim_{R} -|a|_{S^1}^{(j)}\|f\|_{L^2}^2,
\end{equation*}
where $\langle\cdot,\cdot\rangle$ denotes the usual $L^2$ inner product.
\end{proposition}
\begin{proof}
    See, e.g.,~\cite{MR0233064}.
\end{proof}
\begin{remark}\label{matrixgard}
As shown in \cite{MR0233064,MR0206534}, a variant of \Cref{garding} also holds for $N\times N$ matrix-valued symbols. More specifically, if $a\in S^1$ is an $N\times N$ symbol satisfying $\operatorname{Re}(a)\geq 0$ then the associated pseudodifferential operator $Op(a)$ is semi-positive in the sense that $\operatorname{Re}\langle Op(a)f,f\rangle\geq -c\|f\|_{L^2}^2$ for all $f$ in the Schwartz class.
\end{remark}
Next, we recall a (weak) version of the Calderon-Vaillancourt theorem \cite{calderon1972class}. 
\begin{proposition}[Calderon-Vaillancourt theorem]\label{classicalCV} Let $a\in S^0$. There exists an integer $j>0$ depending on the dimension such that 
\begin{equation*}
\|Op(a)\|_{L^2\to L^2}\lesssim\sup_{|\alpha+\beta|\leq j}\|\partial^{\alpha}_{\xi}\partial^{\beta}_xa\|_{L^{\infty}(\mathbb{R}^{2d})},
\end{equation*}
where the implicit constant is universal.
\end{proposition}
Finally, we recall some elementary symbolic calculus which will allow us to perform basic manipulations with pseudodifferential operators.
\begin{proposition}[Algebraic properties of pseudodifferential operators]\label{symbcalc} Let $m_1,m_2\in\mathbb{R}$ and let $a_1\in S^{m_1}$ and $a_2\in S^{m_2}$. The following properties hold.
\begin{enumerate}
\item (Composition property). There is $a\in S^{m_1+m_2-1}$  such that
\begin{equation*}\label{comp}
Op(a_1)Op(a_2)=Op(a_1a_2)+Op(a)
\end{equation*}
and for every $j\in\mathbb{N}_0$, $|a|_{S^{m_1+m_2-1}}^{(j)}$ is controlled by $|a_1|_{S^{m_1}}^{(k)}|a_2|_{S^{m_2}}^{(k)}$ for some $k$ depending on $j$ and $d$.
\\

\item (Adjoint). There is $a\in S^{m_1-1}$ such that
\begin{equation*}\label{adjoint}
Op(a_1)^*=Op(\overline{a}_1)+Op(a)
\end{equation*}
and for every $j\in\mathbb{N}_0$, $|a|_{S^{m_1-1}}^{(j)}$ is controlled by $|a_1|_{S^{m_1}}^{(k)}$ for some $k$ depending on $j$ and $d$.
\\

\item (Commutator). There is $a\in S^{m_1+m_2-2}$ such that
\begin{equation*}\label{commutatorpseudo}
Op(a_1)Op(a_2)-Op(a_2)Op(a_1)=Op(-i\{a_1,a_2\})+Op(a)
\end{equation*}
where $\{\cdot,\cdot\}$ denotes the Poisson bracket, which is defined by
\begin{equation*}
\{a_1,a_2\}=\nabla_{\xi}a_1\cdot\nabla_{x}a_2-\nabla_{\xi}a_2\cdot\nabla_{x}a_1.
\end{equation*}
Moreover, for every $j\in\mathbb{N}_0$, $|a|_{S^{m_1+m_2-2}}^{(j)}$ is controlled by $|a_1|_{S^{m_1}}^{(k)}|a_2|_{S^{m_2}}^{(k)}$ for some $k$ depending on $j$ and $d$.
\end{enumerate}
\end{proposition}
\begin{proof}
    See, e.g.,~\cite[Theorem 2.1.2]{kenig2005variable} for a precise statement and \cite{PsiDo,taylor1991pseudodifferential} for proofs.
\end{proof}
In our construction, we will need the following refinement of the Calderon-Vaillancourt theorem for symbols $a\in S^0$, which ensures that the $L^2\to L^2$ operator bound for $Op(a)$  depends only on the $L^{\infty}$ norm of $a$ when applied to functions localized at sufficiently high frequency. This refinement will be important later when we  attempt to spatially localize the renormalization operator mentioned in the introduction.  We remark that \Cref{CVvariant} is also used in the  paper \cite{jeong2023cauchy} to achieve a similar purpose. We include the simple proof for completeness. 
\begin{proposition}[Calderon-Vaillancourt theorem at high frequency]\label{CVvariant}
Let $a\in S^0$. There is $k_0$ depending on $a$ such that for $k\geq k_0$, $Op(a)$ satisfies the  $L^2\to L^2$ bound,
\begin{equation*}
\|Op(a)S_{\geq k}\|_{L^2\to L^2}\lesssim \|a\|_{L^{\infty}}.
\end{equation*}
That is, the $L^2\to L^2$ bound for $Op(a)$ depends only on the $L^{\infty}$ norm of the symbol $a$ when applied to functions at sufficiently high frequency.
\end{proposition}
\begin{proof}
The proof is a simple scaling argument. The symbol for $S_{> k}$ is of the form $\psi_{k}(\xi):=1-\varphi(2^{-k}\xi)$, where $\varphi$ is a smooth bump function equal to one on the unit ball and supported in $B_2(0)$. Define the symbol $a_k:=a\psi_k$. Let $\lambda>0$ be some constant to be chosen, and define $a_{k,\lambda}(x,\xi):=a_{k}(\lambda^{-1}x,\lambda\xi)$, $v_{\lambda}(x):=v(\lambda x)$. We clearly have
\begin{equation*}
Op(a_k)v=(2\pi)^{-d}\int_{\mathbb{R}^d}a_k(x,\lambda\xi)e^{i\lambda x\cdot\xi}\widehat{v_{\lambda^{-1}}}(\xi)d\xi.
\end{equation*}
Hence,
\begin{equation*}
\|Op(a_k)v\|_{L^2}=\lambda^{-\frac{d}{2}}\|Op(a_{k,\lambda})v_{\lambda^{-1}}\|_{L^2}.
\end{equation*}
By \Cref{classicalCV}, we have 
\begin{equation*}
\begin{split}
\lambda^{-\frac{d}{2}}\|Op(a_{k,\lambda})v_{\lambda^{-1}}\|_{L^2}&\lesssim \lambda^{-\frac{d}{2}}\sup_{|\alpha|,|\beta|\leq j(d)}\|\partial_x^{\beta}\partial_{\xi}^{\alpha}a_{k,\lambda}\|_{L^{\infty}}\|v_{\lambda^{-1}}\|_{L^2}
\\
&=\sup_{|\alpha|,|\beta|\leq j(d)}\|\partial_x^{\beta}\partial_{\xi}^{\alpha}a_{k,\lambda}\|_{L^{\infty}}\|v\|_{L^2},
\end{split}
\end{equation*}
where $j(d)$ depends only on the dimension. To conclude, we therefore only need to show that for a suitable choice of $\lambda$, we have
\begin{equation*}
\sup_{|\alpha|,|\beta|\leq j(d)}\|\partial_x^{\beta}\partial_{\xi}^{\alpha}a_{k,\lambda}\|_{L^{\infty}}\lesssim \|a\|_{L^{\infty}}.
\end{equation*}
The case $|\alpha|=|\beta|=0$ is trivial. Otherwise, taking $\lambda=2^{\frac{k}{2}}$ and using that $a\in S^0$, we find
\begin{equation*}
|\partial_x^{\beta}\partial_{\xi}^{\alpha}a_{k,\lambda}|\lesssim \left|a\right|_{S^0}^{(|\alpha|+|\beta|)}2^{-|\alpha|k}\lambda^{|\alpha|-|\beta|}\lesssim\left|a\right|_{S^0}^{(|\alpha|+|\beta|)}2^{-(|\alpha|+|\beta|)\frac{k}{2}}.
\end{equation*}
The proof is concluded by taking $k$ sufficiently large (depending only on the symbol bounds for $a$).
\end{proof}
Next, we extend the above bounds to the $X^0$ and $Y^0$ spaces.
\begin{proposition}[Operator bounds for $X^0$ and $Y^0$]\label{LEopbounds} Let $a\in S^0$ be time-independent and let $T\lesssim 1$. Then there is $j=j(d)$ such that we have the operator bounds
\begin{equation}\label{LEopbounds1} 
\|Op(a)\|_{X^0\to X^0}+\|Op(a)\|_{Y^0\to Y^0}\lesssim 1+|a|_{S^0}^{(j)}.
\end{equation}
Moreover, there is $k_0>0$ depending only on $a$ such that if $k\geq k_0$, we also have
\begin{equation}\label{hifreqLEop}
\|Op(a)S_{\geq k}\|_{X^0\to X^0}+\|Op(a)S_{\geq k}\|_{Y^0\to Y^0}\lesssim 1+\|a\|_{L^{\infty}}.
\end{equation}
\end{proposition}
\begin{remark}
The inequality (\ref{hifreqLEop}) can be thought of as the analogue of \Cref{CVvariant} for the $X^0$ and $Y^0$ spaces. 
\end{remark}
\begin{proof}
We prove (\ref{LEopbounds1}) and remark on the very minor modifications required to prove (\ref{hifreqLEop}) where necessary. For notational convenience, we let $K_a$  denote the term on the right-hand side of (\ref{LEopbounds1}). We begin with the $X^{0}\to X^0$ bound. By definition, we have
\begin{equation}\label{X0term1}
\begin{split}
\|Op(a)f\|_{X^{0}}^2&=\sum_{k\geq 0}\|S_kOp(a)f\|_{X_k}^2
\\
&\lesssim\sum_{k\geq 0}\|S_k[Op(a),\tilde{S}_k]f\|_{X_k}^2+\sum_{k\geq 0}\|S_kOp(a)\tilde{S}_kf\|_{X_k}^2,
\end{split}
\end{equation}
for some fattened Littlewood-Paley projection $\tilde{S}_k$. For the first term, we can crudely estimate using H\"older in $T$ and dyadic summation, 
\begin{equation*}\label{X0boundcommutator}
\begin{split}
\left(\sum_{k\geq 0}\|S_k[Op(a),\tilde{S}_k]f\|_{X_k}^2\right)^{\frac{1}{2}}&\lesssim \left(\sum_{k\geq 0}2^{k}\|S_k[Op(a),\tilde{S}_k]f\|_{L_T^{\infty}L_x^2}^2\right)^{\frac{1}{2}}
\\
&\lesssim \sup_{k\geq 0}\|[Op(a),\tilde{S}_k]f\|_{L_T^{\infty} H_x^{\frac{1}{2}+\epsilon}}
\\
&\lesssim_{\epsilon} K_a\|f\|_{L_T^{\infty}H_x^{-\frac{1}{2}+\epsilon}}
\\
&\lesssim K_a\|f\|_{X^0},
\end{split}
\end{equation*}
where in the second to third line, we used \Cref{Sobolevbound} and that $[Op(a),\tilde{S}_k]\in OPS^{-1}$ (which has symbol bounds uniform in $k$, thanks to \Cref{symbcalc}). 
\begin{remark}
We remark briefly on one change needed here for the proof of (\ref{hifreqLEop}). If $f$ is replaced by $S_{>k}f$ for some sufficiently large $k$, then in the third line above, we can estimate
\begin{equation*}
K_a\|S_{>k}f\|_{L_T^{\infty}H_x^{-\frac{1}{2}+\epsilon}}\lesssim K_a2^{-k(\frac{1}{2}-\epsilon)}\|f\|_{X^0},  
\end{equation*}
and so, the factor of $K_a$ can be replaced by $1$ in the above estimate by taking $k$ large enough.
\end{remark}

Now, we turn to the second term in (\ref{X0term1}). By square summing and \Cref{classicalCV} (or \Cref{CVvariant} when proving (\ref{hifreqLEop})), it suffices to estimate the $2^{-\frac{k}{2}}X$ component of the $X_k$ norm.
For this, we have
\begin{equation*}
\|Op(a)\tilde{S}_kf\|_{2^{-\frac{k}{2}}X}=\sup_{l\in\mathbb{N}_0}\sup_{Q\in Q_l}2^{\frac{k-l}{2}}\|\chi_{Q}Op(a)\tilde{S}_kf\|_{L_T^2L_x^2}.
\end{equation*}
Using the $L^2\to L^2$ bound for $Op(a)$ from \Cref{classicalCV} and that $[Op(a),\chi_Q]\in OPS^{-1}$ with bounds independent of $l$, we obtain for each $Q\in Q_l$, 
\begin{equation}\label{commutatorQl}
\begin{split}
2^{\frac{k-l}{2}}\|\chi_{Q}Op(a)\tilde{S}_kf\|_{L_T^2L_x^2}&\lesssim K_a 2^{\frac{k-l}{2}}\|\chi_Q\tilde{S}_kf\|_{L_T^2L_x^2}+2^{\frac{k}{2}}K_a\|\tilde{S}_kf\|_{L_T^{2}H_x^{-1}}
\\
&\lesssim K_a\|\tilde{S}_kf\|_{X_k}+2^{-\frac{k}{2}}K_a\|\tilde{S}_kf\|_{L_T^{\infty}L_x^2}.
\end{split}
\end{equation}
Therefore,
\begin{equation*}
\left(\sum_{k\geq 0}\|Op(a)\tilde{S}_kf\|_{X_k}^2\right)^{\frac{1}{2}}\lesssim K_a\|f\|_{X^{0}},
\end{equation*}
which establishes the $X^0\to X^0$ bound. The high-frequency variant (\ref{hifreqLEop}) is proved by using instead \Cref{CVvariant} in place of \Cref{classicalCV} above and using the frequency gain in the latter term in the second line of (\ref{commutatorQl}) to absorb the factor of $K_a$.
\medskip

Next, we turn to the $Y^0\to Y^0$ bound. Again, by definition, we have
\begin{equation*}\label{twotermsY0}
\begin{split}
\|Op(a)f\|_{Y^{0}}^2&=\sum_{k\geq 0}\|S_kOp(a)f\|_{Y_k}^2
\\
&\lesssim \sum_{k\geq 0}\|S_k[Op(a),\tilde{S}_k]f\|_{Y_k}^2+\sum_{k\geq 0}\|S_kOp(a)\tilde{S}_kf\|_{Y_k}^2.
\end{split}
\end{equation*}
For the first term, we estimate using the embedding $L_T^1L_x^2\subset Y_k$ and that $[Op(a),\tilde{S}_k]\in OPS^{-1}$ to obtain
\begin{equation*}
\left(\sum_{k\geq 0}\|S_k[Op(a),\tilde{S}_k]f\|_{Y_k}^2\right)^\frac{1}{2}\lesssim K_a\|f\|_{L_T^1H_x^{-1+\epsilon}}\lesssim K_a\|f\|_{Y^0},
\end{equation*}
where the last inequality follows from the fact that $Y^0\subset L_T^1H_x^{-\frac{1}{2}-\epsilon}$.   Similarly to before, for the bound (\ref{hifreqLEop}) when $f$ is replaced by $S_{>k}f$, we have
\begin{equation*}
K_a\|S_{>k}f\|_{L_T^1H_x^{-1+\epsilon}}\lesssim K_a2^{-k(\frac{1}{2}-2\epsilon)}\|f\|_{Y^0},  
\end{equation*}
and so, the factor of $K_a$ can be replaced by $1$ if $k$ is large enough. For the second term, we use duality. Let $g\in X_k$ with $\|g\|_{X_k}\leq 1$. We have by \Cref{symbcalc} and similar embeddings as above,
\begin{equation*}
\begin{split}
|\langle S_k Op(a)\tilde{S}_kf,g\rangle|&\lesssim \|\tilde{S}_kf\|_{Y_k}\|\tilde{S}_k(Op(\overline{a}))S_kg\|_{X_k}+K_a\|\tilde{S}_kf\|_{L_T^1H_x^{-1}}\|S_kg\|_{L_T^{\infty}L_x^2} 
\\
&\lesssim \|\tilde{S}_kf\|_{Y_k}\|\tilde{S}_k(Op(\overline{a}))S_kg\|_{X_k}+2^{-k(\frac{1}{2}-\epsilon)}K_a\|\tilde{S}_kf\|_{Y_k}.
\end{split}
\end{equation*}
Again, if $k$ is large enough, the $K_a$ factor in the latter term can be discarded.  Using the $X^0\to X^0$ bound already established above, we also have
\begin{equation*}
\|\tilde{S}_k(Op(\overline{a}))S_kg\|_{X_k}\lesssim  \|Op(\overline{a})S_kg\|_{X^0}\lesssim K_a,
\end{equation*}
where $K_a$ can be replaced by $1+\|a\|_{L^{\infty}}$ if $k$ is large enough. The proof of (\ref{LEopbounds1}) is then concluded by dyadic summation.
\end{proof}
In our analysis later, we will sometimes need to estimate commutators of pseudodifferential and paradifferential operators. For this purpose, we recall the following Coifman-Meyer type estimate from (3.6.4) and (3.6.5) of \cite{taylor1991pseudodifferential}.
\begin{proposition}[Coifman-Meyer type bound]\label{CM}
For every $m,\sigma\in\mathbb{R}$ and $P\in OPS^m$, we have
\begin{equation}\label{lowhicom}
\|[P,T_g]f\|_{H^{\sigma}}\leq C\|g\|_{W^{1,\infty}}\|f\|_{H^{\sigma+m-1}},
\end{equation}
where $C>0$  is a constant depending on $P$ and $\sigma$. 
\end{proposition}
\subsection{Multilinear and Moser estimates}
Here we recall several of the multilinear and Moser-type estimates for the local energy and dual local energy spaces defined above. 
\begin{proposition}[Proposition 3.1 in \cite{marzuola2012quasilinear}]\label{Algebraproperty}
Let $s>\frac{d}{2}$. Then for $u,v\in l^1X^s$ we have the algebra property
\begin{equation*}
\|uv\|_{l^1X^s}\lesssim \|u\|_{l^1X^s}\|v\|_{l^1X^s}.
\end{equation*}
We also have the Moser-type estimate,
\begin{equation*}
\|F(u)\|_{l^1X^s}\lesssim \|u\|_{l^1X^s}(1+\|u\|_{l^1X^s})c(\|u\|_{L^{\infty}}),
\end{equation*}
for $s>\frac{d}{2}$ and any smooth function $F$ with $F(0)=0$.
\end{proposition}
We next recall some elementary bilinear estimates for the $l_k^1Y_k$ spaces.
\begin{proposition}[Bilinear estimates]\label{bilinear}  The following bilinear estimates hold for $l_k^1Y_k$ spaces. 
\begin{enumerate}
\item (High-low interactions). If $j<k-4$,
\begin{equation*}
\|S_juS_kv\|_{l_k^1Y_k}\lesssim 2^{j(\frac{d}{2}+1)}2^{-k}\|S_kv\|_{2^{-\frac{k}{2}}X}\|S_ju\|_{l_j^1L_T^{\infty}L_x^2}.
\end{equation*}
\item (Balanced interactions). If $|i-j|\leq 4$ and $i,j\geq k-4$,
\begin{equation*}
\|S_k(S_iuS_jv)\|_{l_k^1Y_k}\lesssim 2^{\frac{jd}{2}}\|S_iu\|_{l_i^1L_T^2L_x^2}\|S_jv\|_{L_T^{\infty}L_x^2}.
\end{equation*}
\end{enumerate}
\end{proposition}
\begin{proof}
This is a slight refinement of Lemma 4.3 in \cite{marzuola2021quasilinear}. The proof is almost identical, so we omit the details. 
\end{proof}
By dyadic summation, the following is a consequence of \Cref{bilinear}.
\begin{proposition}[Paradifferential bilinear estimates]\label{parabil}
Let $s_0>\frac{d}{2}+2$. Then for every $\sigma\geq 0$, we have
\begin{equation*}
\begin{split}
&\|T_uv\|_{l^1Y^{\sigma}}+\|(T_v-v)u\|_{l^1Y^{\sigma}}\lesssim \|u\|_{l^1X^{s_0-1}}\|v\|_{X^{\sigma-1}},
\\
&\|T_uv\|_{l^1Y^{\sigma}}+\|(T_v-v)u\|_{l^1Y^{\sigma}}\lesssim \|u\|_{l^1X^{s_0-2}}\|v\|_{X^{\sigma}}.
\end{split}
\end{equation*}
If $0\leq\sigma\leq s_0$, we also have
\begin{equation}\label{restrictedpara}
\begin{split}
&\|(T_v-v)u\|_{l^1Y^{\sigma}}\lesssim \|u\|_{l^1X^{\sigma-1}}\|v\|_{X^{s_0-1}},
\\
&\|(T_v-v)u\|_{l^1Y^{\sigma}}\lesssim \|u\|_{l^1X^{\sigma-2}}\|v\|_{X^{s_0}}.
\end{split}
\end{equation}
\end{proposition}
We next  state a closely related commutator estimate, which is a slight refinement of the version in \cite{marzuola2012quasilinear}.
\begin{proposition}\label{commutatorwithmultiplier}
Let $s>\frac{d}{2}+2$ and let $A\in S^0$ be a Fourier multiplier. Then we have
\begin{equation*}
\|\nabla [S_{<k-4}g,A(D)]\nabla S_ku\|_{l^1Y^0}\lesssim_A \|g-g_{\infty}\|_{l^1X^{s}}\|S_ku\|_{X^0},
\end{equation*}
\end{proposition}
where $g_{\infty}$ is any constant matrix.
\begin{proof}
The proof of this is essentially identical to the proof of Proposition 3.2 in \cite{marzuola2012quasilinear}. We omit the details.
\end{proof}
\begin{remark}\label{comremark}
We note that if $A\in S^m$ is a Fourier multiplier for some real number $m\geq 0$ then we can write the commutator in the above proposition as
\begin{equation*}
[S_{<k-4}g,A(D)]\nabla S_ku=2^{mk}\tilde{S}_k[S_{<k-4}g,2^{-mk}A(D)\tilde{S}_k]\nabla S_ku    ,
\end{equation*}
for some fattened projection $\tilde{S}_k$. Since $2^{-mk}A(D)\tilde{S}_k\in OPS^0$ with symbol bounds uniform in $k$, we have from \Cref{commutatorwithmultiplier},
\begin{equation*}
\|[S_{<k-4}g,A(D)]\nabla S_ku\|_{l^1Y^0}\lesssim 2^{(m-1)k}\|g-g_{\infty}\|_{l^1X^s}\|S_ku\|_{X^0}. 
\end{equation*}
\end{remark}
The next bound will allow us to precisely estimate certain error terms in the dual local energy space $Y^0$ which involve commutators of pseudodifferential and paradifferential operators. This is essentially a variant of \Cref{CM} but for the $X$ and $Y$ spaces.
\begin{proposition}[$X^{-2}\to Y^0$ commutator estimate]\label{commutatorestimate} Let $T\lesssim 1$, $\mathcal{O}\in OPS^0$ be time-independent with symbol $O\in S^0$ and let $s_0>\frac{d}{2}+2$. Moreover, let $g$ be a function such that $g-g_{\infty}\in l^1X^{s_0}$ for some constant $g_{\infty}$. Then we have the estimate
\begin{equation*}
\|[\mathcal{O},T_g]f\|_{Y^0}\leq C\|g-g_{\infty}\|_{l^1X^{s_0}}\|f\|_{X^{-2}},
\end{equation*}
where $C$ depends only on $\mathcal{O}$. 
\end{proposition}
\begin{proof}
Clearly, it suffices to prove the claim with $g_{\infty}=0$. Moreover, it suffices to work with the principal part of the  commutator  since the remainder is bounded from $H_x^{-2}\to L_x^2$ uniformly in $T$ with norm $\lesssim_{\mathcal{O}}\|g\|_{L_T^{\infty}C^{2,\epsilon}}$ for some sufficiently small $\epsilon>0$, which by Sobolev embedding can be controlled by $\|g\|_{l^1X^{s_0}}$. The principal symbol $p$ for $[T_g,\mathcal{O}]$ is
\begin{equation*}
\begin{split}
p(x,\xi)&:=-i\sum_{k\geq 0}\{S_{<k-4}g(x)\widehat{S}_k(\xi),O\}
\\
&=-i\sum_{k\geq 0}S_{<k-4}g(x)\nabla_{\xi}\widehat{S}_k(\xi)\cdot\nabla_xO+i\sum_{k\geq 0}S_{<k-4}\nabla_xg(x)\widehat{S}_k(\xi)\cdot\nabla_\xi O=:p_1+p_2.
\end{split}
\end{equation*}
First, we consider bounds for $P_1:=Op(p_1)$. Modulo  an operator which is bounded from $H_x^{-2}\to L_x^2$ with norm $\lesssim_\mathcal{O} \|g\|_{L_T^{\infty}C^{2,\epsilon}}$, we can write
\begin{equation*}
P_1=-i\sum_{k\geq 0}S_{<k-4}g(x)(\nabla_{\xi}\widehat{S}_k)(D)\cdot Op(\nabla_xO)\tilde{S}_k+\mathcal{O}_{L_T^{\infty}H_x^{-2}\to L_T^{\infty}L^2_x}(1)    ,
\end{equation*}
where $\tilde{S}_k$ is a slightly fattened Littlewood-Paley projection. As $\nabla_{\xi}\widehat{S}_k$ is localized at frequency $\approx 2^k$, we can use \Cref{bilinear} and dyadic summation to estimate
\begin{equation*}
\|P_1f\|_{Y^0}\lesssim \|g\|_{l^1X^{s_0-1}}\|f\|_{X^{-2}}    .
\end{equation*}
A similar argument for $P_2:=Op(p_2)$ gives 
\begin{equation*}
\|P_2f\|_{Y^0}\lesssim \|\nabla_x g\|_{l^1X^{s_0-1}}\|f\|_{X^{-2}},
\end{equation*}
which concludes the proof.
\end{proof}
Finally, we state  versions of some of the above bilinear and Moser estimates which are phrased in terms of frequency envelopes. This will be convenient for establishing the finer properties of the solution map later on, such as the continuous dependence of the solution on the initial data. From Proposition 3.2 of \cite{marzuola2012quasilinear}, we have the following estimates.
\begin{proposition}[Frequency localized estimates I]\label{freqenv1}
Let $s>\frac{d}{2}$ and let $u,v\in l^1X^s$ with $l^1X^s$ frequency envelopes given by $a_k$ and $b_k$, respectively. Then for each $k\in\mathbb{N}_0$, we have
\begin{equation*}
\|S_k(uv)\|_{l^1X^{s}}\lesssim (a_k+b_k)\|u\|_{l^1X^s}\|v\|_{l^1X^s}.
\end{equation*}
Moreover, if $F$ is a smooth function with $F(0)=0$, then we have 
\begin{equation*}
\|S_k(F(u))\|_{l^1X^s}\lesssim a_k\|u\|_{l^1X^s}(1+\|u\|_{l^1X^s})c(\|u\|_{L^{\infty}}).
\end{equation*}
\end{proposition}
\begin{proposition}[Frequency localized estimates II]\label{freqenv2} Let $s>\frac{d}{2}+2$. The following estimates hold for $k\in\mathbb{N}_0$.
\begin{enumerate}
\item Let $0\leq\sigma\leq s$ and let $u\in l^1X^{\sigma-1}$ and $v\in l^1X^{s-1}$ with corresponding frequency envelopes $a_k$ and $b_k$, respectively. We have
\begin{equation*}\label{freqb1}
\|S_k(uv)\|_{l^1Y^{\sigma}}\lesssim (a_k+b_k)\|u\|_{l^1X^{\sigma-1}}\|v\|_{l^1X^{s-1}}.
\end{equation*}
\item Let $0\leq\sigma\leq s-1$ and let $u\in l^1 X^{\sigma}$ and $v\in l^1X^{s-2}$ with corresponding frequency envelopes $a_k$ and $b_k$, respectively. We have
\begin{equation*}\label{loc3}
\|S_k(uv)\|_{l^1Y^{\sigma}}\lesssim (a_k+b_k)\|u\|_{l^1X^{\sigma}}\|v\|_{l^1X^{s-2}}.
\end{equation*}
\item Let $0\leq\sigma\leq s$ and let $u\in l^1X^{\sigma}$ and $v\in l^1X^{s-2}$ with corresponding frequency envelopes $a_k$ and $b_k$, respectively. We have
\begin{equation*}
\|S_k(vS_{\geq k-4}u)\|_{l^1Y^{\sigma}}\lesssim (a_k+b_k)\|u\|_{l^1X^{\sigma}}\|v\|_{l^1X^{s-2}}.
\end{equation*}
\end{enumerate}
\end{proposition}
\section{Overview of the proof}\label{OOTP}
In this section, we give an overview of the key ideas that go into the proof of \Cref{quadraticmain}. We recall that our essential aim is  to establish local well-posedness for the system
\begin{equation}\label{div-form-outline}
\begin{cases}
&i\partial_tu+\partial_j g^{jk}(u,\overline{u})\partial_k u=F(u,\overline{u},\nabla u,\nabla\overline{u}),\hspace{5mm} u:\mathbb{R}\times\mathbb{R}^d\to\mathbb{C}^m,
\\
&u(0,x)=u_0(x),
\end{cases}
\end{equation}
in  the $l^1H^s$ scale, for  $s>s_0>\frac{d}{2}+2$. As we shall see, our scheme builds on and complements the ideas from \cite{jeong2023cauchy,kenig2005variable, MR2263709, MR1660933, MR2096797, marzuola2021quasilinear}, but also has several important novelties.
\subsection{The linear and paradifferential ultrahyperbolic flows}
The main component of our argument involves a careful analysis of the  linear ultrahyperbolic flow  
\begin{equation}\label{linearizedeqn}
\begin{cases}
&i\partial_tv+\partial_jg^{jk}\partial_k v+b^j\partial_jv+\tilde{b}^j\partial_j\overline{v}=f,
\\
&v(0,x)=v_0(x),
\end{cases}
\end{equation}
which is naturally associated with the linearization of (\ref{div-form-outline}). Here, the metric $g^{jk}$ is  real, nontrapping, symmetric and non-degenerate, and the coefficients $g$, $b$, and $\tilde{b}$ satisfy the asymptotic flatness conditions
\begin{equation}\label{asympflatoutline}
\|g-g_{\infty}\|_{l^1X^{s_0}}+\|\partial_tg\|_{l^1X^{s_0-2}}+\|(b,\tilde{b})\|_{l^1X^{s_0-1}}+\|\partial_t(b,\tilde{b})\|_{l^1X^{s_0-3}}\leq M,
\end{equation}
where $M>0$ is a fixed constant and $g_{\infty}$ is a constant, non-degenerate, symmetric matrix. Note that in the special case where $v$ corresponds to the linearization around a solution $u$ to (\ref{div-form-outline}), the coefficients and inhomogeneous source term in \eqref{linearizedeqn} take the form
\begin{equation}\label{coefficients1}
\begin{cases}
&b^j:=\nabla_u g^{jk}\partial_ku-\nabla_{(\nabla u)_j}F,\hspace{4mm}\tilde{b}^j:=\nabla_{\overline{u}}g^{jk}\partial_ku-\nabla_{(\nabla\overline{u})_j}F,
\\
&f:=(\nabla_{\overline{u}}F-\partial_j(\nabla_{\overline{u}}g^{jk})\partial_ku)\overline{v}+(\nabla_uF-\partial_j(\nabla_ug^{jk})\partial_ku)v,
\end{cases}
\end{equation}
where we have suppressed the dependence on $u$ in the coefficients for simplicity of notation. In our general analysis of (\ref{linearizedeqn}), we will not require that $b$, $\tilde{b}$ and $f$ arise from solutions to (\ref{div-form-outline}) via linearization.
\medskip

An equation that is closely related to (\ref{linearizedeqn}) is the associated linear paradifferential  flow 
\begin{equation}\label{paralinflowoutline}
\begin{cases}
&i\partial_tv+\partial_jT_{g^{jk}}\partial_k v+T_{b^j}\partial_jv+T_{\tilde{b}^j}\partial_j\overline{v}=f,
\\
&v(0,x)=v_0(x),
\end{cases}
\end{equation}
which extracts the leading  part of the linear flow. Here, the paradifferential operator $T_g$ is defined as in (\ref{paraproductdef}). Again, when $v$ arises from the linearization around a solution $u$ to (\ref{div-form-outline}), $b^j$ and $\tilde{b}^j$ remain  as in (\ref{coefficients1}), but now
\begin{equation*}
\begin{split}
f&:=(\nabla_{\overline{u}}F-\partial_j(\nabla_{\overline{u}}g^{jk})\partial_ku)\overline{v}+(\nabla_uF-\partial_j(\nabla_ug^{jk})\partial_ku)v+(\partial_jT_{g^{jk}}\partial_k-\partial_jg^{jk}\partial_k)v
\\
&+(T_{b^j}-b^j)\partial_jv+(T_{\tilde{b}^j}-\tilde{b}^j)\partial_j\overline{v}.
\end{split}
\end{equation*}
In this case, $f$  can be thought of as being comprised of perturbative error terms when measured in the dual local energy space $l^1Y^0$.  These terms either have a suitable algebraic balance of derivatives between the coefficients and $v$ or have coefficient functions that are at high or comparable frequency relative to $v$. 
\subsection{Quantitative nontrapping and the bicharacteristic flow} As in \cite{marzuola2021quasilinear}, to adequately study the linear (and ultimately nonlinear) problem, we will need a  quantitative measure of nontrapping. For our purposes, we will only need to define nontrapping for time-independent metrics $g$ with regularity and decay given by
\begin{equation}\label{datasizeoutline}
\|g-g_{\infty}\|_{l^1H^{s_0}}\leq M,\hspace{5mm}\frac{d}{2}+2<s_0<s  ,
\end{equation}
where $M>0$ is a fixed constant and $g_{\infty}$ is a constant, non-degenerate, symmetric matrix. Note that the condition  \eqref{datasizeoutline} guarantees that $g\in C^{2,\delta}$, which in particular ensures that the corresponding Hamilton flow,
\begin{equation*}
(\dot{x}^t,\dot{\xi}^t)=(\nabla_{\xi}a(x^t,\xi^t),-\nabla_xa(x^t,\xi^t)),\hspace{5mm}a(x,\xi)=-g^{ij}(x)\xi_i\xi_j, \hspace{5mm} (x^0,\xi^0)=(x,\xi),
\end{equation*}
is locally well-posed. The first preliminary objective of \Cref{THE BICHARAC FLOW} is to show that under the nontrapping assumption on $g$ and the asymptotic flatness condition (\ref{datasizeoutline}), the Hamilton flow is in fact globally defined. This is not automatic when $\Delta_g$ is not elliptic. Indeed, although $g^{ij}\xi_i\xi_j$ is conserved by the Hamilton flow, unlike in the elliptic case, it does not necessarily control the size of $|\xi^t|$ in our setting. 
\medskip

The second objective of \Cref{THE BICHARAC FLOW} is to provide a  quantitative measure of nontrapping. For this, we define a function $L: [0,\infty)\to [0,\infty)$ where $L(R)$ measures (roughly speaking) the maximal amount of time any initially  unit speed bicharacteristic can intersect the ball $B_R(0)$. Our definition differs slightly from the definition in \cite{marzuola2021quasilinear}, as they define $L$ in terms of the Hamilton flow projected onto the co-sphere bundle $|\xi^t|=1$. This latter definition is natural in the elliptic case, in light of the conservation of $g^{ij}\xi_i\xi_j$, but is not quite suitable for our problem. Analogously to \cite{marzuola2021quasilinear}, we show that our nontrapping parameter $L$ is stable under small perturbations of the metric, which will be important later on when we analyze the linear and nonlinear Schr\"odinger flows. 
\subsection{The main linear estimate}
The crux of our argument centers around establishing the following key estimate for the linear paradifferential  flow (\ref{paralinflowoutline}):
\begin{equation}\label{linearestimateoutline}
\|v\|_{l^1X^{\sigma}}\leq C(M,L)(\|v_0\|_{l^1H^{\sigma}}+\|f\|_{l^1Y^\sigma}),\hspace{5mm}0\leq\sigma  ,  
\end{equation}
which as a simple consequence yields the following estimate for the linear flow (\ref{linearizedeqn}):
\begin{equation*}
\|v\|_{l^1X^{\sigma}}\leq C(M,L)(\|v_0\|_{l^1H^{\sigma}}+\|f\|_{l^1Y^\sigma}),\hspace{5mm}0\leq\sigma\leq s_0.    
\end{equation*}
Here, $C(M, L)$ is a constant depending on the coefficient size $M$  in (\ref{asympflatoutline}) and on the nontrapping parameter $L$ for $g$ within a fixed compact set whose size depends on the profile of the metric $g$ and the rate of decay of the coefficients $b^j$ and $\tilde{b}^j$. In \Cref{THE LINEAR ULTRA FLOW}, we reduce establishing the above two estimates to establishing the following simpler bound for the linear paradifferential flow:
\begin{equation}\label{simplifiedboundoutline}
\|v\|_{X^{\sigma}}\leq C(M,L)(\|v_0\|_{H^{\sigma}}+\|f\|_{Y^{\sigma}}),\hspace{5mm}\sigma\geq 0,
\end{equation}
in the setting where $\widehat{v}$ is supported at frequencies $\gtrsim 2^{k_1}$, where $k_1$ is some sufficiently large parameter. This latter reduction follows in a straightforward manner as low-frequency errors can be controlled by taking $T$  small enough depending on $k_1$. The reason we perform this reduction is so that we can make use of the more precise pseudodifferential mapping properties in Propositions \ref{CVvariant} and \ref{LEopbounds}, which provide  high frequency operator bounds for pseudodifferential operators that depend only on the $L^{\infty}$ norm of their symbols (as long as the symbol itself does not depend on $k_1$). This will be of critical importance in Sections~\ref{Sec lin flow} and \ref{LED sec},  as we will explain below. 
\subsection{$L^\infty_tL^2_x$ bounds for the linear flow}
We next give an outline of \Cref{Sec lin flow}, where we prove the first of the two main components of the bound (\ref{simplifiedboundoutline}). The main aim of \Cref{Sec lin flow} is to establish control of the $L_T^{\infty}H_x^{\sigma}$ norm of $v$. Throughout the discussion, we assume that $\widehat{v}$ is supported at frequencies larger than $2^{k_1}$. Given $\epsilon>0$, our aim is to prove an estimate of the form
\begin{equation}\label{intro energy norm}
\|v\|_{L_T^{\infty}H_x^{\sigma}}\leq C(M,L)(\|v_0\|_{H^{\sigma}}+\|f\|_{Y^{\sigma}})+\epsilon\|v\|_{X^{\sigma}},\hspace{5mm}\sigma\geq 0,  
\end{equation}
 on a time interval $[0,T]$ whose length depends on $M,$ $L$ and $\epsilon$. In order to  clearly outline the main techniques, we focus on the case $\sigma=0$. The key objective in this part of the proof is to construct a spatially truncated version of the renormalization operator in (\ref{renormalizationnotideal}) which conjugates away the ``main" portion of the term $\operatorname{Re}(b^j)\partial_jv$. As noted in the introduction, obtaining an $L_T^{\infty}L_x^2$ bound for (\ref{paralinflowoutline}) is straightforward in the absence of such a term. Unlike the symbol in (\ref{renormalizationnotideal}), however, we  want the symbol $O(x,\xi)$ of our renormalization operator $\mathcal{O}$ to be time-independent and to belong to $S^0$. In view of the first goal (and also to ensure that our symbol is smooth) we  truncate in frequency and time,  rewriting the  paradifferential linear flow  as
 \begin{equation}\label{paraeqnoutline}
\begin{cases}
&i\partial_tv+\partial_jT_{g^{ij}}\partial_i v+b^j_{<k_0}(0)\partial_jv+\tilde{b}^j_{<k_0}(0)\partial_j\overline{v}=f+\mathcal{R}^1,
\\
&v(0)=v_0.
\end{cases}
\end{equation}
We then prove that if the frequency truncation parameter $k_0$ is large enough and $T$ is small enough, the resulting error term $\mathcal{R}^1$ satisfies the perturbative bound
\begin{equation*}
\|\mathcal{R}^1\|_{Y^0}\leq\epsilon\|v\|_{X^0}.
\end{equation*}
In order  to guarantee smoothness of our symbol $O$, we will only work with the Hamilton flow $(x^t,\xi^t)$ for the truncated symbol $a:=-g^{ij}_{<k_0}(0)\xi_i\xi_j$. By our stability results, the truncated metric $g^{ij}_{<k_0}(0)$ will be nontrapping with comparable parameters to $g^{ij}$ if $k_0$ is sufficiently large and $T$ is sufficiently small. The downside of working exclusively with these truncated quantities, however, is that we will need to obtain an estimate of the form
\begin{equation*}
\|[\mathcal{O},\partial_i(T_{g^{ij}}-g^{ij}_{<k_0}(0))\partial_j]\|_{X^0\to Y^0}\leq\epsilon ,   
\end{equation*}
when we commute the equation with $\mathcal{O}$. Establishing such a bound is not completely trivial, but can be handled expeditiously with the tools developed in \Cref{Prelims}. Knowing this, our construction of $O$  proceeds as follows: First, we fix a large parameter $R$ such that the coefficients in the equation are small outside of $B_{R}(0)$. That is,
\begin{equation}\label{smallnessoutline}
\|(g-g_{\infty})\chi_{>R}\|_{l^1X^{s_0}}+\|(b,\tilde{b})\chi_{>R}\|_{l^1X^{s_0-1}}\ll\epsilon.    
\end{equation}
We then make the ansatz $O:=e^{\psi_1+\psi_2}$, where $\psi_1,\psi_2\in S^0$. The purpose of the symbol $\psi_1$ is to arrange for the leading order cancellation
\begin{equation*}
\{a,\psi_1\}+\operatorname{Re}(b^j_{<k_0}(0))\xi_j\geq 0   , 
\end{equation*}
within the region $B_{R}(0)$ where the coefficient $b^j$ is potentially large. Roughly speaking (but not exactly), we will take
\begin{equation*}
\psi_1(x,\xi):=-\chi_{<2R}(x)\int_{-\infty}^{0}\operatorname{Re}((\chi_{<4R}b_{<k_0}(0))(x^t))\cdot\xi^tdt    .
\end{equation*}
The symbol $\psi_2$ will then be chosen to correct the error terms in the transition region $|x|\approx R$ which appear  when derivatives are applied to the localization $\chi_{<2R}$. The resulting symbol $O$ will turn out to be a classical time-independent $S^0$ symbol, allowing us to avoid the more exotic symbol class (\ref{CKSclass}) used in \cite{kenig2005variable}. We remark importantly that we only attempt to cancel the leading part of the first order term $\operatorname{Re}(b^j_{<k_0}(0))\partial_jv$ in \eqref{paraeqnoutline}  as the remaining first order terms exhibit a favourable energy structure from the point of view of $L^\infty_TL^2_x$ estimates. See \Cref{basicEE} for a more precise description of this.
\medskip

We crucially note that the spatial localization in $O$ comes with one significant caveat. Namely, it only conjugates away the bad first order term within the region $B_R(0)$. Therefore, we still have to estimate the residual error term $\chi_{>R}\operatorname{Re}(b^j_{<k_0}(0))\partial_j\mathcal{O}v$ in $Y^0$. Ideally, such an estimate  would follow easily from the smallness (\ref{smallnessoutline}) of $b^j$ outside of $B_{R}(0)$. However, the symbol bounds for $O$ grow in the parameter $R$. Therefore, we have to somehow ensure  that the $X^0\to X^0$ bounds for $\mathcal{O}$ do not counteract the smallness coming from $b^j$. This is accomplished by using the observation that, unlike the higher order symbol bounds, the $L^{\infty}$ norm of the symbol $O$ is independent of the parameter $R$ (as $R\to\infty$). Therefore, since $\widehat{v}$ is supported at  frequencies $\gtrsim 2^{k_1}$, we can make use of the bounds in \Cref{CVvariant} and \Cref{LEopbounds} to ensure  that we have an estimate essentially of the form
\begin{equation*}
\|\mathcal{O}v\|_{X^0}\lesssim \|O\|_{L^{\infty}}\|v\|_{X^0}.    
\end{equation*}
This is what will ultimately allow us to break any potential circularity in our analysis. As mentioned earlier, an analogous construction  was used to establish well-posedness for the electron MHD equations in \cite{jeong2023cauchy}.
\subsection{Local energy bounds for the linear flow}
In \Cref{LED sec}, we turn to the second of the two main components of the bound (\ref{simplifiedboundoutline}). Again, for simplicity of discussion, we take $\sigma=0$. Here, the aim is to establish control of the local energy component of the $X^0$ norm of $v$ in terms of the dual norm of $f$ and the $L^\infty_TL^2_x$ norm of $v$. More precisely, we aim to prove an estimate  of the form
\begin{equation}\label{LEsimplified}
\|v\|_{X^0}\leq C(M,L)(\|v\|_{L_T^{\infty}L_x^2}+\|f\|_{Y^{0}}).
\end{equation}
Combining this bound with the $L_T^{\infty}L_x^2$ bound \eqref{intro energy norm} (with $\epsilon$ sufficiently small), it is relatively straightforward to obtain the main bound (\ref{simplifiedboundoutline}). To obtain (\ref{LEsimplified}), we implement a novel approach based on the truncation idea used in the $L_T^{\infty}L_x^2$ bound. As before, we begin by fixing $R>0$ so that we have the smallness (\ref{smallnessoutline}) outside of $B_R(0)$. 
\medskip

Our first observation is that we can use the small data result from \cite{marzuola2012quasilinear} (which holds for the general ultrahyperbolic Schr\"odinger flows that we consider here) to reduce having to control the entire local energy component of $v$ to having to obtain the corresponding estimate within the compact set $B_{R}(0)$. Precisely, we can reduce matters to establishing the bound
\begin{equation}\label{compactregionoutline}
\|\chi_{<R}v\|_{L_T^2H_x^{\frac{1}{2}}}\leq C(M,L)(\|v\|_{L_T^{\infty}L_x^2}+\|f\|_{Y^0})+\epsilon\|v\|_{X^0}.     
\end{equation}
Note that in \eqref{compactregionoutline} we  work with  the stronger (but simpler) $L_T^2H_x^{\frac{1}{2}}$ norm  within the compact set $B_R(0)$. The starting point in the proof of this estimate is to rewrite (\ref{paralinflowoutline}) as a system for $\bold{v}:=(v,\overline{v})$:
\begin{equation*}
\partial_t\bold{v}+\textbf{P}\bold{v}+\textbf{B}_{k_0}^0\bold{v}=\bold{R},
\end{equation*}
where $\bold{P}$ is the corresponding principal operator. As in the $L_T^{\infty}L_x^2$ estimate, the operator $\textbf{B}_{k_0}^0$ is a suitable time and frequency truncated version of the first order differential operator in the paralinearized Schr\"odinger equation and $\bold{R}$ is a source term which can be controlled by the right-hand side of (\ref{compactregionoutline}) in $Y^0$. We  write $\bold{P}_{k_0}^0$ as a shorthand for the associated time and frequency truncated principal operator. Precise definitions can be found in \eqref{Refchange2}.
\medskip

The estimate \eqref{compactregionoutline} proceeds via a positive commutator argument. Our implementation can be thought of as a spatially truncated version of Doi's argument in \cite{ichi1996remarks}. Precisely, we aim to construct a real symbol $q\in S^0$ and a corresponding pseudodifferential operator $\bold{Q}$ such that the principal symbol for the commutator $[\bold{Q},\bold{P}_{k_0}^0]$ is elliptic within $B_{R}(0)$ and controls the first order term $\textbf{B}_{k_0}^0$ up to a small error. Like before, we work with the bicharacteristic flow  $(x^t,\xi^t)$ for the time and frequency truncated metric $g^{ij}_{<k_0}(0)$ to ensure that the symbol we construct is smooth and time-independent. To construct $q$, we first fix a secondary parameter $R'\gg R$ to be chosen. Similarly to before, we make the ansatz
\begin{equation*}
q:=e^{C(M)(p_1+p_2+p_3)},
\end{equation*}
where $C(M)$ is a suitably large constant and $p_1,p_2,p_3$ are $S^0$ symbols to be chosen. The choice of $p_1$ will simply ensure the ellipticity of $[\bold{Q},\bold{P}_{k_0}^0]$ in $B_{R}(0)$. We can take
\begin{equation*}
p_1(x,\xi):=-\chi_{<R'}\int_{0}^{\infty}\chi_{<R}(x^t,\xi^t)|\xi^t|dt    .
\end{equation*}
A natural next step would be to correct this symbol in the transition region $|x|\approx R'$ and use the smallness of the coefficients $(b,\tilde{b})$ outside of $B_{R}(0)$ as in the $L_T^{\infty}L_x^2$ estimate \eqref{intro energy norm}. However, this will not work because the $L^{\infty}$ bound for $p_1$  will not be uniform in $R$. Instead, we consider a second symbol $p_2$ whose purpose will be to ensure that the commutator $[\bold{Q},\bold{P}_{k_0}^0]$ controls the first order term $\bold{B}_{k_0}^0$ within the much larger compact set $B_{R'}(0)$ but with an $L^{\infty}$ bound which does not depend on the larger parameter $R'$. Roughly speaking, we will take $p_2$ to be 
\begin{equation*}
p_2(x,\xi):=-\chi_{<R'}\int_{0}^{\infty}\chi_{<R'}(x^t)\sqrt{|(b_{<k_0}(0))(x^t)|^2+|(\tilde{b}_{<k_0}(0))(x^t)|^2+L(R')^{-2}}\langle\xi^t\rangle dt ,   
\end{equation*}
which turns out to be a $S^0$ symbol with the desired properties. The symbol $p_3$ will then be chosen to correct the error in the transition region $|x|\approx R'$ similarly to the $L_T^{\infty}L_x^2$ bound. If $\bold{v}$ is localized at high enough frequency, the multiplier $\bold{Q}$ will then achieve the following key outcomes.
\begin{itemize}
\item $[\bold{Q},\bold{P}_{k_0}^0]$ will have an essentially positive-definite principal symbol which is elliptic of order $1$ within $B_R(0)$. This will permit the use of G\r arding's inequality to control $\chi_{<R}\bold{v}$ in $L_T^2H_x^{\frac{1}{2}}$.
\medskip

\item $[\bold{Q},\bold{P}_{k_0}^0]\bold{v}$ will control the first order term $\chi_{<R'}\bold{B}_{k_0}^0\bold{Q}\bold{v}$.
\medskip

\item If $\bold{v}$ is at high enough frequency $2^{k_1}$,  $\|\bold{Q}S_{>k_1-4}\|_{X^0\to X^0}$ will be independent of $R'$. This will allow us to control $\chi_{\geq R'}\bold{B}_{k_0}^0\bold{v}$ in $Y^0$ by a small factor of $\|v\|_{X^0}$ by taking $R'$ sufficiently large and using the smallness of $b^j$ and $\tilde{b}^j$ outside of $B_{R'}(0)$.
\end{itemize}
The above scheme turns out to be sufficient for closing the estimate (\ref{LEsimplified}). We remark that this method is very robust and works under extremely mild decay assumptions on the coefficients -- we essentially only require integrability along the bicharacteristic flow $(x^t,\xi^t)$. Such integrability is guaranteed by the asymptotic flatness condition (\ref{asympflatoutline}) and the fact that the metric is nontrapping.
\subsection{Well-posedness for the nonlinear equation}
Finally, in \Cref{finalsection} we will make use of the  estimate (\ref{linearestimateoutline}) for both the linear and paradifferential flows as well as its various corollaries to establish well-posedness for the nonlinear flow. Having established our key linear estimate, the scheme for establishing well-posedness is virtually identical to the one implemented in Section 7 of \cite{marzuola2021quasilinear}. We therefore only outline the minor changes, and refer  to \cite{marzuola2021quasilinear} for additional details. The interested reader may also consult \cite{ifrim2023local} for an expository presentation of the overarching well-posedness scheme.
\section{The bicharacteristic flow}\label{THE BICHARAC FLOW}
 In this section, we define  our quantitative measure of nontrapping and   establish basic properties of the bicharacteristic flow corresponding to the symbol $a(x,\xi):=-g^{ij}(x)\xi_i\xi_j$. We begin by fixing $s_0>\frac{d}{2}+2$ and letting $g$ be a time-independent metric satisfying
\begin{equation}\label{metricsize}
\|g-g_{\infty}\|_{l^1H^{s_0}}\leq M,
\end{equation}
for some constant non-degenerate symmetric matrix $g_{\infty}$. We moreover assume the non-degeneracy condition
\begin{equation}\label{nondegenerunif}
c^{-1}|\xi|\leq |g^{ij}\xi_j|\leq c|\xi|,\hspace{5mm}\forall\xi\in\mathbb{R}^d,
\end{equation}
for some constant $c>0$. By Sobolev embedding, we have for some $\delta>0$,
\begin{equation*}
\|g\|_{C^{2,\delta}}\lesssim_{g_{\infty}} 1+M.
\end{equation*}
As a consequence, for each $(x,\xi)\in\mathbb{R}^{2d}$, the bicharacteristic flow $(x^{t},\xi^{t}):=(x^{t}_{(x,\xi)},\xi^{t}_{(x,\xi)})$ given by
\begin{equation}\label{bicharacteristic}
(\dot{x}^t,\dot{\xi}^t)=(\nabla_{\xi}a(x^t,\xi^t),-\nabla_xa(x^t,\xi^t)),\hspace{5mm}(x^0,\xi^0)=(x,\xi),
\end{equation}
is well-defined in a neighborhood of $t=0$ (whose size a priori depends on $(x,\xi)$). 
\medskip

In addition to the above decay and non-degeneracy assumptions, we will also impose the condition that the metric $g$ be nontrapping. The meaning of this is given in a qualitative form by the following definition.
\begin{definition}[Nontrapping metric]\label{nontrapdef}
A non-degenerate metric $g$ is said to be \emph{nontrapping} if for every $(x,\xi)\in\mathbb{R}^d\times (\mathbb{R}^d-\{0\})$ and every compact set $K\subset\mathbb{R}^d$, the bicharacteristic $x^{t}$ intersects $K$ on a compact time interval. 
\end{definition}
As in \cite{marzuola2021quasilinear}, we will need a more quantitative description of the above definition. The quantitative parameter $L=L(R)$ we introduce should measure, in some sense, how long a given bicharacteristic can intersect $B_R(0)$. However, since the bicharacteristic flow satisfies the homogeneity law
\begin{equation}\label{homogeneity}
\xi\mapsto \lambda\xi,\hspace{5mm}t\mapsto\lambda t,
\end{equation}
such a parameter will not be uniform in the size of $\xi$. To deal with this, it is natural to restrict to data $\xi\in \mathbb{S}^{d-1}$. From the non-degeneracy of the metric, this restricts the initial speed of a given bicharacteristic to  approximately unit size.
\medskip

At this point, we fix a non-degenerate, nontrapping metric $g$ satisfying (\ref{metricsize}) and (\ref{nondegenerunif}). By a compactness argument, the function 
\begin{equation*}
L:[0,\infty)\to \mathbb [0,\infty)
\end{equation*}
given formally by
\begin{equation}\label{quantnontrap}
L(R):=\inf\{s\geq 0: |x^{t}|>R,\hspace{2mm} \forall\hspace{2mm}|t|\geq s,\hspace{2mm}\forall (x,\xi)\in B_{R}(0)\times \mathbb{S}^{d-1}\}
\end{equation}
is well-defined. We will use $L:=L(R)$ as a quantitative measure of nontrapping. 
\begin{remark}
In the case where $\Delta_g$ is elliptic, it is automatic that the bicharacteristic flow is globally well-defined because the quantity 
\begin{equation}\label{conservedbyflow}
g^{ij}\xi_i\xi_j    
\end{equation}
is preserved by the flow, which, by ellipticity, implies that $|\xi^t|$ remains bounded uniformly in $t$ by $|\xi|$. The same is not immediate when the symbol $g^{ij}\xi_i\xi_j$ is not elliptic and therefore it still needs to be proved that the bicharacteristic flow is globally well-defined. We  remark that our definition of the nontrapping parameter $L$ is slightly different than the one used in \cite{marzuola2021quasilinear}. In their article, they define $L$ in terms of the maximal time any bicharacteristic for the projected flow onto $\{|\xi^t|=1\}$ can intersect $B_R(0)$. In light of the above discussion, this is a natural definition in the case of an elliptic symbol, but is not so natural for our purposes because the bicharacteristic flow should not in general preserve any normalization of $|\xi|$ (even though (\ref{conservedbyflow}) is still preserved by the flow). We therefore only restrict the initial $\xi$ to the unit sphere in our quantitative measure of nontrapping. 
\end{remark}
Our next proposition addresses the problem of global existence and asymptotic bounds for the bicharacteristic flow when the metric is nontrapping and satisfies the decay condition $g-g_{\infty}\in l^1H^{s_0}$.
\begin{proposition}\label{globally well-defined} Let $s_0>\frac{d}{2}+2$ and let $g$ be a non-degenerate, nontrapping metric satisfying \eqref{metricsize}. Then
\begin{enumerate}
\item For each $(x,\xi)\in \mathbb{R}^{d}\times(\mathbb{R}^{d}-\{0\})$, the bicharacteristic flow for $a(x,\xi):=-g^{ij}\xi_i\xi_j$ is globally defined.
\item For every $\epsilon_0>0$ sufficiently small, there exists $R_0>0$ such that for any initially outgoing bicharacteristic $($i.e. $\dot{x}^t(0)\cdot x \geq 0)$ with data $(x,\xi)\in (\mathbb{R}^d-B_{R_0}(0))\times (\mathbb{R}^d-\{0\})$, $x^{t}$ is defined for $t\geq 0$ and is close to the flat flow in the sense that for all $t\geq 0$, we have
\begin{equation}\label{flatflowclose}
|x^{t}-x+2t g_{\infty}^{ij}\xi_j|\leq  t\epsilon_0|\xi|,\hspace{5mm}|\xi^{t}-\xi|\leq \epsilon_0|\xi|.
\end{equation}
\end{enumerate} 
\end{proposition}
\begin{proof}
The proof of this is very similar to Lemma 5.1 in \cite{marzuola2021quasilinear}. We include the short argument for completeness. We begin by choosing $R_0$ large enough so that $g$ is sufficiently close to the flat metric $g_{\infty}$ in $l^1H^{s_0}$ outside of $B_{\frac{R_0}{2}}(0)$. That is,
\begin{equation}\label{asympsmallnessbootstrap}
\|\chi_{>\frac{R_0}{2}}(g-g_{\infty})\|_{l^1H^{s_0}}<\epsilon    ,
\end{equation}
where $0<\epsilon\ll\epsilon_0\ll 1$ is some sufficiently small constant relative to $\epsilon_0$. We let $(x^t,\xi^t)$ be any initially outgoing bicharacteristic  with data $(x,\xi)\in (\mathbb{R}^d-B_{R_0}(0))\times (\mathbb{R}^d-\{0\})$ and make the bootstrap assumption that the bicharacteristic $(x^t,\xi^t)$ satisfies (\ref{flatflowclose}) on a time interval $t\in [0,T]$. Our goal will be to show that when $\epsilon>0$ is small enough,  the factor of $\epsilon_0$ in the bootstrap hypothesis can be improved to $\frac{\epsilon_0}{2}$. Thanks to the nontrapping assumption on $g$, this will clearly suffice for establishing both (i) and (ii).
\medskip

To close the bootstrap, we note that on $[0,T]$, thanks to (\ref{flatflowclose}) and the fact that $x^t$ is initially outgoing, the bicharacteristic $x^t$ remains outside $B_{\frac{3}{4}R_0}(0)$. Using this and the bootstrap hypothesis, we aim to prove the following simple lemma.
\begin{lemma}\label{timeintegrateddecay}
The following estimate holds for every $t\in [0,T]:$
\begin{equation*}
\begin{split}
\int_{0}^{t}|\nabla_xg(x^s)|ds&\lesssim \epsilon |\xi|^{-1}.  
\end{split}    
\end{equation*}
\end{lemma}
\begin{proof}
We estimate 
\begin{equation*}
\begin{split}
\int_{0}^{t}|\nabla_x g(x^s)|ds&\lesssim \sum_{k\geq 0}\sum_{Q\in Q_k}\int_{0}^{t}|\chi_Q(x^s)(S_k(\chi_{>\frac{R_0}{2}}\nabla_x g))(x^s)|ds
\\
&\lesssim |\xi|^{-1}\sum_{k\geq 0}\sum_{Q\in Q_k}2^k\|\chi_QS_k(\chi_{>\frac{R_0}{2}}\nabla_x g)\|_{L^{\infty}}
\\
&\lesssim |\xi|^{-1}\|\chi_{>\frac{R_0}{2}}\nabla_x g\|_{l^1H^{s_0-1}}
\\
&\leq |\xi|^{-1}\epsilon,
\end{split}
\end{equation*}
where in the second line we used (\ref{flatflowclose}) and the non-degeneracy of $g_{\infty}$, which ensures that the bicharacteristic $x^s$ intersects a cube of size $2^k$ for time at most $\lesssim 2^k|\xi|^{-1}$. In the third line, we used Bernstein's inequality, the fact that $s_0>\frac{d}{2}+2$ and the definition of $l^1H^{s_0-1}$. In the fourth line, we used (\ref{asympsmallnessbootstrap}). This concludes the proof.
\end{proof}
To close the bootstrap, we note that by (\ref{flatflowclose}) we have $|\xi^t|\leq (1+\epsilon_0)|\xi|$. Therefore, by using \Cref{timeintegrateddecay} and integrating in time the equation
\begin{equation*}
\frac{d}{dt}(\xi^t-\xi)=\nabla_xg^{ij}(x^t)\xi_i^t\xi_j^t    
\end{equation*}
we obtain
\begin{equation*}
|\xi^t-\xi|\lesssim \epsilon |\xi|.    
\end{equation*}
Using this bound, integrating in time the equation
\begin{equation*}
\frac{d}{dt}(x^t-x+2tg_{\infty}^{ij}\xi_j)=2(g_{\infty}^{ij}-g^{ij})(x^t)\xi_j^t-2g_{\infty}^{ij}(\xi_j^t-\xi_j)    
\end{equation*}
and using that $|(g-g_{\infty})(x^t)|\lesssim\epsilon$, we also obtain
\begin{equation*}
|x^t-x+2tg_{\infty}^{ij}\xi_j|\lesssim t\epsilon |\xi|   ,
\end{equation*}
which improves the bootstrap (\ref{flatflowclose}) if $\epsilon$ is small enough relative to $\epsilon_0$. This concludes the proof of the proposition.
\end{proof}
The next proposition shows that  the size of the nontrapping function $L$ as well as the bicharacteristic bounds are stable under small perturbations of the metric.
\begin{proposition}\label{perturbationstable}
Let $g_0$ be a non-degenerate nontrapping metric satisfying \eqref{metricsize}. For every sufficiently small $\epsilon_0>0$, there is a radius $R_0(\epsilon_0)>0$ and a constant $C_0>0$ depending only on $M$ and the profile of $g_0$ such that if $g_1$ is another non-degenerate metric satisfying 
\begin{equation}\label{smallperturbation}
\|g_0-g_1\|_{l^1H^{s_0}}< e^{-C_0L(R_0)}
\end{equation}
then the bicharacteristics corresponding to $g_1$ satisfy (ii) in \Cref{globally well-defined} with comparable parameters $R_0$ and $\epsilon_0$ and, moreover, $g_1$ is also nontrapping with comparable parameters $L_1$ and data size $M_1$.  
\begin{proof}
Choosing $e^{-C_0L(R_0)}$ so small  that
\begin{equation*}
\|g_0-g_1\|_{l^1H^{s_0}}\ll\epsilon_0    
\end{equation*}
ensures that the data size $M_1$ is comparable to $M$ and also that the proof of part (ii) of \Cref{globally well-defined} works equally well for the metric $g_1$. It therefore suffices to show that $L_1$ is comparable to $L$ for $R\leq R_0$. To do this, we fix $(x,\xi) \in B_{R_0}(0)\times \mathbb{S}^{d-1}$. The desired conclusion will follow if we can show that the bicharacteristic flows corresponding to $g_0$ and $g_1$ are close within $B_{R_0}(0)$ in the sense that 
\begin{equation}\label{hbound}
|x^{t}_0-x_1^{t}|_{L^{\infty}_{t}}+|\xi^{t}_0-\xi^{t}_1|_{L_{t}^{\infty}}\lesssim e^{-C(M)L(R_0)}
\end{equation}
for times in which $x^{t}_0$ intersects $B_{R_0}(0)$. The proof of this is similar to the proof of Proposition 5.2 in \cite{marzuola2021quasilinear} but since our nontrapping parameter $L$ is slightly different than theirs, we include the short proof. 
\medskip

We implement a simple bootstrap. First, we can restrict to a time interval $J$ such that $|J|\leq L(R_0)$. We will then assume the bound (\ref{hbound}) on some smaller interval $I\subset J$ and  establish the same bound with an improved constant. We begin by writing the equation for $\delta x^t:=x_0^t-x_1^t$ and $\delta\xi^t:=\xi_0^t-\xi_1^t$. Dropping the $i,j$ indices, we obtain
\begin{equation*}\label{linearizedbich}
\begin{cases}
&\frac{d}{dt}\delta x^t=2(g_1-g_0)(x_1^{t})\xi^{t}_1+2(g_0(x_1^t)-g_0(x_0^t))\xi_0^t-2g_0(x_1^t)\delta\xi^t,
\\
&\frac{d}{dt}\delta\xi^t=-\xi^{t}_1\nabla(g_1-g_0)(x_1^t)\xi^{t}_1-\xi^{t}_1(\nabla g_0(x_1^t)-\nabla g_0(x_0^t))\xi_1^{t}+(\xi_0^t\nabla g_0(x_0^t)\xi_0^t-\xi_1^t\nabla g_0(x_0^t)\xi_1^t),
\\
&(\delta x^{0},\delta \xi^0)=(0,0).
\end{cases}
\end{equation*}
By definition, we have $|I|\leq L(R_0)$. Moreover, by a compactness argument, there is a constant $K_0>1$ depending on the profile of $g_0$ (but not on $(x,\xi)$) such that
\begin{equation*}
|\xi_0^t|\lesssim K_0    
\end{equation*}
for every $t\in J$. By the bootstrap hypothesis, this implies the same bound for $\xi_1^t$ on $I$. From this, (\ref{smallperturbation}), the bootstrap hypothesis and the fact that $g_0\in C^2$, we obtain the bound
\begin{equation*}
\frac{d}{dt}[(\delta x^t)^2+(\delta\xi^t)^2]\lesssim e^{-2C_0L(R_0)}+C(K_0)(1+M)[(\delta x^t)^2+(\delta \xi^t)^2] .  
\end{equation*}
By Gr\"onwall's inequality and the bound $|I|\leq L(R_0)$, we obtain
\begin{equation*}
(\delta x^t)^2+(\delta\xi^t)^2\lesssim e^{-2C_0L(R_0)}e^{C(K_0)(1+M)L(R_0)}    
\end{equation*}
on $I$. Choosing $C_0$ large enough improves the bootstrap hypothesis and concludes the proof. 
\end{proof}
\end{proposition}
By combining \Cref{globally well-defined} with \Cref{perturbationstable}, we have the following immediate corollary which gives a precise quantitative bound for $\xi^{t}$.
\begin{corollary}\label{xibond}
Let $g_0$ be as in \Cref{globally well-defined}. Then  the corresponding bicharacteristic $\xi_0^{t}$ (which is defined for all $t$) satisfies the bound
\begin{equation*}
|\xi^{t}_0|\lesssim C_0|\xi|
\end{equation*}
for all $(x,\xi)\in \mathbb{R}^{2d}$ and some constant $C_0>1$ depending only on $M$ and the profile of $g_0$. Moreover, if $g_1$ is any other metric satisfying the conditions of \Cref{perturbationstable}, then the corresponding bicharacteristic $\xi_1^t$ is globally defined and satisfies the same bound with a similar constant.
\end{corollary}
\begin{proof}
For $|\xi|=1$, this follows immediately from \Cref{globally well-defined}, \Cref{perturbationstable}, and the nontrapping assumption. The general case follows from this case and the homogeneity law (\ref{homogeneity}).  
\end{proof}
We next note the following bounds for the $x$ and $\xi$ derivatives of $x^{t}$ and $\xi^{t}$.
\begin{proposition}[Higher regularity bounds]\label{bicharhigher} Let $(x,\xi)\in B_R(0)\times \mathbb{S}^{d-1}$. Let $k$ be a positive integer. Assume that the metric satisfies $g\in C^{k+1}$ and write $M_k:=\|g\|_{C^{k+1}}$. Then if $|x^t|\leq R$, there holds
\begin{equation*}
|\partial_{\xi}^{\alpha}\partial_x^{\beta} x^{t}_{(x,\xi)}|+|\partial_{\xi}^{\alpha}\partial_x^{\beta}\xi^{t}_{(x,\xi)}|\leq e^{C(M_k)L(R)},  \hspace{5mm}|\alpha+\beta|\leq k.
\end{equation*}
\end{proposition}
\begin{proof}
The proof follows by differentiating (\ref{bicharacteristic}) in the $x$ and $\xi$ variables, which leads to a differential inequality for 
\begin{equation*}
\frac{d}{dt}\left( |\partial_{\xi}^{\alpha}\partial_x^{\beta}x^t|^2+|\partial_{\xi}^{\alpha}\partial_x^{\beta}\xi^t|^2\right).
\end{equation*}
 One then concludes by inductively applying Gr\"onwall's inequality. We omit the details which are straightforward.
\end{proof}
The final result of this section shows that functions in $l^1H^{s}$ with $s>\frac{d}{2}+1$ are uniformly integrable along the bicharacteristic flow. This is what will allow us to recover the Mizohata condition.
\begin{proposition}\label{intalongbi}
Let $g$ be as in \Cref{globally well-defined} and let $s>\frac{d}{2}+1$. Let $v\in l^1H^{s}$. Then $v$ is integrable along the bicharacteristic flow and satisfies the bound
\begin{equation*}
\sup_{(x,\xi)\in\mathbb{R}^d\times \mathbb{S}^{d-1}}\|v(x^{t}_{(x,\xi)})\|_{L^1_{t}(\mathbb{R})}\lesssim (1+L(R_0))\|v\|_{l^1H^{s}},
\end{equation*}
where $R_0$ is as in \Cref{globally well-defined}.
\begin{proof}
We abbreviate $x^{t}_{(x,\xi)}$ by $x^{t}$. Without loss of generality, we may assume that $x^{t}$ intersects $B_{R_0}(0)$ only if $|t|<L(R_0)$. We then have 
\begin{equation*}
\int_{-L(R_0)}^{L(R_0)}|v(x^{t})|dt\leq 2L(R_0)\|v\|_{L^{\infty}}\lesssim L(R_0)\|v\|_{l^1H^{s}}.
\end{equation*}
By homogeneity of the flow, it therefore suffices to show that
\begin{equation*}
\int_{L(R_0)}^{\infty}|v(x^{t})|dt\lesssim \|v\|_{l^1H^{s}}.
\end{equation*}
Without loss of generality, we may assume that $x^{t}(L(R_0))$ is outgoing. Using \Cref{globally well-defined}, we see that if $t\geq L(R_0)$, then for every cube $Q\subset \mathbb{R}^d$, $x^{t}$ intersects the cube on a time interval $I$ of size at most $|I|\lesssim |Q|^{\frac{1}{d}}$. Therefore, we have
\begin{equation*}
\begin{split}
\int_{L(R_0)}^{\infty}|v(x^{t})|dt\lesssim \sum_{k\geq 0}\sum_{Q\in Q_k}2^{k} \|\chi_QS_kv\|_{L^{\infty}}\lesssim \|v\|_{l^1H^{s}},
\end{split}
\end{equation*}
where in the last step we used Bernstein's inequality and dyadic summation. Here,  the strict inequality $s>\frac{d}{2}+1$ was what allowed us to retain summability in $k$. This completes the proof.
\end{proof}
\end{proposition}
\section{The linear ultrahyperbolic flow} \label{THE LINEAR ULTRA FLOW}
Let $s_0>\frac{d}{2}+2$ and let $0\leq\sigma\leq s_0$. Here we consider the $l^1H^{\sigma}$ well-posedness of the linear ultrahyperbolic flow,
\begin{equation}\label{linearflow}
\begin{cases}
&i\partial_tv+\partial_jg^{jk}\partial_k v+b^j\partial_jv+\tilde{b}^j\partial_j\overline{v}=f,
\\
&v(0,x)=v_0(x),
\end{cases}
\end{equation}
as well as the corresponding linear paradifferential flow,
\begin{equation}\label{paralinflow}
\begin{cases}
&i\partial_tv+\partial_jT_{g^{jk}}\partial_k v+T_{b^j}\partial_jv+T_{\tilde{b}^j}\partial_j\overline{v}=f,
\\
&v(0,x)=v_0(x).
\end{cases}
\end{equation}
 We make the following basic assumptions on the metric $g^{jk}$ and the coefficients $b^j$ in the above equations:
\begin{enumerate}
\item (Non-degeneracy). The metric $g^{jk}$ is real, symmetric and non-degenerate. That is, there is $c>0$ such that for all $\xi\in\mathbb{R}^d$ we have,
\begin{equation*}\label{non-degeneracy}
c^{-1}|\xi|\leq |g^{jk}\xi_k|\leq c|\xi|.
\end{equation*}
\item (Asymptotic flatness and size). There is a constant, symmetric, non-degenerate matrix $g_{\infty}$ and a constant $M>0$ such that 
\begin{equation}\label{asympflat}
\|g-g_{\infty}\|_{l^1X^{s_0}}+\|\partial_tg\|_{l^1X^{s_0-2}}+\|(b,\tilde{b})\|_{l^1X^{s_0-1}}+\|\partial_t(b,\tilde{b})\|_{l^1X^{s_0-3}}\leq M.
\end{equation}
\item (Asymptotic smallness).  For every $\epsilon_0>0$, there is $R_0>0$ such that
\begin{equation}\label{smallness}
\|(g^{jk}-g^{jk}_{\infty})\chi_{>R_0}\|_{l^1X^{s_0}}+\|(b^j,\tilde{b}^j)\chi_{>R_0}\|_{l^1X^{s_0-1}}\leq\epsilon_0,
\end{equation}
where $0\leq \chi_{>R_0}\leq 1$ is a smooth cutoff which vanishes on $B_{R_0}(0)$ and is equal to $1$ outside of $B_{2R_0}(0)$. 
\item (Nontrapping). The metric is nontrapping with parameter $L$ as defined in (\ref{quantnontrap}).
\end{enumerate}
Note that  condition  (iii) follows from the asymptotic flatness condition  (ii). However, we prefer to make statement (iii) explicit, as it will play a prominent role in the analysis. 
\medskip

In the sequel, we will write $C(L)$ to denote a constant which depends on the parameter $L$ within some fixed compact set whose size depends on the profile of the metric $g$. The main result we aim to prove is the following.
\begin{theorem}\label{linearestimate} Let $s_0>\frac{d}{2}+2$ and $0\leq \sigma\leq s_0$. Moreover, assume that $g^{jk}, b^j,\tilde{b}^j$ satisfy the above assumptions with parameters $M$ and $L$. Then for every $f\in l^1Y^\sigma$, the equation \eqref{linearflow} is  well-posed in $l^1 H^{\sigma}$. Furthermore, there is $T_0>0$ depending on the size of $L$ within a compact set and on the data size $M$ such that for every $0\leq T\leq T_0$, we have
\begin{equation}\label{linearizedestiamte}
\|v\|_{l^1X^{\sigma}}\leq C(M,L) (\|v_0\|_{l^1H^{\sigma}}+\|f\|_{l^1Y^{\sigma}}). 
\end{equation}
The same result holds for the paradifferential flow \eqref{paralinflow} for every $\sigma\geq 0$.
\end{theorem}
As the above result holds for the paradifferential flow for all $\sigma\geq 0$, it is a straightforward consequence to deduce the following frequency envelope variant using similar reasoning to Section 5 of \cite{marzuola2012quasilinear} (see also \cite{ifrim2023local}).
\begin{corollary}\label{freqenvbound}
 Let $\sigma\geq 0$ and assume the other properties in the statement of \Cref{linearestimate}. Let $a_k$ be an admissible $l^1H^{\sigma}$ frequency envelope for $v_0$ and let $b_k$ be an admissible $l^1Y^{\sigma}$ frequency envelope for $f$. Then the solution $v$ to the paradifferential equation \eqref{paralinflow} satisfies the bound
 \begin{equation*}
 \|S_kv\|_{l^1X^{\sigma}}\leq C(M,L)\left(a_k\|v_0\|_{l^1H^{\sigma}}+b_k\|f\|_{l^1Y^{\sigma}}\right)   
 \end{equation*}
on a time interval $[0,T]$ whose length depends on the size of $L$ within a compact set and on the data size $M$. 
\end{corollary}
The main component of the proof of well-posedness for the equations (\ref{linearflow}) and (\ref{paralinflow}) is the energy estimate (\ref{linearizedestiamte}). This is because the adjoint equation, which has essentially the same form, will also satisfy a similar energy estimate. Well-posedness then follows by a standard duality argument. Therefore, we focus our attention mainly on the bound (\ref{linearizedestiamte}).
\subsection{Some simplifying reductions} We begin our analysis by making some straightforward but useful reductions which will allow us to simplify some of the steps in the proof of \Cref{linearestimate}. Our first reduction shows that by restricting the time interval to be small enough, we may assume that $\widehat{v}$ is supported at high frequency. More precisely, we have the following lemma. 
\begin{lemma}[High frequency reduction]\label{hifreqlemma} Let $\epsilon>0$. Under the assumptions of \Cref{linearestimate}, for every $k_1>0$ there is $T_0>0$ depending on $k_1$, $\epsilon$, $M$ and $\sigma$ such that for $0< T\leq T_0$, $v_{>k_1}:=S_{>k_1}v$ satisfies the equation
\begin{equation*}
\begin{cases}
&i\partial_tv_{>k_1}+\partial_jg^{ij}\partial_iv_{>k_1}+b^j\partial_jv_{>k_1}+\tilde{b}^j\partial_j\overline{v_{>k_1}}=h,
\\
&v_{>k_1}(0):=S_{>k_1}v_0,
\end{cases}
\end{equation*}
where $h$ and $S_{\leq k_1}v$ satisfy the estimate
\begin{equation*}
\|S_{\leq k_1}v\|_{l^1X^{\sigma}}+\|h\|_{l^1Y^{\sigma}}\leq C(M,k_1,\sigma)(\|v_0\|_{l^1H^{\sigma}}+\|f\|_{l^1Y^{\sigma}})+\epsilon\|v\|_{l^1X^{\sigma}}
\end{equation*}
for $0\leq\sigma\leq s_0$. The analogous result holds for the paradifferential equation \eqref{paralinflow} for $\sigma\geq 0$.
\end{lemma}
\begin{proof}
We show the proof for the full linear equation. The proof for the paradifferential flow is similar. Using the notation of the lemma, we  easily compute that
\begin{equation*}
\begin{split}
h&=S_{>k_1}f-(\partial_jg^{ij}\partial_iS_{\leq k_1}v+b^j\partial_jS_{\leq k_1}v+\tilde{b}^j\partial_jS_{\leq k_1}\overline{v})
\\
&+S_{\leq k_1}(\partial_jg^{ij}\partial_iv+b^j\partial_j v+\tilde{b}^j\partial_j\overline{v}).
\end{split}
\end{equation*}
We  clearly have
\begin{equation*}
\|S_{>k_1}f\|_{l^1Y^{\sigma}}\lesssim \|f\|_{l^1Y^{\sigma}}.
\end{equation*}
For the remaining source terms, if $0\leq\sigma\leq s_0-1$, we can estimate in $l^1L_T^1H_x^{\sigma}\subset l^1Y^{\sigma}$ in a na\"ive fashion using the frequency projection $S_{\leq k_1}$ to obtain
\begin{equation*}
\|h\|_{l^1Y^{\sigma}}\lesssim \|f\|_{l^1Y^{\sigma}}+\epsilon\|v\|_{l^1L^{\infty}_TH_x^{\sigma}},
\end{equation*}
by applying H\"older's inequality in $T$ and taking $T$ small enough (depending on $k_1$). On the other hand, for $s_0-1<\sigma\leq s_0$, we can instead use the bilinear estimates in \Cref{parabil}  to obtain
\begin{equation*}
\|h\|_{l^1Y^{\sigma}}\lesssim \|f\|_{l^1Y^{\sigma}}+C(M,k_1,\sigma)\|v\|_{l^1L^{\infty}_TL_x^2}.
\end{equation*}
We can estimate the latter term on the right using the crude energy inequality
\begin{equation*}
\|v\|_{l^1L^{\infty}_TL_x^2}\lesssim_M\|v_0\|_{l^1L_x^2}+ \|v\|_{l^1L_T^1H_x^{1}}+\|f\|_{l^1Y^{\sigma}},
\end{equation*}
which follows from a direct energy estimate for (\ref{linearflow}) where the first order terms are estimated directly in $L_T^1L_x^2$. Since $\sigma>1$, we may conclude by applying H\"older in $T$ and taking $T\ll\epsilon$ to control the second term on the right by $\epsilon\|v\|_{l^1X^{\sigma}}$. It remains to estimate $\|S_{\leq k_1}v\|_{l^1X^{\sigma}}$. Using that $S_{\leq k_1}v$ is frequency localized, we easily have
\begin{equation*}
\|S_{\leq k_1}v\|_{l^1X^{\sigma}}\lesssim 2^{k_1(\sigma+\frac{1}{2})}\|S_{\leq k_1}v\|_{l^1L_T^{\infty}L_x^2}.
\end{equation*}
We then note the na\"ive energy type estimate
\begin{equation*}
\begin{split}
\|S_{\leq k_1}v\|_{l^1L_T^{\infty}L_x^2}&\lesssim_{M,k_1} \|v_0\|_{l^1L_x^2}+\|v\|_{l^1L_T^1L_x^2}+\|f\|_{l^1Y^{0}},
\end{split}
\end{equation*}
which follows from inspecting the equation for $S_{\leq k_1}v$ and using the fact that the first and second-order terms in the resulting equation are localized to frequencies $\lesssim k_1$. Then using H\"older in $T$ and taking $T$ small enough (depending on $M$, $k_1$ and $\epsilon$) we can again control the second term on the right by $\epsilon\|v\|_{l^1X^{\sigma}}$.
This concludes the proof of the lemma for (\ref{linearflow}). A very similar argument works for the paradifferential analogue. We omit the details.
\end{proof}
\subsection{Reduction to the paradifferential flow}
As a second reduction, we reduce proving \Cref{linearestimate} to proving the corresponding estimate for the paradifferential equation. We begin by writing (\ref{linearflow}) in the paradifferential form
\begin{equation*}\label{paraeqn1}
\begin{cases}
&i\partial_tv+\partial_jT_{g^{ij}}\partial_i v+T_{b^j}\partial_jv+T_{\tilde{b}^j}\partial_j\overline{v}=f+\mathcal{R},
\\
&v(0)=v_0,
\end{cases}
\end{equation*}
where $\mathcal{R}$ is a remainder term given by
\begin{equation}\label{Remainder}
\begin{split}
\mathcal{R}=(T_{b^j}\partial_jv-b^j\partial_jv)+(T_{\tilde{b}^j}\partial_j\overline{v}-\tilde{b}^j\partial_j\overline{v})+\partial_j(T_{g^{ij}}\partial_i v-g^{ij}\partial_i v).
\end{split}
\end{equation}
Thanks to \Cref{hifreqlemma}, we may harmlessly assume that $v$ is localized to frequencies $\gtrsim 2^{k_1}$. Our next lemma shows that the error term $\mathcal{R}$ can be treated perturbatively if $k_1$ is large enough. 
\begin{lemma}[Paradifferential source terms]\label{parasourcefirst}
Assume that the estimate in \Cref{linearestimate} holds for the paradifferential flow for each $\sigma\geq 0$. Let $\epsilon>0$ and  assume that $\widehat{v}$ is supported at frequencies $|\xi|\gtrsim 2^{k_1}$. Then for $k_1$ large enough and $T$ small enough depending on $\epsilon$ and $k_1$, the remainder term $\mathcal{R}$ satisfies the estimate
\begin{equation*}
\|\mathcal{R}\|_{l^1Y^{\sigma}}\leq C(M,k_1,\sigma)(\|v_0\|_{l^1H^{\sigma}}+\|f\|_{l^1Y^{\sigma}})+\epsilon\|v\|_{l^1X^{\sigma}}.
\end{equation*}
\end{lemma}
\begin{proof}
We show the details for the first term as the estimates for the other two are similar. We split the analysis into two cases. First, assume that $\sigma\leq s_0-\delta$ where $\delta>0$ is such that $s_0-2\delta>\frac{d}{2}+2$. Then since $v$ is localized to frequencies $\gtrsim 2^{k_1}$, we may replace the coefficient $b^j$ in $(T_{b^j}\partial_jv-b^j\partial_jv)$ with $S_{\geq k_1-5}b^j$. Therefore, by (\ref{restrictedpara}) in \Cref{parabil} and Bernstein's inequality, we have
\begin{equation*}
\begin{split}
\|(T_{b^j}\partial_jv-b^j\partial_jv)\|_{l^1Y^{\sigma}}&\lesssim \|S_{\geq k_1-5}b^j\|_{l^1X^{s_0-1-\delta}}\|v\|_{l^1X^{\sigma}}
\\
&\lesssim_M 2^{-\delta k_1}\|v\|_{l^1X^{\sigma}}.
\end{split}
\end{equation*}
 Taking $k_1$ large enough, we therefore have
\begin{equation*}
\|(T_{b^j}\partial_jv-b^j\partial_jv)\|_{l^1Y^{\sigma}}\leq\epsilon\|v\|_{l^1X^{\sigma}}.
\end{equation*}
The other terms in (\ref{Remainder}) can be estimated similarly to obtain
\begin{equation*}
\|\mathcal{R}\|_{l^1Y^{\sigma}}\leq\epsilon\|v\|_{l^1X^{\sigma}}.
\end{equation*}
In the case $s_0 \geq \sigma\geq s_0-\delta>\frac{d}{2}+2$, we use instead the first estimate in \Cref{parabil} to obtain
\begin{equation*}
\begin{split}
\|(T_{b^j}\partial_jv-b^j\partial_jv)\|_{l^1Y^{\sigma}}&\lesssim \|b^j\|_{l^1X^{\sigma-1}}\|v\|_{l^1X^{s_0-2\delta}}
\\
&\lesssim_M 2^{-k_1\delta}\|v\|_{l^1X^{\sigma}},
\end{split}
\end{equation*}
where we used the fact that $v$ is localized to frequencies greater than $2^{k_1}$. Estimating the other terms in (\ref{Remainder}) in a similar fashion, and again taking $k_1$ large enough, we obtain
\begin{equation*}
\|\mathcal{R}\|_{l^1Y^{\sigma}}\leq \epsilon\|v\|_{l^1X^{\sigma}}.
\end{equation*}
This concludes the proof.
\end{proof}
\subsection{Reduction to the \texorpdfstring{$X^{\sigma}$}{} estimate} To summarize what we have so far, it now suffices to establish (\ref{linearizedestiamte}) for the paradifferential flow under the assumption that $v$ is localized to high frequency. As one final simplification, we reduce the proof of this estimate for the paradifferential flow to the corresponding $X^{\sigma}$ estimate without the $l^1$ summability. For this, we will need the small data result from \cite{marzuola2012quasilinear}.
\begin{theorem}[Small data well-posedness]\label{smalldata} Let $b^j$, $\tilde{b}^j$, $g^{ij}$ and $M$, $\sigma$ be as above. Let $0<T\leq 1$. For every $\sigma\geq 0$, there is $\delta>0$ such that if $M\leq\delta$ then \eqref{paralinflow} is well-posed in both $H^{\sigma}$ and $l^1H^{\sigma}$ with the uniform bounds
\begin{equation*}
\begin{split}
\|v\|_{X^{\sigma}}&\lesssim \|v_0\|_{H^{\sigma}}+\|f\|_{Y^{\sigma}},
\\
\|v\|_{l^1X^{\sigma}}&\lesssim \|v_0\|_{l^1H^{\sigma}}+\|f\|_{l^1Y^{\sigma}}.
\end{split}
\end{equation*}
\end{theorem}
\begin{remark}
Strictly speaking, the small data result above is only explicitly stated in the case when $g_{\infty}$ is the identity, but as remarked on page 1154 of \cite{marzuola2012quasilinear}, the result is also true when $g_{\infty}$ is of the form we consider here, and the estimates above follow almost verbatim from the proof of Proposition 4.1 in their paper.
\end{remark} 
We may now phrase our final reduction as follows.
\begin{lemma}\label{l1reduction} Let $b^j$, $\tilde{b}^j$, $g^{ij}$ and $M$, $\sigma$ be as in \Cref{linearestimate}. Assume  that the paradifferential flow  \eqref{paralinflow} admits the estimate
\begin{equation}\label{nol1estimate}
\|v\|_{X^{\sigma}}\leq C(M,L)(\|v_0\|_{H^{\sigma}}+\|f\|_{Y^{\sigma}})   
\end{equation}
for each $\sigma\geq 0$. Then the corresponding estimate in \Cref{linearestimate} in $l^1X^{\sigma}$ also holds for \eqref{paralinflow} for each $\sigma \geq 0$.
\end{lemma}
\begin{proof}
We can again harmlessly assume that $v$ is localized to frequencies $\gtrsim 2^{k_1}$. Now let $\epsilon>0$ and let $R(\epsilon)$ be such that (\ref{smallness}) holds. Using \Cref{parabil} and \Cref{smalldata}, our first aim will be to reduce to estimating $v$ in a compact set. More precisely, we aim to prove the estimate
\begin{equation}\label{interior reduction}
\|\chi_{>2R}v\|_{l^1X^{\sigma}}\lesssim \|v_0\|_{l^1H^{\sigma}}+\|f\|_{l^1Y^{\sigma}}+\|\chi_{<4R}v\|_{l^1X^{\sigma}}.
\end{equation}
This is a straightforward computation which follows by inspecting the equation for $v_{ext}:=\chi_{>2R} v$. Indeed, if we define $g_{ext}:=\chi_{>R} g+\chi_{\leq R} g_{\infty}$, $b_{ext}:=\chi_{>R}b$ and $\tilde{b}_{ext}:=\chi_{>R}\tilde{b}$, we obtain
\begin{equation*}
\begin{cases}
&i\partial_tv_{ext}+\partial_iT_{g^{ij}_{ext}}\partial_jv_{ext}+T_{b^j_{ext}}\partial_jv_{ext}+T_{\tilde{b}_{ext}^j}\partial_j\overline{v}_{ext}=f_{ext} ,
\\
&v_{ext}(0)=\chi_{>2R}v(0),
\end{cases}    
\end{equation*}
where
\begin{equation*}
\begin{split}
f_{ext}&:=\chi_{>2R}f+[\partial_iT_{g^{ij}}\partial_j+T_{b^j}\partial_j+T_{\tilde{b}^j}\partial_j,\chi_{>2R}]v+(\partial_iT_{g^{ij}_{ext}}\partial_j-\partial_iT_{g^{ij}}\partial_j)v_{ext} 
\\
&+(T_{b_{ext}^j}\partial_j-T_{b^j}\partial_j)v_{ext}+(T_{\tilde{b}_{ext}^j}\partial_j-T_{\tilde{b}^j}\partial_j)\overline{v}_{ext}.
\end{split}
\end{equation*}
Making use of \Cref{parabil} and paradifferential calculus, we can easily estimate 
\begin{equation*}
\begin{split}
\|[\partial_iT_{g^{ij}}\partial_j+T_{b^j}\partial_j+T_{\tilde{b}^j}\partial_j,\chi_{>2R}]v\|_{l^1Y^{\sigma}}&\leq C(M,R)(\|\chi_{<4R}v\|_{l^1X^{\sigma}}+\|v\|_{l^1L_T^1H_x^{\sigma}})
\\
&\leq C(M,R)\|\chi_{<4R}v\|_{l^1X^{\sigma}}+\delta\|v\|_{l^1X^{\sigma}} 
\end{split}
\end{equation*}
for some small $\delta>0$. We note that in the last inequality, we used H\"older's inequality in $T$ and took $T$ sufficiently small depending on $R$ and $M$. Using the disjointness of the supports of $g_{ext}-g$ and $v_{ext}$, we obtain from the embedding $l^1L_T^1H_x^{\sigma}\subset l^1Y^{\sigma}$ and paradifferential calculus,
\begin{equation*}
\|(\partial_iT_{g^{ij}_{ext}}\partial_j-\partial_iT_{g^{ij}}\partial_j)v_{ext}\|_{l^1Y^{\sigma}}\lesssim_M \|v\|_{l^1L_T^1H_x^{\sigma}}\lesssim \delta\|v\|_{l^1X^{\sigma}}.    
\end{equation*}
We can similarly estimate the last two terms in the definition of $f_{ext}$. In light of this and the small data result \Cref{smalldata} which applies to the  equation for $v_{ext}$, we obtain  (\ref{interior reduction}). We have therefore reduced the estimate for $v$ in $l^1X^{\sigma}$ to obtaining the  bound
\begin{equation*}\label{l^1redbound}
\|\chi_{<4R}v\|_{l^1X^{\sigma}}\leq C(M,L)(\|v_0\|_{l^1H^{\sigma}}+\|f\|_{l^1Y^{\sigma}}).
\end{equation*}
However, this simply follows from (\ref{nol1estimate}) and the fact that the $l^1X^{\sigma}$ and $X^{\sigma}$ norms are equivalent within the set $B_{4R}(0)$ (with equivalence constant depending on $R$).
\end{proof}
\section{The \texorpdfstring{$L^2$}{} estimate for the linear flow}\label{Sec lin flow}
We begin our analysis by showing that we can close an estimate for the $L_T^{\infty}H_x^{\sigma}$ norm of a solution to the paradifferential linear equation (\ref{paralinflow}) up to a small error term in $X^{\sigma}$ as long as the time interval is small enough. Thanks to \Cref{hifreqlemma}, we may from here on harmlessly assume that
\begin{equation*}\label{highfreqassumption}
\supp(\widehat{v})\subset \{|\xi|>2^{k_1}\}
\end{equation*}
for some large parameter $k_1$ to be chosen. We will make this assumption for the rest of the section. The main estimate we wish to prove is the following.
\begin{proposition}[$L^2$ estimate for the paradifferential linear flow]\label{mainEEpara}\label{mainEE} Let $s_0$, $g^{ij}$, $b^j$ and $\tilde{b}^j$ be as in \Cref{linearestimate} with parameters $M$ and $L$. Let $\epsilon>0$. There is $T_0=T_0(\epsilon)>0$ such that for $0\leq T\leq T_0$, we have the  a priori bound for $v$ satisfying \eqref{paralinflow},
\begin{equation*}\label{paraL2est}
\|v\|_{L^{\infty}_TH^{\sigma}}\leq C(M,L)(\|v_0\|_{H^{\sigma}}+\|f\|_{Y^{\sigma}})+\epsilon\|v\|_{X^{\sigma}},
\end{equation*}
for every $\sigma\geq 0$.
\end{proposition} 
As noted earlier, by $C(M,L)$ we mean a constant which depends on $M$ and the trapping parameter $L$ within some fixed compact set (which is allowed to depend on $\epsilon$). The main obstruction to establishing \Cref{mainEEpara} is essentially the presence of the real part of the first order term $T_{\operatorname{Re}(b^j)}\partial_jv$. This is characterized somewhat by the following basic estimate for a truncated version of the linear flow in which the coefficient $b^j$ is purely imaginary. 
\begin{lemma}[Basic energy estimate]\label{basicEE} 
Let $g^{ij}$ be smooth, real and symmetric and let $b^j$ and $\tilde{b}^j$ be smooth functions. Assume that we have the size condition \eqref{asympflat}. Moreover, let $A(x,D)\in OPS^1$ be a time-independent pseudodifferential operator with symbol satisfying $Re(A)\geq 0$ and assume that $v$  solves the equation
\begin{equation}\label{idealequation}
i\partial_tv+\partial_iT_{g^{ij}}\partial_jv+i\operatorname{Im}(b^j)\partial_jv+\tilde{b}^j\partial_j\overline{v}+iA(x,D)v=f.
\end{equation}
Then for every $0<\delta\ll 1$ there is $T_0>0$ depending on $M$, $\delta$ and $A$ such that for $0<T\leq T_0$, we have the $L^2$ estimate,
\begin{equation*}\label{basicEEeqn}
\|v\|_{L_T^{\infty}L_x^2}^2\lesssim \|v_0\|_{L_x^2}^2+\|v\|_{X^{0}}\|f\|_{Y^{0}}+\delta\|v\|_{X^{0}}^2.
\end{equation*}
\end{lemma}
In the above lemma, we allow for the extra first order term $iA(x,D)v$. This will afford us  some flexibility when dealing with commutations of the principal operator $\partial_jT_{g^{ij}}\partial_i$ with various zeroth order Fourier multipliers and pseudodifferential operators later on when we deal with the full linear paradifferential flow.
\begin{proof}
We start with the basic energy identity:
\begin{equation*}
\begin{split}
\|v(t)\|_{L_x^2}^2+2\operatorname{Re}\langle A(x,D) v,v\rangle&=\|v_0\|_{L_x^2}^2+2\operatorname{Re}\langle i\partial_iT_{g^{ij}}\partial_jv,v\rangle-2\operatorname{Re}\langle \operatorname{Im}(b^j)\partial_jv,v\rangle+2\operatorname{Re}\langle i\tilde{b}^j\partial_j\overline{v},v\rangle
\\
&-2\operatorname{Re}\langle if,v\rangle,
\end{split}
\end{equation*}
which holds for each $0\leq t\leq T$. Here $\langle\cdot,\cdot\rangle$ denotes the inner product on $L^2_tL_x^2$. 
Unlike the operator $\partial_ig^{ij}\partial_j$, the paradifferential operator $\partial_iT_{g^{ij}}\partial_j$ is not quite self-adjoint. However, we do have the relation
\begin{equation*}
\operatorname{Re}\langle i\partial_iT_{g^{ij}}\partial_jv,v\rangle=\operatorname{Re}\langle i\partial_i(T_{g^{ij}}-g^{ij})\partial_jv,v\rangle.
\end{equation*}
By standard paradifferential calculus and the fact that $\|g^{ij}\|_{L_T^{\infty}C^{2,\alpha}}\leq C(M)$ for some $\alpha>0$, we have
\begin{equation*}
\|\partial_i(T_{g^{ij}}-g^{ij})\partial_jv\|_{L_x^2}\lesssim_M \|v\|_{L_x^2}.
\end{equation*}
Hence, by H\"older in $T$ and taking $T$ sufficiently small, we have
\begin{equation*}
2\operatorname{Re}\langle i\partial_iT_{g^{ij}}\partial_jv,v\rangle\lesssim_M T\|v\|_{L_T^{\infty}L_x^2}^2\leq\delta\|v\|_{X^{0}}^2.
\end{equation*}
Now, we turn to the other terms in the energy estimate. Integrating by parts and making use of Sobolev embeddings, we obtain the bound
\begin{equation*}\label{firstorderterm}
-2\operatorname{Re}\langle \operatorname{Im}(b^j)\partial_jv,v\rangle+2\operatorname{Re}\langle i\tilde{b}^j\partial_j\overline{v},v\rangle\lesssim MT\|v\|_{L_T^{\infty}L_x^2}^2\leq \delta\|v\|_{X^{0}}^{2},
\end{equation*}
if $T$ is small enough. Moreover, by the $Y^*=X$ duality, we have
\begin{equation*}
-2\operatorname{Re}\langle if,v\rangle\lesssim \|v\|_{X^{0}}\|f\|_{Y^{0}}.
\end{equation*}
Therefore, if $T$ is small enough, we arrive at the bound
\begin{equation*}
\|v\|_{L_T^{\infty}L_x^2}^2+\operatorname{Re}\langle A(x,D)v,v\rangle\lesssim \|v_0\|_{L_x^2}^2+\|v\|_{X^{0}}\|f\|_{Y^{0}}+\delta\|v\|_{X^0}^2.
\end{equation*}
Finally, by the sharp G\r{a}rding inequality \Cref{garding} and H\"older in time, we have
\begin{equation*}
\operatorname{Re}\langle A(x,D)v,v\rangle\gtrsim_A -T\|v\|_{L^{\infty}_TL_x^2}^2.
\end{equation*}
Taking $T$ sufficiently small concludes the proof.
\end{proof}
The remainder of this section will be essentially devoted to transforming the equation (\ref{paralinflow}) into an equation of the ideal form (\ref{idealequation}). Our primary means of doing this will be to construct a time-independent pseudodifferential renormalization operator $\mathcal{O}=Op(O)\in OPS^0$ which upon commuting $\mathcal{O}$ with the equation achieves this transformation within a compact ball $B_R(0)$. The hope is then to use the asymptotic smallness (\ref{smallness}) to control the residual error terms outside $B_{R}(0)$. Quite a bit of care is needed here to avoid a circular argument because the higher order symbol bounds for $O$ will grow in the parameter $R$, and so, at first glance, the operator bounds for $\mathcal{O}$ could counteract any smallness coming from the remaining error terms. Therefore, we will need to carefully track the dependence of the operator bounds for $\mathcal{O}$ on the parameters $R$ and $L$. In our construction, it will turn out that the $L^{\infty}$ norm of the symbol $O$ will have a $R$ independent bound (as $R\to\infty$). Therefore, for large enough $k_1$, the operator $\mathcal{O}S_{\geq k_1}$ will have $R$ independent $L^2\to L^2$, $X^0\to X^0$ and $Y^0\to Y^0$ bounds thanks to \Cref{CVvariant} and \Cref{LEopbounds}, respectively. This is how we will break the potential circularity.
\subsection{First order truncations} Since we want the symbol for $\mathcal{O}$ to be time-independent and smooth, our first aim will be to show that the first order paradifferential coefficients in (\ref{paralinflow}) can be replaced by smooth time-independent coefficients localized at a suitable frequency scale. To achieve this, let us fix another large parameter $k_0$ with $0\ll k_0\ll k_1$ to be chosen. We can rearrange the paradifferential equation as
\begin{equation}\label{paraeqn}
\begin{cases}
&i\partial_tv+\partial_jT_{g^{ij}}\partial_i v+b^j_{<k_0}(0)\partial_jv+\tilde{b}^j_{<k_0}(0)\partial_j\overline{v}=f+\mathcal{R}^1,
\\
&v(0)=v_0,
\end{cases}
\end{equation}
where
\begin{equation}\label{R1terms}
\mathcal{R}^1=(b^j_{<k_0}(0)\partial_jv-T_{b^j}\partial_jv)+(\tilde{b}^j_{<k_0}(0)\partial_j\overline{v}-T_{\tilde{b}^j}\partial_j\overline{v}).
\end{equation}
We have the following short lemma which shows that for large enough $k_0$, $k_1$ and small enough $T$, the error term $\mathcal{R}_1$ can be treated perturbatively.
\begin{lemma}\label{secondsource} For $k_0$ and $k_1$ sufficiently large and $T$ sufficiently small, we have
\begin{equation*}\label{secondsourceeqn}
\|\mathcal{R}^1\|_{Y^{\sigma}}\leq \epsilon \|v\|_{X^{\sigma}}.
\end{equation*}
\end{lemma}
\begin{proof}
We estimate the first term in (\ref{R1terms}) as the other term is essentially identical. By Bernstein's inequality, averaging in $T$ and the assumption (\ref{asympflat}), we have
\begin{equation*}
\|b^j_{<k_0}-b^j_{<k_0}(0)\|_{l^1X^{s_0-1}}\lesssim_{M}2^{2k_0}T.
\end{equation*}
Therefore, by the assumption $k_1\gg k_0$, \Cref{parabil} and taking $T$ small enough (depending on $k_0$ and $M$), we have
\begin{equation*}
\|(b^j_{<k_0}-b^j_{<k_0}(0))\partial_jv\|_{Y^{\sigma}}=\|T_{(b^j_{<k_0}-b^j_{<k_0}(0))}\partial_jv\|_{Y^{\sigma}}\leq \epsilon \|v\|_{X^{\sigma}}.
\end{equation*}
Next, using $k_1\gg k_0$, we can write
\begin{equation*}
\begin{split}
T_{b^j}\partial_jv-b^j_{<k_0}\partial_jv&=T_{S_{\geq k_0}b^j}\partial_jv.
\end{split}
\end{equation*}
So, from \Cref{parabil}, there is $\delta>0$ depending only on $s_0$ such that
\begin{equation*}
\begin{split}
\|T_{b^j}\partial_jv-b^j_{<k_0}\partial_jv\|_{Y^{\sigma}}&\lesssim \|S_{\geq k_0}b^j\|_{l^1X^{s_0-1-\delta}}\|v\|_{X^{\sigma}}
\\
&\lesssim_M 2^{-k_0\delta}\|v\|_{X^{\sigma}}. 
\end{split}
\end{equation*}
The above term can be controlled by $\epsilon\| v\|_{X^{\sigma}}$ by taking $k_0$ large enough. This completes the proof.
\end{proof}
\subsection{Commuting with derivatives}
The next step is to commute (\ref{paralinflow}) with $\langle\nabla\rangle^{\sigma}$. This will essentially reduce matters to proving an $L^2$ estimate for the paradifferential flow and get us one step closer to a situation in which we can apply \Cref{basicEE}. This would typically be a completely straightforward matter since the equation is already in paradifferential form; however, the commutation of the principal operator $\mathcal{P}$ with $\langle\nabla\rangle^{\sigma}$ will generate a further first order term which cannot be treated perturbatively in the large data regime. 
\\

To proceed, we define $u:=\langle\nabla \rangle^{\sigma}v$. We also compactify the notation for the principal and new first order terms by defining
\begin{equation*}\label{simplified notation}
\begin{split}
&\mathcal{P}:=\partial_jT_{g^{ij}}\partial_i,
\\
&\mathcal{B}:=b^j_{<k_0}(0)\partial_j-[\mathcal{P},\langle\nabla\rangle^{\sigma}]\langle\nabla\rangle^{-\sigma},
\\
&\tilde{\mathcal{B}}:=\tilde{b}^j_{<k_0}(0)\partial_j.
\end{split}
\end{equation*}
By commuting (\ref{paraeqn}) with $\langle\nabla\rangle^{\sigma}$, we obtain
\begin{equation*}
i\partial_tu+\mathcal{P}u+\mathcal{B}u+\tilde{\mathcal{B}}\overline{u}=\langle\nabla\rangle^{\sigma}f+\mathcal{R}_{\sigma}^1+\mathcal{R}_{\sigma}^2,
\end{equation*}
where $\mathcal{R}_{\sigma}^1:=\langle\nabla\rangle^{\sigma}\mathcal{R}^1$ and
\begin{equation*}\label{R2def}
\begin{split}
\mathcal{R}_{\sigma}^2&:=-[\langle\nabla\rangle^{\sigma},b^j_{<k_0}(0)]\partial_jv-[\langle\nabla\rangle^{\sigma},\tilde{b}^j_{<k_0}(0)]\partial_j\overline{v}.
\end{split}
\end{equation*}
Thanks to \Cref{secondsource}, we have a suitable estimate for $\mathcal{R}_{\sigma}^1$ in $Y^0$ which allows us to treat this term perturbatively. The following lemma shows that $\mathcal{R}_{\sigma}^2$ can be estimated na\"ively in $L_T^1L_x^2\subset Y^0$.
\begin{lemma}\label{Source terms from commutation}
For $T$ small enough, the source term $\mathcal{R}_{\sigma}^2$ satisfies the bound
\begin{equation}\label{R2est}
\|\mathcal{R}_{\sigma}^2\|_{L_T^1L_x^2}\leq\epsilon\|v\|_{X^{\sigma}}.
\end{equation}
\end{lemma}
\begin{proof}
Since $k_0\ll k_1$ and $\widehat{v}$ is supported at frequencies $\gtrsim 2^{k_1}$, we can write
\begin{equation*}
[\langle\nabla\rangle^{\sigma},b_{<k_0}^j(0)]\partial_jv=[\langle\nabla\rangle^{\sigma},T_{b^j_{<k_0}(0)}]\partial_jv.
\end{equation*}
Hence, by \Cref{CM}, Sobolev embedding and the regularity assumptions on $b^j$, we have
\begin{equation*}
\|[\langle\nabla\rangle^{\sigma},b^j_{<k_0}(0)]\partial_jv\|_{L_T^1L_x^2}\lesssim_{M,k_0} \|v\|_{L_T^1H_x^{\sigma}}\lesssim_M T\|v\|_{L_T^{\infty}H_x^{\sigma}}.
\end{equation*}
The other term in $\mathcal{R}_{\sigma}^2$ can be estimated similarly. Hence, by taking $T$ small enough, we obtain (\ref{R2est}), as desired.
\end{proof}
Next, we further frequency and time truncate the commutator in the term $\mathcal{B}$. As we will see later, such truncations will ensure  that our renormalization operator $\mathcal{O}$  belongs to $OPS^0$. Note that while we  cannot directly truncate  the principal operator $\mathcal{P}$ because it is second order, it is reasonable to expect that we can do this (as long as the truncation is sharp enough) for commutators involving $\mathcal{P}$, which are first order. We therefore define time and frequency truncated variants of $\mathcal{P}$, $\mathcal{B}$ and $\tilde{\mathcal{B}}$ (technically, this last term is unchanged) via
\begin{equation*}\label{truncated}
\begin{split}
&\mathcal{P}_{k_0}^0:=\partial_jg_{<k_0}^{ij}(0)\partial_i,
\\
&\mathcal{B}_{k_0}^0:=b^j_{<k_0}(0)\partial_j-[\mathcal{P}_{k_0}^{0},\langle\nabla\rangle^{\sigma}]\langle\nabla\rangle^{-\sigma},
\\
&\tilde{\mathcal{B}}_{k_0}^{0}:=\tilde{b}^j_{<k_0}(0)\partial_j,
\end{split}
\end{equation*}
and obtain the equation
\begin{equation}\label{preppedeqn}
i\partial_tu+\mathcal{P}u+\mathcal{B}_{k_0}^0u+\tilde{\mathcal{B}}_{k_0}^0\overline{u}=\langle\nabla\rangle^{\sigma}f+\mathcal{R}_{\sigma}^1+\mathcal{R}_{\sigma}^2+\mathcal{R}_{\sigma}^3,
\end{equation}
where
\begin{equation*}\label{R3def}
\mathcal{R}_{\sigma}^3:=(\mathcal{B}_{k_0}^0-\mathcal{B})u=-[\mathcal{P}_{k_0}^{0}-\mathcal{P},\langle\nabla\rangle^{\sigma}]v.
\end{equation*}
 The next lemma treats the new source term $\mathcal{R}_{\sigma}^3.$
\begin{lemma}\label{Rk3est} For $k_0$ and $k_1$ large enough and  $T$ small enough, we have
\begin{equation}\label{R3bound}
\|\mathcal{R}_{\sigma}^3\|_{Y^0}\leq \epsilon\|v\|_{X^{\sigma}}.
\end{equation}
\end{lemma}
\begin{proof}
We begin by writing
\begin{equation*}\label{Pdecomp}
\mathcal{P}_{k_0}^0-\mathcal{P}=(\partial_ig^{ij}_{<k_0}(0)-T_{\partial_ig^{ij}})\partial_j+(g^{ij}_{<k_0}(0)-T_{g^{ij}})\partial_i\partial_j.
\end{equation*}
As with the estimate for $\mathcal{R}_{\sigma}^2$, we have
\begin{equation*}
\|[\langle\nabla\rangle^{\sigma},(\partial_ig^{ij}_{<k_0}(0)-T_{\partial_ig^{ij}})]\partial_jv\|_{L_T^1L_x^2}\leq\epsilon\|v\|_{X^{\sigma}},
\end{equation*}
by taking $T$ small enough.  The term $[\langle\nabla\rangle^{\sigma},(T_{g^{ij}}-g^{ij}_{<k_0}(0))]\partial_i\partial_jv$ is more difficult to deal with since it is like an operator of order $\sigma+1$ applied to $v$, and therefore cannot be estimated in $L_T^1L_x^2$ without losing derivatives. Consequently, we must estimate it in the weaker space $Y^{0}$. Since $k_1\gg k_0$, we have the identity
\begin{equation*}
[\langle\nabla\rangle^{\sigma},(T_{g^{ij}}-g^{ij}_{<k_0}(0))]\partial_i\partial_jv=\sum_{k\geq 0}\tilde{S}_k[\langle\nabla\rangle^{\sigma},S_{<k-4}(g^{ij}-g^{ij}_{<k_0}(0))]\partial_i\partial_jS_kv    ,
\end{equation*}
where $\tilde{S}_k$ is a fattened Littlewood-Paley projection. Therefore, by almost orthogonality, \Cref{commutatorwithmultiplier} and \Cref{comremark} we have
\begin{equation*}
\begin{split}
\|[\langle\nabla\rangle^{\sigma},(T_{g^{ij}}-g^{ij}_{<k_0}(0))]\partial_i\partial_jv\|_{Y^0}&\lesssim \|g^{ij}-g^{ij}_{<k_0}(0)\|_{l^1X^{s_0-\delta}}\left(\sum_{k\geq 0}2^{2k(\sigma-1)}\|S_k\nabla v\|_{X^{0}}^2\right)^{\frac{1}{2}}
\\
&\lesssim \|g^{ij}-g^{ij}_{<k_0}(0)\|_{l^1X^{s_0-\delta}}\|v\|_{X^{\sigma}}
\end{split}
\end{equation*}
for some $\delta>0$. By taking $k_0$ large enough and then $T$ small enough, we can estimate using Bernstein type inequalities and the fundamental theorem of calculus,
\begin{equation*}
\|g^{ij}-g^{ij}_{<k_0}(0)\|_{l^1X^{s_0-\delta}}\leq\epsilon.   
\end{equation*}
Combining this with the above estimates concludes the proof of (\ref{R3bound}), as desired.
\end{proof}
To summarize what we have so far, $u:=\langle\nabla\rangle^{\sigma}v$ solves the equation
\begin{equation}\label{summary}
i\partial_tu+\mathcal{P}u+\mathcal{B}_{k_0}^0u+\tilde{\mathcal{B}}_{k_0}^0\overline{u}=\mathcal{R},
\end{equation}
where the source term $\mathcal{R}$ can be estimated in $Y^0$ by
\begin{equation*}
\|\mathcal{R}\|_{Y^0}\leq C\|f\|_{Y^{\sigma}}+\epsilon\|v\|_{X^{\sigma}},
\end{equation*}
for some universal constant $C$.
\subsection{Renormalization construction}
Now we are ready to construct the renormalization operator $\mathcal{O}$ whose role will be to transform (\ref{summary}) into an equation essentially of the form (\ref{idealequation}). As alluded to earlier, the main enemy we have to deal with is the first order term $\operatorname{Re}(\mathcal{B}_{k_0}^0)u$. The strategy will be to construct an operator with symbol in $S^0$ which conjugates away the ``worst part" of this term. As noted in \cite{kenig2005variable}, conjugating the entire term away would give a symbol that does not belong to $S^0$. We opt therefore to conjugate away only a portion of the first-order term whose principal part is supported within some large compact set $B_{R}(0)$. The hope is that the remaining error term will contribute errors of size $\approx\epsilon\|v\|_{X^{\sigma}}$ due to the smallness of the coefficients in (\ref{smallness}) outside of $B_R(0)$. As mentioned earlier, this does not come for free. The trade-off is that we will also need to control the $X^0\to X^0$, $Y^0\to Y^0$, and $L^2\to L^2$ norms of our renormalization operator to ensure that the smallness is retained when applying this operator (as the $\xi$ derivatives of its symbol will not have uniform in $R$ bounds). 
\medskip

The details of this construction will be given below. To set the stage, let us fix a large constant $R\gg 1$ to be chosen. We also define for each $\rho>0$, the function $\chi_{<\rho}(x):=\chi(\rho^{-1}x)$ where $\chi$ is a radial cutoff function equal to $1$ on the unit ball and vanishing outside $|x|>2$. As a first constraint, we demand for $R$ to be such that \eqref{smallness} holds with some $R_0<\frac{R}{8}$ and $\epsilon_0\ll \epsilon$. The bulk of the renormalization construction is given by the following proposition.
\begin{proposition}\label{Oconst}
Let $u$ be as above. Let $k_0$ be large enough so that $g^{ij}_{<k_0}(0)$ is a nontrapping metric with comparable parameters to $g^{ij}(0)$ (the existence of which is guaranteed by \Cref{perturbationstable}). Define the truncated symbol $a(x,\xi):=-g^{ij}_{<k_0}(0)\xi_i\xi_j$, which is the principal symbol for $\mathcal{P}_{k_0}^0$. Write also $iB(x,\xi):=i\operatorname{Re}(b^j_{<k_0}(0))\xi_j+i\{a,\langle\xi\rangle^{\sigma}\}\langle\xi\rangle^{-\sigma}$ to denote the principal symbol of $\operatorname{Re}\mathcal{B}_{k_0}^0$ and
\begin{equation*}
H_a:=\nabla_{\xi}a\cdot\nabla_x-\nabla_x a\cdot\nabla_{\xi}
\end{equation*}
to denote the Hamiltonian vector field for $a$. Let the parameters $R$, $M$ and $L$ be as above. Then there exists a smooth, non-negative, real-valued, time-independent symbol $O\in S^0$ with the following properties.
\begin{enumerate}
\item (Positive commutator with good error). There exists $r\in S^1$ such that if $T$ is sufficiently small,   
\begin{equation*}
H_aO+\chi_{<2R}B(x,\xi)O(x,\xi)+r(x,\xi)O(x,\xi)\geq 0,\hspace{5mm}\|Op(r)\|_{X^0\to Y^0}\lesssim_M\epsilon.
\end{equation*}
\item (Uniform $L^2$ bound at high frequency). For $k_0$, $k_1$ large enough and $T$ small enough depending on $R$, $M$ and $L$, $\mathcal{O}:=Op(O)$ satisfies the estimates
\begin{equation}\label{L^2bound}
\|\mathcal{O}u\|_{L^2}\approx \|u\|_{L^2},\hspace{5mm}\|\mathcal{O}u\|_{Y^0}\lesssim \|u\|_{Y^0},\hspace{5mm}\|\mathcal{O}u\|_{X^0}\lesssim \|u\|_{X^0},
\end{equation}
with implicit constants depending only on $M$ and on $L$ within a fixed compact set whose size is independent of $R$.
\item (Even in $\xi$ within $B_{\frac{R}{8}}(0)$). The symbol $s:=O(x,\xi)-O(x,-\xi)$ is supported in the region $|x|>\frac{R}{8}$ and for $k_1$ large enough, there holds
\begin{equation*}
\|Op(s)S_{\geq k_1}\|_{Y^0\to Y^0}\lesssim 1,
\end{equation*}
with implicit constants depending only on $M$ and on $L$ within a fixed compact set whose size is independent of $R$.
\end{enumerate}
\end{proposition}
The first property will allow us to transform (\ref{summary}) into an equation of the type (\ref{idealequation}) up to an error term supported outside $B_{R}(0)$ (plus an acceptable remainder). The second property ensures that the $L^2\to L^2$, $Y^0\to Y^0$ and $X^0\to X^0$ operator bounds for $\mathcal{O}$ do not depend on $R$, at least at high frequency. The third property ensures that $\mathcal{O}:=Op(O)$ commutes with complex conjugation to leading order (i.e.~within $B_{\frac{R}{8}}(0)$ where the coefficient $\tilde{b}^j$ can be large). The second and third properties will be important for avoiding the circularity mentioned earlier when trying to estimate the error terms supported outside $B_{\frac{R}{8}}(0)$.
\medskip

We also emphasize that $a$ is the principal symbol for the truncated operator $\mathcal{P}_{k_0}^0$ and not $\mathcal{P}$. This is to ensure that $O$ will be a classical (time-independent) $S^0$ symbol with bounds not depending on higher derivatives of $g^{ij}$ (however, they will depend on the frequency truncation scale $2^{k_0}$). The trade-off is that when commuting the equation for $u$ with $\mathcal{O}$, we will need to estimate an additional first order error term of the form
\begin{equation*}
[\mathcal{P}-\mathcal{P}_{k_0}^0,\mathcal{O}]u
\end{equation*}
in $Y^0$. It will turn out that this can be made small by taking $k_0, k_1$ large enough and $T$ small enough. We will discuss how to estimate this term later. For now, we start by proving \Cref{Oconst}.
\begin{proof}
 We make the ansatz $O(x,\xi)=e^{\psi(x,\xi)}$ where $\psi$ is some smooth real-valued function to be chosen. We begin by trying to enforce condition (i). For this, we recall that the vector field $H_a$ corresponds to differentiation along the Hamilton flow of $a$, which is given by (\ref{bicharacteristic}). That is,
\begin{equation*}
(H_a\psi)(x,\xi)=\frac{d}{dt}\psi(x^{t},\xi^{t})_{|t=0},
\end{equation*}
where $(x^{t},\xi^{t})$ are the bicharacteristics for $a$ with initial data $(x,\xi)$. We will perform our construction in two stages. That is, we will define two symbols $\psi_1$ and $\psi_2$ in $S^0$. The symbol $\psi_1$ will be chosen so that $H_a\psi_1$ cancels the bulk of the term $\chi_{<2R}B(x,\xi)$ but possibly with an additional error term which isn't small but has the redeeming feature that it is supported in the transition region $|x|\approx R$ where $g^{ij}_{<k_0}(0)$ is close to the corresponding flat metric. The second symbol $\psi_2$ will be chosen to correct $\psi_1$ so that the error term can be made sufficiently small. The full symbol $\psi$ will then be defined by $\psi:=\psi_1+\psi_2$. Inspired by the previous works \cite{craig1995microlocal, ichi1996remarks, hayashi1994remarks,MR2263709}, our starting point is to consider the ideal ``symbol"
\begin{equation*}
\psi_{ideal}(x,\xi):=-\frac{1}{2}\chi_{>1}(|\xi|)\int_{-\infty}^{0}B(x^{t}_{(x,\xi)},\xi^{t}_{(x,\xi)})+B(x^{t}_{(x,-\xi)},\xi^{t}_{(x,-\xi)})dt,
\end{equation*}
where $\chi_{>1}(|\xi|)$ is an increasing Fourier multiplier selecting frequencies $\geq 1$. We note that since $g^{ij}_{<k_0}(0)$ is nontrapping and $b^j,\nabla_xg^{ij}\in l^1X^{s_0-1}$, the integral in $\psi_{ideal}$ is well-defined. On a formal level, the commutator of the principal part of the equation with $Op(e^{\psi_{ideal}})$  conjugates away the leading part of the term $\operatorname{Re}\mathcal{B}_{k_0}^0u$, but as mentioned above, the symbol $\psi_{ideal}$ is not a classical $S^0$ symbol,  so it is not ideal to work with such a construction directly. In order to resolve this issue, we localize this symbol to the compact set $B_{2R}(0)$ by  instead defining
\begin{equation*}
\begin{split}
\psi_1(x,\xi):=-\frac{1}{2}\chi_{>1}(|\xi|)\chi_{<2R}(x)\int_{-\infty}^{0}(\chi_{<4R}B)(x^{t}_{(x,\xi)},\xi^{t}_{(x,\xi)})+(\chi_{<4R}B)(x^{t}_{(x,-\xi)},\xi^{t}_{(x,-\xi)})dt.
\end{split}
\end{equation*}
The corresponding pseudodifferential operator $Op(e^{\psi_1})$ will conjugate away the leading part of the first order term $\operatorname{Re}\mathcal{B}_{k_0}^0u$ within the ball $B_{2R}(0)$, which is the region where the $X^0\to Y^0$ operator bounds for $\mathcal{B}_{k_0}^0$ are expected to be large. The difficulty is then shifted to controlling the remaining errors in the exterior region, but now we have the benefit of $\psi_1$ being 
a genuine $S^0$ symbol (this fact will be confirmed below). We  remark that since $B(x,\xi)$ is real, $\psi_{1}$ is as well. Moreover, $\psi_1$ is even in $\xi$. 
\medskip

Since $B$ is odd in $\xi$, it is straightforward to verify that we have the leading order cancellation,  
\begin{equation}\label{firstsymb}
H_a\psi_1+\chi_{<2R}B(x,\xi)\geq -KR^{-1}|\chi'(\frac{1}{2}R^{-1}r)||\xi|-K\chi_{<2}(|\xi|),
\end{equation}
where $K>0$ is such that $K\gg_M \|\psi_{ideal}\|_{L^{\infty}}$. We remark that $K$ is uniformly bounded in $R$ because of \Cref{intalongbi}. The term on the right-hand side of \eqref{firstsymb} is not quite suitable for defining a symbol $r$ ensuring the bound in (i) (the corresponding operator need not have small $X^0\to Y^0$ bound due to the insufficient spatial decay in the first term). For this reason, we seek to further correct $\psi_1$ by a symbol $\psi_2$ which is supported in the region $|x|\gtrsim R$. Precisely, our aim will be to construct $\psi_2$ so that  
\begin{equation}\label{positivitybound1}
H_a\psi_2-KR^{-1}|\chi'(\frac{1}{2}R^{-1}r)||\xi|-K\chi_{<2}(|\xi|)+r(x,\xi)\geq 0
\end{equation}
where $r\in S^1$ is a suitable remainder term satisfying the bound in (i). Before proceeding, to simplify the notation somewhat, for the remainder of the proof we will write $A:=A(x)$ as a shorthand for $g^{ij}_{<k_0}(0)$ and $A_{\infty}$ as a shorthand for $g_{\infty}^{ij}$. We also define the functions $\theta(x,\xi):=\angle (x,A_{\infty}\xi)$, $\alpha(x,\xi):=\angle(x,A\xi)$, $\beta(\xi):=\angle(A\xi,A_{\infty}\xi)$ and $\gamma(x,\xi):=\frac{1}{2}(1+\cos(\theta))$. 
\medskip

Now, to proceed, we begin by recalling that the assumption (\ref{smallness}) ensures that we have the bounds
\begin{equation}\label{Aasymptotic}
|A-A_{\infty}|+|\nabla A|\ll \epsilon,\hspace{5mm} |x|>\frac{R}{8}.
\end{equation}
In particular, $A$ is close to the flat metric in $L^{\infty}$ when $|x|>\frac{R}{8}$. Now, let:
\begin{enumerate}
\item $\rho$ be a smooth, increasing function such that $\rho'\approx 1$ for $\frac{1}{7}\leq r\leq 2$, $\rho=0$ for $r\leq \frac{1}{8}$ and $\rho=1$ for $r\geq 3$. Define $\rho_R(x)=\rho(R^{-1}r)$ and $\rho_{\theta}(x,\xi)=\rho_R(x\gamma)$.
\item For $c\in [-1,1]$ and some fixed positive $\delta_0\ll 1$, let $\varphi_{<c}$ be a decreasing smooth function which vanishes for $x>c+\delta_0$ and is identically one for $x\leq c$. Define also $\varphi_{>c}:=1-\varphi_{<c}$. 
\end{enumerate}
We then define the symbol $\psi_2$ by
\begin{equation}\label{psi2def}
\begin{split}
\psi_2(x,\xi)&:=K'\chi_{>1}(|\xi|)\left(\rho_R\varphi_{<-\frac{1}{2}}(\cos(\theta))-\rho_{\theta}\varphi_{>-\frac{1}{2}}(\cos(\theta))\right)
\end{split}
\end{equation}
where $K'\gg K$ is a constant to be chosen.  We note that the weight $\rho_R$ is increasing in the direction of the bicharacteristics in the regions of phase space where they are outgoing with respect to the flat metric. In such regions, this will give a good bound from below for the bulk of $H_a\psi_2$. The purpose of $\rho_{\theta}$ will be to accomplish the same task in the incoming region as well as the regions of phase space where $A_\infty\xi$ is nearly orthogonal to $x$. In such regions, a purely radially increasing cutoff (such as $\rho_R$) would be insufficient. The reason we use the average $\frac{1}{2}(1+\cos(\theta))$ in the definition of $\rho_{\theta}$ is to ensure that $\rho_{\theta}$ still vanishes for a suitable range of $r$ on the support of $\varphi_{>-\frac{1}{2}}(\cos(\theta))$ ($r<\frac{R}{8}$, say). This, in particular,  ensures that the pointwise error between $A$ and $A_{\infty}$ is small on the support of $\rho_{\theta}\varphi_{>-\frac{1}{2}}(\cos(\theta))$. To verify that $\psi_2$ has the required properties, we first make note of the following simple algebraic computation.
\begin{lemma}\label{algebraic computation}
For $r>\frac{R}{8}$, we have
\begin{equation*}
A\xi\cdot \nabla_x \cos(\theta)=|A\xi|\left(\frac{\sin^2(\theta)}{r}+\delta(x,\xi)\right),
\end{equation*}
where $\delta(x,\xi)$ is an error term with  $|\delta(x,\xi)|\ll \frac{1}{r}$.
\end{lemma}
\begin{proof}
This is a simple computation. We have
\begin{equation}\label{algebraiccomp}
\begin{split}
A\xi\cdot\nabla_x \cos(\theta)&=\frac{|A\xi|}{r}(\cos(\beta)-\cos(\alpha)\cos(\theta))
\\
&=\frac{|A\xi|}{r}\sin^2(\theta)+\frac{|A\xi|}{r}((\cos(\beta)-1)+\cos(\theta)(\cos(\theta)-\cos(\alpha))).
\end{split}
\end{equation}
By non-degeneracy of $A$ and $A_{\infty}$ and (\ref{Aasymptotic}), we have 
\begin{equation*}
|\cos(\alpha)-\cos(\theta)|+|\cos(\beta)-1|\ll 1,\hspace{5mm}r\geq\frac{R}{8}.
\end{equation*}
Taking $\delta$ to be the coefficient of $|A\xi|$ in the second term in the second line of (\ref{algebraiccomp}) concludes the proof.
\end{proof}
Now, we compute the Hamilton vector field applied to $\psi_2$. We define the remainder symbol $r\in S^1$ by
\begin{equation}\label{rdefinition}
r(x,\xi):=-\xi_i\xi_j\nabla_{\xi}\psi_2\cdot \nabla_x A^{ij}+K''\chi_{<2}(|\xi|),
\end{equation}
where $K''\gg K'$ is some sufficiently large constant. We note that $r$ essentially consists of the part of $H_a\psi_2$ in which $\psi_2$ is differentiated in $\xi$. This is expected to contribute a small $X^0\to Y^0$ operator norm because its principal part includes a factor of $\nabla_x A$ which is small in $l^1X^{s_0-1}$ when $|x|>\frac{R}{8}$. The subprincipal terms will contribute small $L_T^1L_x^2\to L_T^1L_x^2$ operator bounds by taking $T$ to be sufficiently small. We then have
\begin{equation}\label{Com step}
H_{a}\psi_2+r(x,\xi)\geq -2A\xi\cdot\nabla_x\psi_2+K''\chi_{<2}(|\xi|).
\end{equation}
 We now expand the first term on the right-hand side of \eqref{Com step} to obtain
\begin{equation*}\label{commutatoridentity}
\begin{split}
-A\xi\cdot\nabla_x\psi_2=-\frac{K'}{R}\chi_{>1}(|\xi|)|A\xi|\bigg(\cos(\alpha)\rho'(R^{-1}r)&\varphi_{<-\frac{1}{2}}(\cos(\theta))+\frac{R}{r}\left(\sin^2(\theta)+r\delta\right)\rho(R^{-1}r)\varphi'_{<-\frac{1}{2}}(\cos(\theta))\bigg)
\\
+\frac{K'}{R}\chi_{>1}(|\xi|)|A\xi|\bigg(\frac{1}{2}(\cos(\alpha)+\cos(&\beta))\rho'(R^{-1}r\gamma)\varphi_{>-\frac{1}{2}}(\cos(\theta))
\\
+\frac{R}{r}\big(\sin^2(\theta)+&r\delta\big)\rho(R^{-1}r\gamma)\varphi'_{> -\frac{1}{2}}(\cos(\theta))\bigg),
\end{split}
\end{equation*}
where $\alpha$ and $\beta$ are as in \Cref{algebraic computation}. If $\epsilon_0$ is small enough in (\ref{smallness}), we observe that on the support of $\rho'(R^{-1}r)\varphi_{<-\frac{1}{2}}(\cos(\theta))$, we have $\cos(\alpha)<-\frac{1}{3}$. Additionally, $(\sin^2(\theta)+r\delta)$ is non-negative on the support of $\rho(R^{-1}r)\varphi'_{<-\frac{1}{2}}(\cos(\theta))$ and $\rho(R^{-1}r\gamma)\varphi'_{>-\frac{1}{2}}(\cos(\theta))$.  Moreover, on the support of $\rho'(R^{-1}r\gamma)\varphi_{>-\frac{1}{2}}(\cos(\theta))$, we have $(\cos(\alpha)+\cos(\beta))>\frac{1}{3}$. 
\medskip

 By non-degeneracy of $A$, we can choose $K'$ depending only on $g$ so that 
\begin{equation*}
K'|A\xi|\gg K|\xi|.
\end{equation*}
Combining the above, we can arrange for
\begin{equation*}
-A\xi\cdot \nabla_x\psi_2\geq KR^{-1}|\chi'(\frac{r}{2R})||\xi|-\frac{K''}{2}\chi_{<2}(|\xi|),
\end{equation*}
where $K''$ is as in (\ref{rdefinition}). We then define the full symbol $\psi$ by
\begin{equation*}
\psi:=\psi_1+\psi_2.
\end{equation*}
It is left to verify the properties (i), (ii) and (iii) in \Cref{Oconst}. The positive commutator bound
\begin{equation*}
H_aO+\chi_{<2R}B(x,\xi)O(x,\xi)+r(x,\xi)O(x,\xi)\geq 0
\end{equation*}
follows easily from the chain rule and the above construction if $K'$ is large enough. Next, we verify that $r\in S^1$ and that $r$ has the operator bound
\begin{equation}\label{rbound}
\|Op(r)\|_{X^0\to Y^0}\leq\epsilon.    
\end{equation}
The fact that $r\in S^1$ is clear so we turn our attention to (\ref{rbound}). Using the definition of $r$, we can write
\begin{equation*}
Op(r)=(\chi_{>\frac{R}{8}}\nabla_xA^{ij})\cdot Op(\xi_i\xi_j\nabla_{\xi}\psi_2)+K''\chi_{<2}(|D|). 
\end{equation*}
Using the embedding $L_T^1L_x^2\subset Y^0$ and  that $\xi_i\xi_j\nabla_{\xi}\psi_2\in S^1$, we can estimate using simple paradifferential calculus and \Cref{symbcalc},
\begin{equation*}
\|(\chi_{>\frac{R}{8}}\nabla_xA^{ij}-T_{\chi_{>\frac{R}{8}}\nabla_xA^{ij}})\cdot Op(\xi_i\xi_j\nabla_{\xi}\psi_2)\|_{L_T^1L_x^2\to Y^0}\lesssim_{M,k_0} 1  .
\end{equation*}
Therefore, by H\"older's inequality in $T$, we have for $T$ small enough (depending on $M$ and $k_0$), 
\begin{equation*}
\|(\chi_{>\frac{R}{8}}\nabla_xA^{ij}-T_{\chi_{>\frac{R}{8}}\nabla_xA^{ij}})\cdot Op(\xi_i\xi_j\nabla_{\xi}\psi_2)\|_{X^0\to Y^0}\leq \epsilon. 
\end{equation*}
Hence, we now only need to show that the $X^0\to Y^0$ norm for $T_{\chi_{>\frac{R}{8}}\nabla_xA^{ij}}\cdot Op(\xi_i\xi_j\nabla_{\xi}\psi_2)$ can be made small. For this, let us define $\tilde{r}\in S^0$ by
\begin{equation*}
\tilde{r}(x,\xi)=\langle\xi\rangle^{-1}\xi_i\xi_j\nabla_{\xi}\psi_2.
\end{equation*}
By \Cref{symbcalc}, one can verify that the operator $\langle\nabla\rangle Op(\tilde{r})-Op(\xi_i\xi_j\nabla_{\xi}\psi_2)$ is bounded from $L_T^1L_x^2\to L_T^1L_x^2$ with norm depending only on $M$ and $k_0$. Therefore, by taking $T$ small, we can make the $X^0\to Y^0$ bound of this operator smaller than $\epsilon$. From \Cref{parabil} and the smallness assumption \eqref{smallness}, we then have
\begin{equation*}
\begin{split}
\|T_{\chi_{>\frac{R}{8}}\nabla_xA^{ij}}\cdot Op(\xi_i\xi_j\nabla_{\xi}\psi_2)\|_{X^0\to Y^0}&\lesssim \|\chi_{>\frac{R}{8}}\nabla_xA^{ij}\|_{l^1X^{s_0-1}}\|\langle\nabla\rangle\|_{X^0\to X^{-1}}\|Op(\tilde{r})\|_{X^0\to X^0}+\epsilon
\\
&\lesssim \epsilon \|Op(\tilde{r})\|_{X^0\to X^0}+\epsilon
\\
&\lesssim_{M} \epsilon.
\end{split}
\end{equation*}
Clearly, the $X^0\to Y^0$ bound for the remaining subprincipal term $K''\chi_{<2}(|D|)$ can be made small by taking $T$ small. This concludes the proof of (i). Now, we turn to (ii) and (iii). For this, we need the following lemma involving symbol bounds for $O$.
\begin{lemma}\label{Osymbounds}
The symbol $O$ constructed above satisfies the following bounds.
\begin{enumerate}
\item ($R$ independent $L^{\infty}$ bound). There is a constant $C_0$ depending only on the profile of $g(0)$ and on $M$ but not on $R$ or $k_0$ such that
\begin{equation*}
\|O(x,\xi)\|_{L^{\infty}_{x,\xi}}\lesssim C_0.
\end{equation*}
\item (Higher order symbol bounds). For every $|\alpha+\beta|\geq 2$, there is a constant $C_{\alpha,\beta}$ depending on $M$, $L(R)$, $R$ and $k_0$ such that
\begin{equation*}
\|\langle\xi\rangle^{|\alpha|}\partial^{\alpha}_{\xi}\partial^{\beta}_xO(x,\xi)\|_{L^{\infty}_{x,\xi}}\leq C_{\alpha,\beta}.
\end{equation*}
If $|\alpha+\beta|=1$, the constant can be taken to be uniform in $k_0$.
\end{enumerate}
\end{lemma}
The crucial thing to note here is that only the higher order symbol bounds for $O$ depend on $R$ and $k_0$ while the $L^{\infty}$ bound does not. 
\begin{proof} 
Clearly, it suffices to show each of the above symbol bounds for $\psi_1$ and $\psi_2$. Given the requisite bounds for $\psi_1$, the bounds for $\psi_2$ are clear. Therefore, we focus on $\psi_1$. We begin with the $L^{\infty}$ bound. By homogeneity of the bicharacteristic flow, it further suffices to show that
\begin{equation}\label{CKSbound}
\int_{\mathbb{R}}|B(x^{t},\xi^{t})|dt\lesssim_M C_0.
\end{equation}
By homogeneity and a change of variables, we have
\begin{equation*}
\int_{\mathbb{R}}|B(x^{t},\xi^{t})|dt=\int_{\mathbb{R}}|\xi|^{-1}|B(x^{t}_{\omega},|\xi|\xi^{t}_{\omega})|dt,
\end{equation*}
where $(x^t_{\omega},\xi^t_{\omega})$ denote the bicharacteristics with data $(x,\omega):=(x,\xi|\xi|^{-1})$. Then, we use \Cref{xibond} and the definition of the symbol $B$ to obtain 
\begin{equation*}
\begin{split}
|B(x^{t}_{\omega},|\xi|\xi^{t}_{\omega})|&\leq C_0|\xi||(b^j_{<k_0}(0))(x^{t}_{\omega})|+C_0|\xi||(\nabla_xA)(x^{t}_{\omega})|.
\end{split}
\end{equation*}
The estimate (\ref{CKSbound}) then follows (after possibly relabelling $C_0$) from \Cref{intalongbi}, using the fact that $b^{j}_{<k_0}(0), \nabla_xA\in l^1H^{s_0-1}$ with norm $\lesssim_M 1$. This yields the $L^{\infty}$ bound for $\psi_1$. The higher order symbol bounds follow immediately from \Cref{bicharhigher} and repeated applications of the chain rule.
\end{proof}
Now, we return to the proof of (\ref{L^2bound}). From the above lemma and \Cref{CVvariant}, we have the $L^2$ bound,
\begin{equation}\label{L^2bound1}
\|\mathcal{O}u\|_{L^2}\lesssim C_0\|u\|_{L^2}
\end{equation}
for $k_1$ large enough, with universal implicit constant. We next aim to establish the bound
\begin{equation}\label{inversebound}
\|u\|_{L^2}\lesssim C_0\|\mathcal{O}u\|_{L^2}.
\end{equation}
Using \Cref{symbcalc}, we see that $Op(e^{-\psi})$ is an approximate inverse for $Op(e^{\psi})$ in the sense that we have
\begin{equation*}
Op(e^{-\psi})Op(e^{\psi})=1+Op(q),
\end{equation*}
where $q\in S^{-1}$ with symbol bounds depending only on the symbol bounds for $\psi$. Therefore, we have
\begin{equation*}
u=S_{\geq k_1-4}u=Op(e^{-\psi})S_{\geq k_1-4}\mathcal{O}u+Op(\tilde{q})u ,   
\end{equation*}
where $\tilde{q}\in S^{-1}$ with uniform in $k_1$ symbol bounds. Hence, from \Cref{Sobolevbound} we obtain
\begin{equation*}
\|u\|_{L^2_x}\lesssim C_0\|\mathcal{O}u\|_{L^2_x}+C_1\|u\|_{H^{-1}_x},
\end{equation*}
where $C_0$ depends only on $M$ and $g(0)$ and $C_1$ depends on a finite collection of semi-norms $|O|_{S^0}^{(j)}$. Since $\|u\|_{H^{-1}_x}\lesssim 2^{-k_1}\|u\|_{L^2_x}$, we can take $k_1$ large enough so that
\begin{equation*}
\|u\|_{L^2_x}\lesssim C_0\|\mathcal{O}u\|_{L^2}.
\end{equation*}
This gives (\ref{inversebound}). The $Y^0\to Y^0$ and $X^0\to X^0$ bounds for $\mathcal{O}$ follow from \Cref{LEopbounds}. This establishes property (ii) of \Cref{Oconst}. The proof of property (iii) follows almost identical reasoning to the proof of (ii), using the fact that $\psi$ is even in $\xi$ for $|x|<\frac{R}{8}$. This completes the proof of \Cref{Oconst}. 
\end{proof}
\subsection{Proof of \Cref{mainEE}}
Now, we  complete the proof of \Cref{mainEE}. We will slightly abuse notation from here on and write $\lesssim_M$ to mean that the implicit constant in the corresponding estimate depends on $M$ and $C_0$ as above (but not on $R$). Moreover, we let $\mathcal{R}$  generically denote an error term such that
\begin{equation}\label{remainderform}
\|\mathcal{R}\|_{Y^{0}}\leq C(M,L)(\|f\|_{Y^{\sigma}}+\|v_0\|_{H^{\sigma}})+\epsilon\|v\|_{X^{\sigma}}.
\end{equation}
We apply $\mathcal{O}:=Op(O)$ from \Cref{Oconst} to equation \eqref{preppedeqn}. Writing $w:=\mathcal{O}u$, we obtain 
\begin{equation*}
\begin{split}
&i\partial_tw+\mathcal{P}w+\mathcal{O}\mathcal{B}_{k_0}^{0}u+[\mathcal{O},\mathcal{P}]u+\mathcal{O}\tilde{\mathcal{B}}_{k_0}^0\overline{u}=\mathcal{R},
\end{split}
\end{equation*}
where by the $Y^0\to Y^0$ bound for $\mathcal{O}$ (see (ii) in \Cref{Oconst}), $\mathcal{R}$ still satisfies the estimate (\ref{remainderform}) as long as $k_1$ is large enough. Performing similar frequency and time truncations as before and commuting $\mathcal{O}$ with the first order terms, we obtain 
\begin{equation*}
\begin{split}
&i\partial_tw+\mathcal{P}w+i\operatorname{Im}(\mathcal{B}_{k_0}^{0})w+\tilde{\mathcal{B}}_{k_0}^0\overline{w}+[\mathcal{O},\mathcal{P}_{k_0}^0]u+\chi_{<2R}\operatorname{Re}(\mathcal{B}_{k_0}^{0})\mathcal{O}u=\tilde{\mathcal{R}},
\end{split}
\end{equation*}
where
\begin{equation}\label{moreremainder}
\begin{split}
\tilde{\mathcal{R}}&=\mathcal{R}-[\mathcal{O},\tilde{\mathcal{B}}_{k_0}^0]\overline{u}-[\mathcal{O},\mathcal{B}_{k_0}^0]u-\chi_{\geq 2R}\operatorname{Re}(\mathcal{B}_{k_0}^0)\mathcal{O}u+\tilde{B}_{k_0}^0(\overline{\mathcal{O}u}-\mathcal{O}\overline{u})+[\mathcal{O},(\mathcal{P}_{k_0}^0-\mathcal{P})]u.
\end{split}
\end{equation}
We next estimate $\tilde{\mathcal{R}}$. To begin, note that the second and third terms in (\ref{moreremainder}) are zeroth order and can be estimated in $L_T^1L_x^2\subset Y^0$, so that
\begin{equation*}
\|[\mathcal{O},\tilde{\mathcal{B}}_{k_0}^0]\overline{u}+[\mathcal{O},\mathcal{B}_{k_0}^0]u\|_{Y^0}\lesssim_{M,L}T\|v\|_{X^{\sigma}},
\end{equation*}
which by taking $T$ small can be controlled by $\epsilon \|v\|_{X^{\sigma}}$. To get a suitable error estimate for the fourth term in (\ref{moreremainder}), we first note that by property (ii) in \Cref{Oconst}, we have
\begin{equation*}
\|\mathcal{O}u\|_{X^0}\lesssim C_0\|u\|_{X^0}  
\end{equation*}
if $k_1$ is large enough. Here, we recall crucially that $C_0$ is a $R$ independent constant. Therefore, it suffices to establish the bound
\begin{equation}\label{fourthtermbound}
\begin{split}
\|\chi_{\geq 2R}\operatorname{Re}(\mathcal{B}_{k_0}^0)\|_{X^0\to Y^0}\leq\epsilon. 
\end{split}
\end{equation}
Clearly, it suffices to work with the principal part of $\chi_{\geq 2R}\operatorname{Re}(\mathcal{B}_{k_0}^0)$ as the error term is bounded from $L_T^1L^2_x\to L_T^1L^2_x$. We can expand the principal part as
\begin{equation*}
\begin{split}
\chi_{\geq 2R}\operatorname{Re}(b_{<k_0}(0))m_1(D)+\chi_{\geq 2R}\nabla_x Am_2(D) 
\end{split}
\end{equation*}
 where $m_1, m_2\in S^1$ are suitable (matrix-valued) Fourier multipliers with symbol bounds independent of $M$, $L$ and $R$. We can replace the coefficients of $m_1$ and $m_2$ above with  the paradifferential operators $T_{\chi_{\geq 2R}\operatorname{Re}(b_{<k_0}(0))}$ and $T_{\chi_{\geq 2R}\nabla A}$, as the error is an operator which maps $L_T^1L_x^2$ to $L_T^1L_x^2$ with norm depending only on $M$ and $k_0$. Therefore, if $T$ is small enough, such an error term can be discarded. Using \Cref{parabil} and the asymptotic smallness (\ref{smallness}), the remaining term satisfies
\begin{equation*}
\|T_{\chi_{\geq 2R}\operatorname{Re}(b_{<k_0}(0))}m_1(D)+T_{\chi_{\geq 2R}\nabla A}m_2(D)\|_{X^0\to Y^0}\lesssim_M \epsilon (\|m_1(D)\|_{X^{0}\to X^{-1}}+\|m_2(D)\|_{X^{0}\to X^{-1}})\lesssim\epsilon.    
\end{equation*}
To deal with the fifth term in (\ref{moreremainder}), we do a similar analysis. Using the definition of $\tilde{B}_{k_0}^0$ and property (iii) in \Cref{Oconst}, we can write
 \begin{equation*}
\tilde{B}_{k_0}^0(\overline{\mathcal{O}u}-\mathcal{O}\overline{u})=\chi_{>\frac{R}{8}}\tilde{b}^j_{<k_0}(0)\partial_j(\overline{\mathcal{O}u}-\mathcal{O}\overline{u})     .
 \end{equation*}
 If $\widehat{u}$ is supported at high enough frequency, we can estimate using (ii) in \Cref{Oconst},
 \begin{equation*}
 \|\partial_j(\overline{\mathcal{O}u}-\mathcal{O}\overline{u})\|_{X^{-1}}\lesssim_M \|u\|_{X^0}.    
 \end{equation*}
 Combining this with the smallness
 \begin{equation*}
 \|\chi_{>\frac{R}{8}}\tilde{b}_{<k_0}(0)\|_{l^1X^{s_0-1}}\leq\epsilon   , 
 \end{equation*}
 we can argue as with the previous term to obtain
 \begin{equation*}
 \|\tilde{B}_{k_0}^0(\overline{\mathcal{O}u}-\mathcal{O}\overline{u})\|_{Y^0}\lesssim_M \epsilon\|v\|_{X^{\sigma}}.    
 \end{equation*}
Now, we turn to the most tricky part, which is estimating the last term in (\ref{moreremainder}). For this, we have the following lemma.
\begin{lemma}[Commutator bound]\label{trickycomlemma} For $k_0$ large enough and $T$ sufficiently small, there holds
\begin{equation}\label{trickycommutator}
\|[\mathcal{O},(\mathcal{P}_{k_0}^0-\mathcal{P})]u\|_{Y^0}\leq \epsilon\|v\|_{X^{\sigma}}.
\end{equation}
\end{lemma}
\begin{proof}
Clearly, we can write
\begin{equation}\label{threeterms}
\begin{split}
[\mathcal{O},(\mathcal{P}-\mathcal{P}_{k_0}^0)]&=[\mathcal{O},(T_{\partial_ig^{ij}}-\partial_i g^{ij}_{<k_0}(0))\partial_j]+(T_{g^{ij}}-g^{ij}_{<k_0}(0))[\mathcal{O},\partial_i\partial_j]
\\
&+[\mathcal{O},(T_{g^{ij}}-g^{ij}_{<k_0}(0))]\partial_i\partial_j.
\end{split}
\end{equation}
The first term is zeroth order and is bounded from $L_T^1L_x^2\to L_T^1L_x^2$. Indeed, by taking $T$ small enough, it is a straightforward consequence of \Cref{symbcalc} and \Cref{CM} that
\begin{equation*}
\|[\mathcal{O},(T_{\partial_ig^{ij}}-\partial_i g^{ij}_{<k_0}(0))\partial_j]\|_{L_T^{\infty}L_x^2\to Y^0}\leq \epsilon.
\end{equation*}
The second and third terms are first order, and, as usual, must be dealt with carefully to extract the necessary smallness. We start with the second term which is a bit easier. Since $(T_{g^{ij}_{<k_0}(0)}-g_{<k_0}^{ij}(0))[\mathcal{O},\partial_i\partial_j]$ is bounded from $L_T^1L_x^2\to L_T^1L_x^2$, we can replace $T_{g^{ij}}-g_{<k_0}^{ij}(0)$ with $T_{g^{ij}-g_{<k_0}^{ij}(0)}$. Then by \Cref{parabil} and Bernstein inequalities, we have
\begin{equation*}
\begin{split}
\|T_{g^{ij}-g_{<k_0}^{ij}(0)}[\mathcal{O},\partial_i\partial_j]u\|_{Y^0}&\lesssim (\|S_{\geq k_0}g^{ij}\|_{l^1X^{s_0-1-\delta}}+\|g^{ij}_{<k_0}-g^{ij}_{<k_0}(0)\|_{l^1X^{s_0-1-\delta}})\|\langle\nabla\rangle^{-1}[\mathcal{O},\partial_i\partial_j]u\|_{X^{0}}
\\
&\lesssim_M 2^{-(1+\delta)k_0}\|\langle\nabla\rangle^{-1}[\mathcal{O},\partial_i\partial_j]u\|_{X^{0}},
\end{split}
\end{equation*}
for some $\delta>0$. As $\langle\nabla\rangle^{-1}[\mathcal{O},\partial_i\partial_j]\in OPS^0$, it suffices to consider its principal part when estimating the last term. This is because the subprincipal part is (crudely) bounded from $L_T^{\infty}H_x^{-\frac{1}{2}}$ to $L_T^{\infty}H_x^{\frac{1}{2}}\subset X^0$, so we can control such terms by using the fact that $u$ is localized to frequencies $\gtrsim 2^{k_1}$ to gain a smallness factor $2^{-\frac{k_1}{2}}$, and then take $k_1$ sufficiently large. To estimate the principal symbol $c_p$ for $\langle\nabla\rangle^{-1}[\mathcal{O},\partial_i\partial_j]$, we can use that $\nabla g^{ij}(0), b^j\in C^{1,\delta}$ and \Cref{bicharhigher} to obtain the bound
\begin{equation*}
\|c_p\|_{L^{\infty}}\lesssim \|\nabla_xO\|_{L^{\infty}}\lesssim_{M,R,L} 1,
\end{equation*}
with implicit constant independent of $k_0$. Therefore, by taking $k_0$ and $k_1$ large enough and applying \Cref{LEopbounds}, we obtain 
\begin{equation*}
2^{-(1+\delta)k_0}\|Op(c_p)u\|_{X^{0}}\lesssim 2^{-(1+\delta)k_0}\|c_p\|_{L^{\infty}}\|u\|_{X^0}\leq \epsilon \|v\|_{X^{\sigma}}.
\end{equation*}
Consequently, the second term in (\ref{threeterms}) can be estimated by
\begin{equation*}
\|(T_{g^{ij}}-g^{ij}_{<k_0}(0))[\mathcal{O},\partial_i\partial_j]u\|_{Y^0}\leq \epsilon\|v\|_{X^{\sigma}}.
\end{equation*}
It remains to estimate the third term in (\ref{threeterms}) which is the most delicate because the commutator itself involves the metric $g^{ij}$ at high frequencies. Our aim as above is to show that 
\begin{equation*}
\|[\mathcal{O},(T_{g^{ij}}-g^{ij}_{<k_0}(0))]\partial_i\partial_j u\|_{Y^0}\leq\epsilon\|v\|_{X^{\sigma}}.
\end{equation*}
Intuitively, this should be possible by taking $k_0$ large enough. There are, however, two complications in dealing with this. Firstly, the symbol bounds for $O$ depend on $k_0$. Secondly, the coefficient in the paradifferential operator $T_{g^{ij}}$ has limited regularity, so the standard pseudodifferential calculus cannot be directly applied. Our strategy is to split this term into three parts to separate the issues. We write
\begin{equation}\label{twoparts}
\begin{split}
[\mathcal{O},(T_{g^{ij}}-g^{ij}_{<k_0}(0))]\partial_i\partial_ju&=[\mathcal{O},(g^{ij}_{<m}-g^{ij}_{<k_0})]\partial_i\partial_ju
\\
&+[\mathcal{O},(T_{g^{ij}}-g^{ij}_{<m})]\partial_i\partial_ju
\\
&+[\mathcal{O},(g^{ij}_{<k_0}-g^{ij}_{<k_0}(0))]\partial_i\partial_ju,
\end{split}
\end{equation}
where $m$ is some universal parameter with $k_0\ll m\ll k_1$. For the first term, we do not need to worry about the presence of any functions of limited regularity, but we still need to worry about the dependence of $O$ on $k_0$. For the second term, by taking $m$ large enough, the $k_0$ dependence in $O$ should be a non-issue, which puts us in a position to use \Cref{commutatorestimate}. Control of the final term follows by taking $T\ll 2^{-2k_0}$ and averaging in $T$.
\medskip

Let us begin by analyzing the first term. The principal symbol $c_{p}$ for $[\mathcal{O},(g^{ij}_{<m}-g^{ij}_{<k_0})]\partial_i $ is given by 
\begin{equation*}
c_p=\{O,(g^{ij}_{<m}-g^{ij}_{<k_0})\}\xi_i.
\end{equation*}
Analogously to the principal part for the second term in (\ref{threeterms}), we have the bound
\begin{equation*}
\|\langle\xi\rangle\nabla_{\xi}O\|_{L^{\infty}}\lesssim_{M,R,L} 1.
\end{equation*}
To estimate the full commutator, we then use \Cref{symbcalc} to write
\begin{equation*}
[\mathcal{O},(g^{ij}_{<m}-g^{ij}_{<k_0})]\partial_i=\nabla_x (g^{ij}_{<m}-g^{ij}_{<k_0})\cdot Op(\nabla_{\xi}O\xi_i)+Op(r)
\end{equation*}
where $r\in S^0$ (with symbol bounds depending on $m$). Arguing as in the estimate for the second term in (\ref{threeterms}), it follows by using  \Cref{parabil}, then \Cref{LEopbounds}, then  taking $k_0$ large enough and $T$ small enough (depending on $m$, $M$, $R$ and $L$) that
\begin{equation*}
\|[\mathcal{O},(g^{ij}_{<m}-g^{ij}_{<k_0})]\partial_i\partial_ju\|_{Y^0}\leq \epsilon \|v\|_{X^{\sigma}}.
\end{equation*}
This takes care of the first term in (\ref{twoparts}). For the second term, it suffices to estimate $[\mathcal{O},T_{g^{ij}-g^{ij}_{<m}}]\partial_i\partial_ju$, as the error will be bounded from $L_T^1L_x^2\to L_T^1L_x^2$. To estimate this term, we simply use \Cref{commutatorestimate} to obtain
\begin{equation*}
\|[\mathcal{O},T_{g^{ij}-g^{ij}_{<m}}]\partial_i\partial_ju\|_{Y^0}\lesssim_{M,L,R,k_0}\|g^{ij}_{<m}-g^{ij}\|_{l^1X^{s_0-\delta}}\|u\|_{X^{0}}.
\end{equation*}
We then recall the smallness bound
\begin{equation*}
\|g^{ij}_{<m}-g^{ij}\|_{l^1X^{s_0-\delta}}\lesssim_M 2^{-\delta m},
\end{equation*}
which tells us that if $m$ is large enough relative to $k_0$, $R$, $L$ and $M$ then we have the estimate
\begin{equation*}
\|[\mathcal{O},T_{g^{ij}-g^{ij}_{<m}}]\partial_i\partial_ju\|_{Y^0}\leq\epsilon\|u\|_{X^0}\lesssim \epsilon \|v\|_{X^{\sigma}}.
\end{equation*}
 Finally, by averaging in $T$ and arguing similarly to the above, the last term in (\ref{twoparts}) can be controlled by $\epsilon\|v\|_{X^{\sigma}}$ by taking $T$ small enough. This completes the proof of \Cref{trickycomlemma}.
\end{proof}
Using the above lemma and \Cref{Oconst}, we now arrive at the following equation for $w$:
\begin{equation*}
\begin{split}
&i\partial_tw+\mathcal{P}w+i\operatorname{Im}(\mathcal{B}_{k_0}^{0})w+\tilde{\mathcal{B}}_{k_0}^0\overline{w}+[\mathcal{O},\mathcal{P}_{k_0}^0]u+\chi_{<2R}\operatorname{Re}(\mathcal{B}_{k_0}^{0})\mathcal{O}u=\mathcal{R},
\end{split}
\end{equation*}
where $\mathcal{R}$ is as in \eqref{remainderform}. To conclude, we make one final reduction. From \Cref{symbcalc}, $\mathcal{O}^{-1}:=Op(e^{-\psi})$ is an approximate inverse for $\mathcal{O}$ in the sense that we have $Op(\mathcal{O})Op(\mathcal{O}^{-1})=1+Op(q)$ for $q\in S^{-1}$. Therefore,  by estimating the error term generated by $Op(q)$ in $L_T^1L_x^2$, we can write
\begin{equation*}
\begin{split}
&i\partial_tw+\mathcal{P}w+i\operatorname{Im}(\mathcal{B}_{k_0}^{0})w+\tilde{\mathcal{B}}_{k_0}^0\overline{w}+iA(x,D)\mathcal{O}^{-1}w=\mathcal{R},
\end{split}
\end{equation*}
where
\begin{equation*}
A(x,D):=-i[\mathcal{O},\mathcal{P}_{k_0}^0]-i\chi_{<2R}\operatorname{Re}(\mathcal{B}_{k_0}^0)\mathcal{O}+Op(r)\mathcal{O},
\end{equation*}
and  $\mathcal{R}$ is again of the form (\ref{moreremainder}) (as long as $T$ is small enough).  By construction, $A(x,D)\mathcal{O}^{-1}$ is a time-independent pseudodifferential operator of order $1$ with non-negative principal symbol in $S^1$. Therefore, the above equation for $w$ is now in the form (\ref{idealequation}) with a source term $\mathcal{R}$ satisfying (\ref{remainderform}). Hence, \Cref{mainEEpara} easily follows by applying \Cref{basicEE}. 
\section{The local energy decay estimate}\label{LED sec} 
In this section, we complement the $L^2$ estimate in the previous section with an estimate for the local energy component of the norm $\|v\|_{X^{\sigma}}$ for a solution $v$ to (\ref{paralinflow}). For every $\sigma\geq 0$, we denote the local energy component of $X^{\sigma}$ by 
\begin{equation*}
\|v\|_{\mathcal{X}^{\sigma}}=\left(\sum_{j\geq 0}2^{2j(\sigma+\frac{1}{2})}\|S_ju\|_{X}^2\right)^{\frac{1}{2}}.    
\end{equation*}
We remark that we have the obvious embedding $\|v\|_{\mathcal{X}^{\sigma}}\lesssim \|v\|_{L_T^2H_x^{\sigma+\frac{1}{2}}}$.
\subsection{The local energy estimate}
The local energy estimate we will need for (\ref{paralinflow}) is given by the following proposition.
\begin{proposition}\label{LEreduction}
Let $\sigma\geq 0$ and let $s_0$, $g^{ij}$, $b^j$ and $\tilde{b}^j$ be as in \Cref{linearestimate} with parameters $M$ and $L$. Suppose that $v$ solves \eqref{paralinflow} and let $\epsilon>0$. There is $T_0=T_0(\epsilon)>0$ such that for $0\leq T\leq T_0$, we have the  local energy bound 
\begin{equation}\label{LEestimate_ref}
\|v\|_{\mathcal{X}^{\sigma}}\leq C(M,L)(\|v\|_{L_T^{\infty}H_x^{\sigma}}+\|f\|_{Y^{\sigma}})+\epsilon\|v\|_{X^{\sigma}},
\end{equation}
where $C(M,L)$ depends on $M$ and on the parameter $L$ within some fixed compact set depending on $\epsilon$. 
\end{proposition}
Fix $\delta>0$ to be some small parameter to be chosen. From (\ref{summary}) in the previous section, we can choose $k_0$ sufficiently large and $T$ sufficiently small so that $u:=\langle\nabla\rangle^{\sigma}v$ solves the equation
\begin{equation*}\label{preppedeqn2}
i\partial_tu+\mathcal{P}u+\mathcal{B}_{k_0}^0u+\tilde{\mathcal{B}}_{k_0}^0\overline{u}=\mathcal{R},
\end{equation*}
with the remainder estimate
\begin{equation*}\label{gbound}
\|\mathcal{R}\|_{Y^0}\lesssim \|f\|_{Y^{\sigma}}+\delta\|v\|_{X^{\sigma}}.
\end{equation*}
Also, as in the previous section, we may assume that $u$ is localized to frequencies $\gtrsim 2^{k_1}$, where $k_1$ is some sufficiently large parameter to be chosen. Unlike with the $L^2$ estimate, however, we will not need the added energy structure coming from the complex-conjugate first order term. It is therefore convenient to write the equation as a system in $u$ and $\overline{u}$. In doing this, we obtain the following compact form of the paradifferential linear equation:
\begin{equation*}\label{systemform2}
\partial_t\bold{u}+\textbf{P}\bold{u}+\textbf{B}_{k_0}^0\bold{u}=\bold{R},
\end{equation*}
where
\begin{equation}\label{Refchange2}
\textbf{P}:=
i
\begin{pmatrix}
-\mathcal{P} & 0 \\
0 & \mathcal{P}
\end{pmatrix},\hspace{5mm}\textbf{B}_{k_0}^0:=i\begin{pmatrix}-\mathcal{B}_{k_0}^0 & -\mathcal{\tilde{B}}_{k_0}^0 \\
\hspace{2mm}\overline{\mathcal{\tilde{B}}_{k_0}^0} & \hspace{2mm}\overline{\mathcal{B}_{k_0}^0}
\end{pmatrix},\hspace{5mm}\bold{u}:=\begin{pmatrix}
u \\
\overline{u}
\end{pmatrix},
\end{equation}
and $\bold{R}$ is a source term satisfying the bound
\begin{equation}\label{sourceR1bound}
\|\bold{R}\|_{Y^0}\lesssim \|f\|_{Y^{\sigma}}+\delta\|v\|_{X^{\sigma}}.
\end{equation}
 We define analogously to before the truncated principal operator $\bold{P}_{k_0}^0$ by replacing the nonzero entries in $\bold{P}$ with $\mathcal{P}_{k_0}^0$ in the natural way.  
 \medskip
 
By using \Cref{smalldata} and arguing similarly to the proof of \Cref{l1reduction}, for each $R>0$ large enough and $T$ small enough, there holds
\begin{equation*}
\|\chi_{\geq R}\bold{u}\|_{\mathcal{X}^{0}}\leq C(M,R)(\|v\|_{L_T^{\infty}H_x^{\sigma}}+\|\bold{R}\|_{Y^{0}}+\|\chi_{<2R}\bold{u}\|_{L_T^2H_x^{\frac{1}{2}}})+\epsilon\|v\|_{X^{\sigma}}.
\end{equation*}
Therefore, to prove \eqref{LEestimate_ref},
it suffices to establish the bound 
\begin{equation}\label{inbound}
\|\chi_{<2R}\bold{u}\|_{L_T^2H_x^{\frac{1}{2}}}\leq C(M,L)(\|v\|_{L_T^{\infty}H_x^{\sigma}}+\|\bold{R}\|_{Y^{0}})+\epsilon\|v\|_{X^{\sigma}}.
\end{equation}
This latter estimate is where we will concentrate the bulk of our efforts in this section. 
\subsection{Interior estimate}
Now we turn to establishing the required interior estimate (\ref{inbound}). The main construction we will need is given by the following result, which can very loosely be thought of as a spatially truncated version of Doi's construction in \cite{ichi1996remarks}. Our method will work under far less stringent decay assumptions, however. For similar reasons to the previous section, we will again work with the principal symbol for the truncated operator $\bold{P}_{k_0}^0$ in our analysis rather than $\bold{P}$ directly (at the cost of estimating a term with the same flavor as (\ref{trickycommutator})). We will also write $|\bold{B}_{k_0}^0|$ to denote the maximum of the absolute values of the entries of the principal symbol for $\bold{B}_{k_0}^0$.
\begin{proposition}\label{compactregion} Let $C(M)>1$ be a constant depending on $M$ to be chosen. Moreover, let $k_0$ be large enough so that $g^{ij}_{<k_0}(0)$ is nontrapping with comparable parameters to $g^{ij}(0)$ (which is possible by \Cref{perturbationstable}). Define $a:=-g^{ij}_{<k_0}(0)\xi_i\xi_j$. Then for every $R'\gg R$ sufficiently large,  there is a smooth, non-negative, time-independent $S^0$ symbol $q\geq 1$ with the following properties: 
\begin{enumerate}
\item (Positive commutator in $B_{R'}(0)$ with small error). There exists $r\in S^1$ such that if $R'$ and $k_1$ are large enough and $T$ is sufficiently small relative to $R'$ and $k_1$, then we have 
\begin{equation*}
H_aq+C(M)rq\gtrsim C(M) \chi_{<R'}|\bold{B}_{k_0}^0|q,\hspace{5mm}\|Op(rq)S_{\geq k_1}\|_{X^0\to Y^0}\lesssim\frac{\epsilon}{C(M)}.
\end{equation*}
\item (Ellipticity in $B_{2R}(0)$). 
\begin{equation*} 
H_aq+C(M)rq\geq C(M)\chi_{<2R}|\xi|q,
\end{equation*}
where $r$ is as in (i).
\item (Zeroth order symbol bound). There is a constant $C_0(M,R)$ depending on $M$ and $R$ but not on $R'$ such that
\begin{equation*}
|q|\leq C_0.
\end{equation*}
\item (First order symbol bound). There is a constant $C_1(M,R,R')$ depending on $M$, $R$ and $R'$ such that
\begin{equation*}
|\xi||\nabla_{\xi}q|+|\nabla_xq|\leq C_1.    
\end{equation*}
\item (Higher order symbol bounds). There is a constant $C_2(M,R,R',k_0)$ depending on $M$, $R$, $R'$ and $k_0$ such that 
\begin{equation*}
\langle\xi\rangle^{|\alpha|}|\partial_{\xi}^{\alpha}\partial_x^{\beta}q|\lesssim_{\alpha,\beta} C_2,\hspace{5mm}|\alpha+\beta|\geq 2.
\end{equation*}
\end{enumerate}
\end{proposition}
In the $R$ and $R'$ dependent constants above, we also allow for dependence on $L$ within $B_{R}(0)$ and $B_{R'}(0)$, respectively. The first property will allow us to control the contribution of the first-order terms in the equation within the larger compact set $B_{R'}(0)$, up to a small error term, as long as $\bold{u}$ is localized at high enough frequency. The second property will give us the required control of $\chi_{<2R}\bold{u}$ in $L_T^2H_x^{\frac{1}{2}}$ up to a suitable  error term. We importantly remark  that the zeroth order symbol bounds in (iii) for $q$ 
 depend only on $M$ and $R$ (more precisely, $L(R)$). This will ensure that the $Y^0\to Y^0$ bound for $Op(q)$  depends only on $M$ and $R$ as long as $\bold{u}$ is at sufficiently high frequency, thanks to \Cref{LEopbounds}. As a consequence, we may argue similarly to the previous section and treat the first-order terms in the region outside of $B_{R'}(0)$ perturbatively as long as $R'$ is large enough relative to $R$. We note that unlike in the construction in \Cref{Oconst}, this second parameter $R'$ is needed because the uniform norm of the symbol $q$ necessarily depends on the smaller radius $R$. 
 \medskip
 
 We emphasize that the first order symbol bounds in (iv) depend on $M$, $R$ and $R'$ but not on $k_0$. The purpose of this will be to control an error term that is similar to the commutator (\ref{trickycommutator}) from the previous section by taking $k_0$ large relative to $M$, $R$, and $R'$. The higher order symbol bounds in  property (v) will ensure that $q$ is a classical $S^0$ symbol and will allow us to estimate lower order error terms in $L_T^1L_x^2$ by taking $T$ small depending on $M$, $R$, $R'$ and $k_0$, similarly to the previous section.
\begin{proof}
We begin by defining a smooth function that will be suitable for controlling the size of the first order coefficients within the larger compact set $B_{R'}(0)$. A reasonable choice is the following:
\begin{equation*}
\eta_{R'}=\chi_{<2R'}\sqrt{|\tilde{b}_{<k_0}(0)|^2+|b_{<k_0}(0)|^2+|\nabla_x g_{<k_0}^{ij}(0)|^2+L(2R')^{-2}}    .
\end{equation*}
The term $L(2R')^{-2}$ is for technical convenience. It  ensures that $\eta_{R'}$ is smooth and allows us to invoke \Cref{intalongbi}  to obtain uniform integrability along the bicharacteristic flow for the truncated metric $g^{ij}_{<k_0}(0)$ with a bound independent of $R'$. Precisely, we have 
\begin{equation}\label{etaintalongbi}
\int_{\mathbb{R}}\eta_{R'}(x^{t},\xi^{t})|\xi^t|dt\leq C_0
\end{equation}
where $C_0$ is as above. Moreover, for $|x|\leq R'$, we clearly have $|\nabla_xg^{ij}_{<k_0}(0)||\xi|+|\bold{B}_{k_0}^0|\lesssim \eta_{R'} |\xi|$. Now, we move to constructing the symbol $q$. We start by defining a preliminary symbol $p_1$ via 
\begin{equation*}
p_1(x,\xi):=-\chi_{>1}(|\xi|)\chi_{<R'}\int_{0}^{\infty}(\chi_{<2R}+\eta_{R'})(x^{t},\xi^t)|\xi^{t}| dt,
\end{equation*}
where similarly to the construction for $O$ in \Cref{Oconst}, we localized the symbol in space to $B_{R'}(0)$ so that it will ultimately belong to $S^0$. As in \Cref{Oconst}, $H_ap_1$ will generate an error term coming from the localization $\chi_{<R'}$. To deal with this, we correct $p_1$ by another symbol $p_2$. To define $p_2$, we take our cue from the definition  (\ref{psi2def}) in the previous section. Using the same notation as in (\ref{psi2def}) with the parameter $R'$ replacing $R$ in all instances, we define
\begin{equation*}
p_2(x,\xi):=K'\chi_{>1}(|\xi|)\left(\rho_{R'}\varphi_{<-\frac{1}{2}}(\cos(\theta))-\rho_{\theta}\varphi_{>-\frac{1}{2}}(\cos(\theta))\right)    ,
\end{equation*}
where $K':=K'(R,M)$ is a constant such that 
\begin{equation}
K'\gg \sup_{(x,\xi)\in\mathbb{R}^{2d}}\chi_{>1}(|\xi|)\int_{0}^{\infty}(\chi_{<2R}+\eta_{R'})(x^{t},\xi^t)|\xi^{t}| dt.    
\end{equation}
We note that thanks to the nontrapping assumption and (\ref{etaintalongbi}), $K'$ can be chosen to depend only on $R$ and $M$, but not on $R'$. We then define $p:=p_1+p_2$ and analogously to (\ref{rdefinition}), we define the remainder symbol $r$ by
\begin{equation*}
r(x,\xi):=-\xi_i\xi_j\nabla_{\xi}p_2\cdot \nabla_x g^{ij}_{<k_0}(0)+K''\chi_{<2}(|\xi|)    ,
\end{equation*}
where $K''\gg K'$ is some sufficiently large constant. We then define the required symbol $q$ by
\begin{equation}\label{qdef}
q:=e^{C(M)p},
\end{equation}
for some sufficiently large constant $C(M)>0$. Now, we turn to establishing each property in \Cref{compactregion}. First, arguing similarly to the proof of the first property in \Cref{Oconst}, we compute directly that
\begin{equation}\label{commutatorqcomp}
\begin{split}
H_aq+C(M)rq\geq C(M)(\chi_{<2R}+\chi_{<R'}\eta_{R'})|\xi|q.
\end{split}
\end{equation}
From this, we immediately obtain the positive commutator bounds in (i) and (ii) in \Cref{compactregion}. The $X^0\to Y^0$ estimate for $Op(rq)S_{\geq k_1}$ follows from properties (iii)-(v) (to be established below), \Cref{LEopbounds} and  the fact that $\|\chi_{>R'}\nabla_xg^{ij}_{<k_0}(0)\|_{L^{\infty}}\to 0$ as $R'\to \infty$. Next, we verify the symbol bounds (iii)-(v). It clearly suffices to establish the analogous symbol bounds for the symbol $p_1$. To do this, we split
\begin{equation*}
p_1:=-\chi_{>1}(|\xi|)\chi_{<R'}\left(\int_{0}^{\infty}\eta_{R'}(x^{t},\xi^{t})|\xi^{t}| dt+\int_{0}^{\infty}\chi_{<2R}(x^{t},\xi^t)|\xi^t|dt\right)=:\chi_{>1}(|\xi|)\chi_{<R'}(a_1+a_2).    
\end{equation*}
By a change of variables and homogeneity, we have for each $\xi\neq 0$,
\begin{equation*}
a_i(x,\xi)=a_i\left(x,\frac{\xi}{|\xi|}\right),\hspace{5mm} i=1,2.
\end{equation*}
By (\ref{etaintalongbi}) and \Cref{xibond}, we then easily verify property (iii) for $\chi_{>1}(|\xi|) \chi_{<R'} a_1$. By the nontrapping assumption, one may verify (iii) for the symbol $\chi_{>1}(|\xi|) \chi_{<R'} a_2$ as well. Properties (iv) and (v) for both $\chi_{>1}(|\xi|)\chi_{<R'}a_1$ and $\chi_{>1}(|\xi|)\chi_{<R'}a_2$ are a straightforward consequence of \Cref{bicharhigher}. This completes the proof of \Cref{compactregion}.
\end{proof}
Now, we turn to establishing the main estimate (\ref{inbound}). We begin by defining the symbols $\bold{q}$ and $|\bold{q}|$:
\begin{equation*}
\bold{q}:=\begin{pmatrix}
q & 0 \\
0 & -q
\end{pmatrix},\hspace{5mm} |\bold{q}|:=q\bold{I}_{2\times 2}.
\end{equation*}
Define $\bold{Q}:=\frac{1}{2}Op(\bold{q})+\frac{1}{2}Op(\bold{q})^*$. Performing a similar calculation to \Cref{basicEE}, we note that $\bold{P}$ is skew-adjoint up to a $L_x^2\to L_x^2$ bounded error. Therefore, it is a straightforward algebraic manipulation to verify the inequality
\begin{equation*}
\operatorname{Re}\langle \bold{Q}\bold{P}\bold{u},\bold{u}\rangle\geq \frac{1}{2}\operatorname{Re}\langle  [\bold{Q},\bold{P}]\bold{u},\bold{u}\rangle-C_2\|v\|_{L_T^{\infty}H_x^{\sigma}}^2,
\end{equation*}
where $C_2$ is as in \Cref{compactregion}.
We then obtain the basic preliminary energy estimate, 
\begin{equation}\label{poscom}
\begin{split}
&\frac{1}{2}\operatorname{Re}\langle \bold{Q}\bold{u},\bold{u}\rangle(T)+ \int_{0}^{T}\operatorname{Re}\langle (\frac{1}{2}[\bold{Q},\bold{P}_{k_0}^0]+\bold{Q}\bold{B}_{k_0}^0)\bold{u},\bold{u}\rangle dt
\\
\leq &\frac{1}{2}\operatorname{Re}\langle \bold{Q}\bold{u},\bold{u}\rangle(0)+\frac{1}{2}\int_{0}^{T}\operatorname{Re}\langle [\bold{Q},(\bold{P}_{k_0}^0-\bold{P})]\bold{u},\bold{u}\rangle dt+\int_{0}^{T}\operatorname{Re}\langle 
\bold{Q}\bold{R},\bold{u}\rangle dt+C_2\|\bold{u}\|_{L_T^{\infty}L_x^2}^2.
\end{split}
\end{equation}
Now, we estimate each term in (\ref{poscom}). By Cauchy-Schwarz and \Cref{Sobolevbound}, we have
\begin{equation*}
|\operatorname{Re}\langle \bold{Q}\bold{u},\bold{u}\rangle(T)|+|\operatorname{Re}\langle \bold{Q}\bold{u},\bold{u}\rangle(0)|\leq C_2\|\bold{u}\|_{L_T^{\infty}L_x^2}^2. 
\end{equation*}
Next, by \Cref{compactregion}, the principal symbol $\bold{c}(x,\xi)$ of $[\bold{Q},\bold{P}_{k_0}^0]$ satisfies,
\begin{equation*}
\bold{c}(x,\xi)+C(M)rq\bold{I}_{2\times 2}
\geq \frac{1}{2}C(M)(\chi_{<2R}|\xi|q+\chi_{<R'}|\bold{B}_{k_0}^0|q)\bold{I}_{2\times 2}.
\end{equation*}
We can therefore choose $C(M)$ large enough so that
\begin{equation*}
\bold{c}(x,\xi)+C(M)rq\bold{I}_{2\times 2}-\chi_{<2R}|\xi|q\bold{I}_{2\times 2}-\frac{1}{2}C(M)\chi_{<R'}q|\bold{B}_{k_0}^0|\bold{I}_{2\times 2}\geq \bold{0}.
\end{equation*}
Then, the classical G\r{a}rding inequality \Cref{garding} along with its matrix version (see \Cref{matrixgard}) yields
\begin{equation}\label{gardingapplicationLE}
\begin{split}
\int_{0}^{T}\operatorname{Re}\langle (\frac{1}{2}[\bold{Q},\bold{P}_{k_0}^0]+\bold{Q}\bold{B}_{k_0}^0)\bold{u},\bold{u}\rangle dt&\gtrsim \|\chi_{<2R}\bold{u}\|_{L_T^2H_x^{\frac{1}{2}}}^2-C_2\|\bold{u}\|_{L_T^{\infty}L_x^2}^2-C(M)\|Op(rq)\bold{u}\|_{Y^0}\|\bold{u}\|_{X^0}
\\
&-\|\bold{Q}\chi_{>R'}\bold{B}_{k_0}^0\bold{u}\|_{Y^0}\|\bold{u}\|_{X^0},
\end{split}
\end{equation}
where we applied H\"older's inequality in $T$ to control the lower order error term in \Cref{matrixgard} by the $L_T^{\infty}L_x^2$ norm of $\bold{u}$ and the $Y^*=X$ duality to control the remaining first order terms. To control the first $Y^0$ error term on the right, we use property (i) from \Cref{compactregion} to estimate
\begin{equation*}
C(M)\|Op(rq)\bold{u}\|_{Y^0}\|\bold{u}\|_{X^0}\lesssim_M \epsilon\|v\|_{X^{\sigma}}^2,
\end{equation*}
which holds as long as $\bold{u}$ is localized at high enough frequency (i.e.~$k_1$ is large enough). To control the latter $Y^0$ error term, we first note that by \Cref{symbcalc}, the embedding $L_T^1L_x^2\subset Y^0$ and H\"older in $T$,  we have
\begin{equation*}
\|\bold{Q}\chi_{>R'}\bold{B}_{k_0}^0\bold{u}\|_{Y^0}\leq \|\chi_{>R'}\bold{B}_{k_0}^0\bold{Q}\bold{u}\|_{Y^0}+C_2\|\bold{u}\|_{L_T^{\infty}L_x^2}.   
\end{equation*}
Then by using \Cref{parabil} and arguing as with the analogous terms in the previous section, we have
\begin{equation*}
\|\chi_{>R'}\bold{B}_{k_0}^0\bold{Q}\bold{u}\|_{Y^0}\lesssim \|\chi_{>R'}(\tilde{b}_{<k_0}(0),b_{<k_0}(0),\nabla_xg^{ij}_{<k_0}(0))\|_{l^1X^{s_0-1}}\|\bold{Q}S_{>k_1-4}\|_{X^0\to X^0}\|\bold{u}\|_{X^0}+C_2\|\bold{u}\|_{L_T^{\infty}L_x^2}.    
\end{equation*}
Using the fact that the $L^{\infty}$ norm of the symbol $\bold{q}$ depends only on $R$ and $M$ and not on $R'$, we can take $k_1$ and $R'$ large enough so that \Cref{LEopbounds} and (\ref{smallness}) ensure that
\begin{equation*}
\|\chi_{>R'}(\tilde{b}_{<k_0}(0),b_{<k_0}(0),\nabla_xg^{ij}_{<k_0}(0))\|_{l^1X^{s_0-1}}\|\bold{Q}S_{>k_1-4}\|_{X^0\to X^0}\leq\epsilon.    
\end{equation*}
It  then follows by Cauchy-Schwarz and (\ref{gardingapplicationLE}) that we have
\begin{equation*}
\int_{0}^{T}\operatorname{Re}\langle (\frac{1}{2}[\bold{Q},\bold{P}_{k_0}^0]+\bold{Q}\bold{B}_{k_0}^0)\bold{u},\bold{u}\rangle dt\gtrsim \|\chi_{<2R}\bold{u}\|_{L_T^2H_x^{\frac{1}{2}}}^2-C_2\|\bold{u}\|_{L_T^{\infty}L_x^2}^2-\epsilon\|v\|_{X^{\sigma}}^2.    
\end{equation*}
Next, we estimate the contribution of the second term in the second line of (\ref{poscom}). The procedure here is essentially identical to the estimate in \eqref{trickycommutator}. Using the symbol bounds for $q$ in \Cref{compactregion} (specifically, that the derivatives of $q$ up to first order have uniform in $k_0$ bounds), we can estimate by taking $k_0$ large enough and $T$ small enough as in the proof of \Cref{trickycomlemma} to obtain
\begin{equation*}\label{trickycommutator2}
\|[\bold{Q},(\bold{P}_{k_0}^0-\bold{P}_k)]\bold{u}\|_{Y^0}\leq\epsilon \|v\|_{X^{\sigma}}.
\end{equation*} 
To estimate the third term in the second line of \eqref{poscom}, we use the $Y^*=X$ duality and \Cref{LEopbounds} to obtain
\begin{equation*}
\begin{split}
\int_{0}^{T}\operatorname{Re}\langle \bold{Q}\bold{R},\bold{u}\rangle dt &\leq C_0\|\bold{R}\|_{Y^0}\|v\|_{X^{\sigma}},
\end{split}
\end{equation*}
where the constant $C_0$ depends only on $M$ and $R$ if $k_1$ is large enough. Taking $T$ small enough in (\ref{sourceR1bound}) and using Cauchy-Schwarz, we  have 
\begin{equation*}\label{desiredsourcebound}
\int_{0}^{T}\operatorname{Re}\langle \bold{Q}\bold{R},\bold{u}\rangle dt\leq C_0\|f\|_{Y^{\sigma}}^2+\epsilon^2 \|v\|_{X^{\sigma}}^2.
\end{equation*}
Putting the above estimates together, we obtain
\begin{equation*}
\|\chi_{<2R}\bold{u}\|_{L_T^2H_x^{\frac{1}{2}}}\leq C_2(\|v\|_{L_T^{\infty}H_x^{\sigma}}+\|f\|_{Y^{\sigma}})+\epsilon\|v\|_{X^{\sigma}}.
\end{equation*}
This establishes (\ref{inbound}),
which completes the proof of 
\Cref{LEreduction}.
\section{Proof of the main linear estimate}\label{proofoflinest}
In this short section, we complete the proof of \Cref{linearestimate} by combining \Cref{mainEE} and \Cref{LEreduction}. First, note that by  \Cref{hifreqlemma}, \Cref{parasourcefirst} and \Cref{l1reduction}, it suffices to establish for small enough $T$, the bound
\begin{equation}\label{reducedbound}
\|v\|_{X^{\sigma}}\leq C(M,L)(\|v_0\|_{H^{\sigma}}+\|f\|_{Y^{\sigma}}),\hspace{5mm}\sigma\geq 0,    
\end{equation}
when $v$ is a solution to (\ref{paralinflow}) with $\widehat{v}$ supported at frequencies $\gtrsim 2^{k_1}$ for some arbitrarily large (but fixed) parameter $k_1$. Let $\epsilon>0$ be a small positive constant to be chosen. By \Cref{mainEEpara} and \Cref{LEreduction}, we have the initial estimate
\begin{equation}\label{startingest}
\|v\|_{L_T^{\infty}H_x^{\sigma}}+\|v\|_{\mathcal{X}^{\sigma}}\leq C(M,L)(\|v_0\|_{H^{\sigma}}+\|f\|_{Y^{\sigma}})+\epsilon\|v\|_{X^{\sigma}}.    
\end{equation}
We would like to strengthen this bound by replacing the left-hand side of (\ref{startingest}) with $\|v\|_{X^{\sigma}}$, which would suffice to complete the proof. For this, we require control of the slightly stronger (than the $L_T^{\infty}H_x^{\sigma}$) norm
\begin{equation*}
\|v\|_{\mathcal{Z}^{\sigma}}:=\left(\sum_{j\geq 0}2^{2j\sigma}\|S_jv\|_{L_T^{\infty}L_x^2}^2\right)^{\frac{1}{2}}.    
\end{equation*}
Since  $\|v\|_{X^\sigma}\lesssim \|v\|_{\mathcal{Z}^\sigma}+\|v\|_{\mathcal{X}^\sigma}$, (\ref{reducedbound}) will follow from (\ref{startingest}) and the following lemma, for $\epsilon$ small enough (depending on $M$ and $L$).
\begin{lemma}\label{controlofprincipal} Under the above assumptions, $v$ satisfies the following estimate in the space $\mathcal{Z}^{\sigma}$:
\begin{equation*}
\|v\|_{\mathcal{Z}^{\sigma}}\leq C(M,L)(\|v_0\|_{H^{\sigma}}+\|f\|_{Y^{\sigma}}+\|v\|_{\mathcal{X}^{\sigma}})+\epsilon\|v\|_{X^\sigma}.
\end{equation*}    
\end{lemma}
\begin{proof}
We begin by defining $v_k:=S_kv$ for each $k\geq 0$. We see that $v_k$ satisfies the equation
\begin{equation}\label{paraeqn2}
\begin{cases}
&i\partial_tv_k+\partial_jT_{g^{ij}}\partial_i v_k+T_{b^j}\partial_jv_k+T_{\tilde{b}^j}\partial_j\overline{v}_k=S_kf+\mathcal{R}_k,
\\
&v_k(0)=S_kv_0,
\end{cases}
\end{equation}
where
\begin{equation*}
\mathcal{R}_k:=[T_{g^{ij}},S_k]\partial_i\partial_j\tilde{S}_kv+[T_{\partial_jg^{ij}},S_k]\partial_i\tilde{S}_kv+[T_{b^j},S_k]\partial_j\tilde{S}_kv+[T_{\tilde{b}^j},S_k]\partial_j\tilde{S}_k\overline{v}
\end{equation*}
and $\tilde{S}_k$ is a fattened version of the dyadic multiplier $S_k$. By dyadic summation and \Cref{mainEE}, the proof of the lemma will be concluded if we can show that
\begin{equation}\label{Restimatefinal}
\|\mathcal{R}_k\|_{Y^{\sigma}}\leq C(M)\|\tilde{S}_kv\|_{\mathcal{X}^{\sigma}}+\epsilon\|\tilde{S}_kv\|_{X^{\sigma}}
\end{equation}
for some $\epsilon>0$ sufficiently small. This is an easy consequence of \Cref{CM} for the latter three terms as we can estimate these in $L_T^1H_x^{\sigma}$ and take $T$ small. To estimate the remaining term, we first observe that
\begin{equation*}
[T_{g^{ij}},S_k]\partial_i\partial_j\tilde{S}_kv=[T_{S_{<k}g^{ij}},S_k]\partial_i\partial_j\tilde{S}_kv=[S_{<k}g^{ij},S_k]\partial_i\partial_j\tilde{S}_kv+[T_{S_{<k}g^{ij}}-S_{<k}g^{ij},S_k]\partial_i\partial_j\tilde{S}_kv. 
\end{equation*}
The latter term above can be estimated easily in $L_T^1H_x^{\sigma}$ by the right-hand side of (\ref{Restimatefinal}) by using paradifferential calculus and then by taking $T$ small. For the remaining term, we use that
\begin{equation*}
[S_{<k}g^{ij},S_k]\partial_i\partial_j\tilde{S}_kv=2^{-k}L(S_{<k}\nabla_xg^{ij},\partial_i\partial_j\tilde{S}_kv) ,
\end{equation*}
where $L$ is a translation invariant operator of the form
\begin{equation*}
L(\phi_1,\phi_2)(x)=\int \phi_1(x+y)\phi_2(x+z)K(y,z)dydz,\hspace{5mm}\|K\|_{L^1}\lesssim 1.    
\end{equation*}
See, for instance, \cite{tao2001global2}. As the spaces $l^1X^{s_0}$ and $\mathcal{X}^{\sigma}$ are translation invariant (in that they admit translation invariant equivalent norms), it follows from \Cref{bilinear} that we have
\begin{equation*}
\|[S_{<k}g^{ij},S_k]\partial_i\partial_j\tilde{S}_kv\|_{Y^{\sigma}}\leq C(M)\|\tilde{S}_kv\|_{\mathcal{X}^{\sigma}}.    
\end{equation*}
This completes the proof of the lemma.
\end{proof}
\section{Well-posedness for the nonlinear flow}\label{finalsection}
Now, we proceed with the proof of \Cref{quadraticmain}. By differentiating (\ref{fulleqn}), we obtain an equation for $(u,\nabla u)$ of the form (\ref{div-form}). Therefore, it suffices to prove the second part of the theorem for (\ref{div-form}). Given the key estimate and well-posedness in \Cref{linearestimate}, the scheme for proving this follows a very similar path to  \cite[Section 7]{marzuola2021quasilinear}. We only outline the main results and procedure here for the convenience of the reader, and refer to the corresponding parts of \cite{marzuola2021quasilinear} where relevant. A fully detailed exposition of a simplified version of the scheme that we employ below can  be found in \cite{ifrim2023local}. 
\\

The starting point is to rewrite the equation
\begin{equation}\label{finalsectioneqn}
\begin{cases}
&i\partial_tu+\partial_j g^{ij}(u,\overline{u})\partial_i u=F(u,\overline{u},\nabla u,\nabla\overline{u}),
\\
&u(0,x)=u_0(x),    
\end{cases}    
\end{equation}
in the paradifferential form
\begin{equation*}\label{paradifferentialflow}
\begin{cases}
&i\partial_tu+\partial_j T_{g^{ij}}\partial_i u+T_{b^j}\partial_ju+T_{\tilde{b}^j}\partial_j\overline{u}=G(u,\overline{u},\nabla u,\nabla\overline{u}),
\\
&u(0,x)=u_0(x),    
\end{cases}    
\end{equation*}
where 
\begin{equation*}
b:=-\partial_{(\nabla u)}F,\hspace{5mm}\tilde{b}:=-\partial_{(\nabla \overline{u})}F
\end{equation*}
and
\begin{equation*}
G(u,\overline{u},\nabla u,\nabla\overline{u}):=(\partial_jT_{g^{ij}}\partial_i-\partial_jg^{ij}\partial_i)u+F(u,\overline{u},\nabla u,\nabla\overline{u})+T_{b^j}\partial_ju+T_{\tilde{b}^j}\partial_j\overline{u}.
\end{equation*}

\subsection{Existence of  \texorpdfstring{$l^1X^s$}{}  solutions to the nonlinear equation} Our first aim is to establish existence of $l^1X^s$ solutions to the equation (\ref{div-form}) for small time. This is given by the following proposition.
\begin{proposition}\label{Existence} Let $s>\frac{d}{2}+2$ and let $u_0\in l^1H^s$ with $\|u_0\|_{l^1H^s}=M$. Suppose that $g(u_0)$ is a nontrapping, non-degenerate metric with parameters $R_0$ and $L$. Then there is $T_0>0$ depending on $M$, $L(R_0)$ and $R_0$ such that for every $T\leq T_0$, there exists a solution $u\in l^1X^s$ to \eqref{div-form} such that
\begin{enumerate}
\item ($l^1X^s$ bound).
\begin{equation*}
\|u\|_{l^1X^s}\leq C(M,L)\|u_0\|_{l^1H^s}.
\end{equation*}
\item(Smallness outside $B_{R_0}$).
\begin{equation*}
\|\chi_{>R_0}u\|_{l^1X^s}\leq 2\epsilon.   
\end{equation*}
\item (Comparable nontrapping parameter).
\begin{equation*}
L(u)\leq 2L(u_0).    
\end{equation*}
\end{enumerate}
\end{proposition}
As in Section 7 of \cite{marzuola2021quasilinear}, for each $n\geq 0$ we consider the following iteration scheme for the paradifferential form of the nonlinear equation:
\begin{equation}\label{iteration}
\begin{cases}
&i\partial_tu^{n+1}+\partial_iT_{g^{ij}(u^n)}\partial_j u^{n+1}+T_{b^j(u^n)}\partial_ju^{n+1}+T_{\tilde{b}^j(u^n)}\partial_j\overline{u}^{n+1}=G(u^n),
\\
&u^{n+1}(0,x)=u_0(x),
\end{cases}
\end{equation}
with initialization $u^0=0$. Here, we are suppressing the dependence on derivatives of $u^n$ and its complex conjugate in $b^j$, $\tilde{b}^j$ and $G$. It is clear that \Cref{Existence} will follow from our next proposition, which addresses the  convergence and bounds for the iteration scheme.
\begin{proposition}\label{uniformboundsonscheme}
Let $s,M,L,R_0, T_0$  and $u_0$ be as in \Cref{Existence}. Then there exists a constant $C(M,L)$ such that for every $n\geq 0$ there exists a solution $u^{n}$ to \eqref{iteration} on $[0,T]$  such that
\begin{enumerate}
\item ($l^1X^s$ bound).
\begin{equation*}
\|u^n\|_{l^1X^s}\leq C(M,L)\|u_0\|_{l^1H^s}.
\end{equation*}
\item(Smallness outside $B_{R_0}$).
\begin{equation*}
\|\chi_{>R_0}u^n\|_{l^1X^s}\leq 2\epsilon   . 
\end{equation*}
\item (Comparable nontrapping parameter).
\begin{equation*}
L(u^n)\leq 2L(u_0).    
\end{equation*}
\end{enumerate}
Moreover, there is a function $u\in l^1X^s$ satisfying the same bounds as above such that $u^n$ converges strongly to $u$ in $l^1X^{\sigma}$ for every $0\leq\sigma<s$.
\end{proposition}
\begin{remark}
    For simplicity of presentation, we have omitted the parameter $\frac{d}{2}+2<s_0<s$ used in \cite[Section 7]{marzuola2021quasilinear} from the statements of the results in this section. This parameter still needs to be taken into account in the (omitted) proofs to ensure that the bounds for the low-frequency coefficients $g,$ $b$, and $\tilde{b}$ stay under control in each iteration.
\end{remark}
The proof of the above proposition follows from a virtually identical line of reasoning as  \cite[Sections 7.1-7.3]{marzuola2021quasilinear}. We simply use \Cref{perturbationstable} and \Cref{linearestimate} in place of the analogues in their proof. We omit the details.
\subsection{Uniqueness and the weak Lipschitz bound}
In this subsection, we establish uniqueness of solutions in the class $l^1X^s$ when $s>\frac{d}{2}+2$. In fact, our uniqueness result follows as a corollary of a weak Lipschitz type bound as noted in the following proposition.
\begin{proposition}\label{weaklip}
Let $s>\frac{d}{2}+2$ and let $u^1_0\in l^1H^s$. Assume that $g(u_0^1)$ is a non-degenerate, nontrapping metric with parameters $M,$ $R_0$ and $L$ as above. Suppose that $u_0^2\in l^1H^s$ is another initial datum satisfying
\begin{equation*}
\|u_0^2\|_{l^1H^s}\lesssim M    ,
\end{equation*}
and suppose that $u_0^2$ is close to $u_0^1$ in the $l^1L^2$ topology in the sense that
\begin{equation*}
\|u_0^1-u_0^2\|_{l^1L^2}\ll_M e^{-C(M)L(R_0)}.   
\end{equation*}
Then the following statements hold:
\begin{enumerate}
\item $g(u_0^2)$ is nontrapping with comparable parameters to $g(u_0^1)$.
\item The solutions $u^1$ and $u^2$ generated by $u_0^1$ and $u_0^2$ exist on a time interval $[0,T]$ whose length depends only on the parameters $M$, $R_0$ and $L(R_0)$.
\item For $0\leq\sigma<s_0-1$, we have the following weak Lipschitz type bound:
\begin{equation*}
\|u^1-u^2\|_{l^1X^{\sigma}}\leq C(M,L)\|u_0^1-u_0^2\|_{l^1H^{\sigma}}    .
\end{equation*}
\end{enumerate}
\end{proposition}
\begin{proof}
The proof follows an identical line of reasoning as Section 7.4 in \cite{marzuola2021quasilinear} except that we use \Cref{perturbationstable} in place of Proposition 5.2 in \cite{marzuola2021quasilinear} to prove (i).
\end{proof}
\subsection{Frequency envelope bounds and continuous dependence}
In this final subsection, our main objective is to establish continuous dependence for (\ref{finalsectioneqn}). More precisely, for $s>\frac{d}{2}+2$, we want to show that the data-to-solution map (given nontrapping data) $u_0\mapsto u$ is continuous from $l^1H^s$ to $l^1X^s$. As in \cite{marzuola2021quasilinear}, the main ingredient is the following frequency envelope bound for the solution $u\in l^1X^s$ in terms of the data.
\begin{proposition}\label{freqenvfinal} Let $u\in l^1X^s$ be a solution to \eqref{finalsectioneqn} as in \Cref{Existence} with initial data $u_0\in l^1H^s$. Let $a_k$ be an admissible frequency envelope for $u_0$ in $l^1H^s$. Then the solution $u$ satisfies the bound
\begin{equation*}
\|S_ku\|_{l^1X^s}\leq a_kC(M,L)\|u_0\|_{l^1H^s}.    
\end{equation*}
\end{proposition}
\begin{proof}
The proof follows identical reasoning as the proof of Proposition 7.5 in \cite{marzuola2021quasilinear}. The only difference is that we use \Cref{freqenvbound} in place of the analogous bound in their proof.
\end{proof}
Armed with \Cref{freqenvfinal}, the proof of the continuity of the data-to-solution map in Section 7.6 of \cite{marzuola2021quasilinear} now applies verbatim to establish the same property in our setting. 
 \bibliographystyle{plain}
\bibliography{refs.bib}
\end{document}